\documentclass[preprint]{elsarticle}

\usepackage{amsmath,amsfonts,amssymb,amsthm}
\usepackage{bm}
\usepackage[section]{algorithm}
\usepackage{algorithmic}
\usepackage{color}
\usepackage{colortbl}
\usepackage[T1]{fontenc}
\usepackage{latexsym}
\usepackage{setspace}
\usepackage{verbatim}
\usepackage[backref,pdfpagemode=UseOutlines,colorlinks=true,
breaklinks=true,hyperindex,linkcolor=blue,
anchorcolor=black,citecolor=blue,filecolor=magenta,
menucolor=red,urlcolor=blue,bookmarks=true,
bookmarksopen=true,pdfpagelayout=SinglePage,
pdfpagetransition=Dissolve]{hyperref}
\usepackage{lineno}

\modulolinenumbers[2]

\newcommand{\Z}{\mathbb{Z}}
\newcommand{\R}{\mathbb{R}}
\newcommand{\E}{\mathbb{E}}

\newcommand{\T}{\mathcal{T}}

\newcommand{\s}{\mathcal{S}}

\newcommand{\lue}{\textit{lue}}
\newcommand{\lte}{\textit{lte}}
\newcommand{\pd}{\Lambda}
\newcommand{\po}{\Pi}

\newtheorem{thm}{Theorem}
\newtheorem{lem}[thm]{Lemma}
\newtheorem{pro}[thm]{Proposition}
\newtheorem{cor}[thm]{Corollary}
\newtheorem{dfi}{Definition}
\newtheorem{rmk}{Remark}
\newtheorem{cla}[thm]{Claim}

\begin{document}

\begin{frontmatter}

  \title{A fast algorithm  for computing irreducible triangulations of
    closed  surfaces in  $\E^d$}

\author[sr]{Suneeta Ramaswami\fnref{fn1}}
\ead{rsuneeta@camden.rutgers.edu}

\author[ms]{Marcelo Siqueira\corref{cor}\fnref{fn2}}
\ead{mfsiqueira@mat.ufrn.br}

\cortext[cor]{Corresponding author}
\fntext[fn1]{Partially supported by NSF  grant CCF-0830589}
\fntext[fn2]{Partially supported by CNPq Grant 305845/2012-8 and CNPq Grant 486951/2012-0}

\address[sr]{Rutgers University, Department of Computer Science, 227 Penn Street, BSB Room 321, Camden, NJ 08102-1656, USA}
\address[ms]{Universidade Federal do Rio Grande do Norte, Departamento
  de Matem\'atica, CCET Sala 70, Natal, RN 59078-970, Brazil}

\begin{abstract}
  We give a fast  algorithm for computing an irreducible triangulation
  $\T^{\prime}$ of  an oriented, connected,  boundaryless, and compact
  surface $\s$  in $\E^d$ from  any given triangulation $\T$  of $\s$.
  If  the genus  $g$ of  $\s$ is  positive, then  our  algorithm takes
  $\mathcal{O}( g^2 +  g n )$ time to  obtain $\T^{\prime}$, where $n$
  is  the number of  triangles of  $\T$.  Otherwise,  $\T^{\prime}$ is
  obtained in  linear time  in $n$.  While  the latter upper  bound is
  optimal,  the former upper  bound improves  upon the  currently best
  known  upper bound by  a $\lg  n /  g$ factor.   In both  cases, the
  memory space required by our algorithm is in $\Theta(n)$.
\end{abstract}


\end{frontmatter}

\linenumbers

\section{Introduction}\label{sec:intro}
Let $\s$ be a compact surface with empty boundary.  A triangulation of
$\s$ can be viewed as a  ``polyhedron'' on $\s$ such that each face is
a triangle  with three distinct  vertices and the intersection  of any
two distinct triangles  is either empty, a single  vertex, or a single
edge (including its two vertices).   A classical result from the 1920s
by Tibor Rad\'o asserts that every compact surface with empty boundary
(i.e.,   usually   called   a   \textit{closed  surface})   admits   a
triangulation~\cite{GX13}.   Let $e$  be any  edge of  a triangulation
$\T$ of  $\s$.  The  \textit{contraction} of $e$  in $\T$  consists of
contracting  $e$ to a  single vertex  and collapsing  each of  the two
triangles     meeting    $e$     into    a     single     edge    (see
Figure~\ref{fig:contraction}).  If  the result of  contracting $e$ in
$\T$ is  still a triangulation  of $\s$, then  $e$ is said to  be {\em
  contractible}; else  it is {\em  non-contractible}.  A triangulation
$\T$ of $\s$ is said to be {\em irreducible} if and only if every edge
of $\T$ is  non-contractible.  Barnette and Edelson~\cite{BE89} showed
that  all closed  surfaces  have finitely  many irreducible  surfaces.
More recently, Boulch, de Verdi\`ere, and Nakamoto~\cite{BVN13} showed
the same result for compact surfaces with a nonempty boundary.

Irreducible  triangulations have proved  to be  an important  tool for
tackling   problems  in  combinatorial   topology  and   discrete  and
computational  geometry.   The   reasons  are  two-fold.   First,  all
irreducible  triangulations  of  any  given  compact  surface  form  a
``basis'' for  all triangulations of the same  surface.  Indeed, every
triangulation of the surface can be  obtained from at least one of its
irreducible   triangulations   by    a   sequence   of   {\em   vertex
  splittings}~\cite{SCH91,SUL06a},   where    the   vertex   splitting
operation  is  the inverse  of  the  edge  contraction operation  (see
Figure~\ref{fig:contraction}).      Second,    some     problems    on
triangulations can be solved by considering irreducible triangulations
only.  In  particular, irreducible  triangulations have been  used for
proving the  existence of geometric  realizations (in some  $\E^d$) of
triangulations   of   certain    surfaces,   where   $\E^d$   is   the
$d$-dimensional  Euclidean   space~\cite{ABE-M07,BN08},  for  studying
properties       of        diagonal       flips       on       surface
triangulations~\cite{NEG94,NEG99,CGMN02,KNS09,HNOS11},              for
characterizing  the  structure   of  flexible  triangulations  of  the
projective plane~\cite{CL98},  and for finding lower  and upper bounds
for the maximum number of cliques in an $n$-vertex graph embeddable in
a given  surface~\cite{DFJSW11}.  Irreducible triangulations  are also
``small'', as their number of vertices  is at most linear in the genus
of the surface~\cite{NO95,BVN13}.  However,  the number of vertices of
all irreducible triangulations of the same surface may vary, while any
irreducible triangulation of smallest size (known as \textit{minimal})
has $\Theta( \sqrt{g}  )$ vertices if the genus $g$  of the surface is
positive~\cite{JR80}.

The  sphere  has a  unique  irreducible  triangulation,  which is  the
boundary  of  a tetrahedron~\cite{SR34}.   The  torus  has exactly  21
irreducible triangulations,  whose number of  vertices varies from  7 to
10~\cite{LAV90}.   The  projective  plane  has  only  two  irreducible
triangulations,   one  with   6  vertices   and  the   other   with  7
vertices~\cite{BAR82}.   The Klein bottle  has exactly  29 irreducible
triangulations   with   number  of   vertices   ranging   from  8   to
11~\cite{SUL06b}.   Sulanke devised and  implemented an  algorithm for
generating  all irreducible  triangulations of  compact  surfaces with
empty   boundary~\cite{SUL06a}.    Using   this   algorithm,   Sulanke
rediscovered   the  aforementioned   irreducible   triangulations  and
generated  the  complete sets  of  irreducible  triangulations of  the
double  torus,  the triple  cross  surface,  and  the quadruple  cross
surface.

The  idea  behind  Sulanke's  algorithm  is  to  generate  irreducible
triangulations   of   a   surface   by   modifying   the   irreducible
triangulations of  other surfaces of smaller Euler  genuses (the Euler
genus of a surface equals  the usual genus for nonorientable surfaces,
and  equals  twice the  usual  genus  for  orientable surfaces).   The
modifications include  vertex splittings and the  addition of handles,
crosscaps, and crosshandles. Unfortunately,  the lack of a known upper
bound on the number of  vertex splittings required in the intermediate
stages  of  the  algorithm   prevented  Sulanke  from  establishing  a
termination criterion for all surfaces.  Furthermore, his algorithm is
impractical   for  surfaces  with   Euler  genus   $\ge  5$,   as  his
implementation could  take centuries  to generate the  quintuple cross
surface  on  a cluster  of  computers with  an  average  CPU speed  of
2GHz~\cite{SUL06a}.   To  the  best   of  our  knowledge,  no  similar
algorithm for compact surfaces with a nonempty boundary exists.

\subsection{Our contribution}
Here, we  give an algorithm for  a problem closely related  to the one
described  above: \textit{given  any triangulation  $\T$ of  a compact
  surface  $\s$  with empty  boundary,  find \textit{one}  irreducible
  triangulation, $\T^{\prime}$, of $\s$  from $\T$}. In particular, if
the genus $g$ of $\s$ is positive, then we show that $\T^{\prime}$ can
be  computed in $\mathcal{O}(  g n  + g^2  )$ time,  where $n$  is the
number of triangles of  $\T$. Otherwise, $\T^{\prime}$ can be computed
in $\mathcal{O}(  n )$  time, which is  optimal.  In either  case, the
space requirement  is in $\Theta(n)$.   To the best of  our knowledge,
the previously best known (time) upper bound is $\mathcal{O}( n\lg n +
g \lg n + g^4 )$  for the algorithm given by Schipper in~\cite{SCH91}.
In his  complexity analysis, Schipper  assumed that $g$ is  a constant
depending  only  on   $\s$,  and  thus  stated  the   upper  bound  as
$\mathcal{O}( n\lg  n )$. While  it is true  that $g$ is  an intrinsic
feature of $\s$,  we may have $m \in  \Theta( \sqrt{g} )$~\cite{JR80},
where $m$ is the number of vertices of $\T$, which implies that $n \in
\Theta( g  )$ (see Section~\ref{sec:back}).   Thus, we state  the time
bounds in  terms of both  $g$ and $n$.   Since our algorithm  can more
efficiently generate  \textit{one} irreducible triangulation  from any
given  triangulation  of  $\s$,  it  can  potentially  be  used  as  a
``black-box'' by a fast and  alternative method (to that of Sulanke's)
for  generating \textit{all} irreducible  triangulations of  any given
surface.

\subsection{Application to the triquad conversion problem}
The algorithm for  computing irreducible triangulations described here
was recently  incorporated into  an innovative and  efficient solution
\cite{TRS}  to the  problem of  converting a  triangulation $\T$  of a
closed surface into a quadrangulation with the same set of vertices as
$\T$      (known      as      the     \textit{triquad      conversion}
problem~\cite{BLPPSTZ13}).  This new  solution takes $\mathcal{O}( g n
+  g^2 )$  time, where  $n$ is  the number  of triangles  of  $\T$, to
produce  the  quadrangulation if  the  genus  $g$  of the  surface  is
positive.   Otherwise, it  takes  linear time  in  $n$.  The  solution
improves upon the approach of computing a perfect matching on the dual
graph of $\T$, for which the best known upper bound is $\mathcal{O}( n
\lg^2  n  )$   amortized  time~\cite{DS10}.   In~\cite{TRS},  the  new
solution   is   experimentally  compared   with   two  simple   greedy
algorithms~\cite{PS96,VEL00} and  the approach based  on the algorithm
in~\cite{DS10}.        It       outperforms       the       approaches
in~\cite{PS96,VEL00,DS10}  whenever $n$ is  sufficiently large  and $g
\ll  n$,  which is  typically  the  case  for triangulations  used  in
computer  graphics and  engineering  applications.  We  hope that  the
solution in~\cite{TRS} to the triquad conversion problem increases the
practical   interest    for   algorithms   to    compute   irreducible
triangulations of surfaces.

\subsection{Organization}
The   remainder    of   this   paper   is    organized   as   follows:
Section~\ref{sec:back} introduces the  notation, terminology and basic
definitions used throughout the paper. Section~\ref{sec:rwork} reviews
prior work on algorithms  for computing irreducible triangulations and
related  algorithms   (e.g.,  algorithms  for   mesh  simplification).
Section~\ref{sec:algo} describes our proposed algorithm in detail, and
analyzes its time and space complexities. Section~\ref{sec:expresults}
presents  an experimental  comparison of  the implementation  of three
algorithms   for  computing   irreducible   triangulations:  ours;   a
randomized, brute-force  algorithm; and  the one proposed  by Schipper
in~\cite{SCH91}.  Finally,  section~\ref{sec:conc} summarizes our main
contributions, and discusses future research directions.

\section{Notation, terminology, and background}\label{sec:back}
Let $\E^d$  denote the  $d$-dimensional Euclidean (affine)  space over
$\R$, and let  $\R^d$ denote the associated vector  space of $\E^d$. A
subset  of $\E^d$  that is  homeomorphic  to the  open unit  interval,
$\mathbb{B}^1 = ( 0 , 1 ) \subset \E$, is called an {\em open arc}.  A
subset   of  $\E^d$  that   is  homeomorphic   to  the   open  circle,
$\mathbb{B}^2 = \{ ( x , y ) \in  \E^2 \mid x^2 + y^2 < 1 \}$, of unit
radius is called an {\em open  disk}. Recall that a subset $\s \subset
\E^d$  is called  a {\em  topological surface},  or {\em  surface} for
short, if each  point $p$ in $\s$ has an open  neighborhood that is an
open disk.   According to this  definition, a surface is  a ``closed''
object in the  sense that it has an empty  boundary. Here, we restrict
our attention to the class consisting of all {\em oriented, connected,
  and compact surfaces in $\E^d$},  and we use the term ``surface'' to
designate a member of this class (unless stated otherwise).

The notions of  triangle mesh and quadrilateral mesh  of a surface are
synonyms for the well-known terms triangulation and quadrangulation of
a  surface, respectively,  in topological  graph theory  and algebraic
topology~\cite{GT01,GX13}.    Informally,   a   triangulation   (resp.
quadrangulation) of a surface is a  way of cutting up the surface into
triangular (resp.  quadrilateral) regions  such that these regions are
images  of triangles  (resp.  quadrilaterals)  in the  plane,  and the
vertices and  edges of these planar  triangles (resp.  quadrilaterals)
form a  graph with  certain properties. To  formalize these  ideas, we
rely on the  notions of subdivision of a surface,  as nicely stated by
Guibas and Stolfi~\cite{GS85}, and of a graph embedded on a surface.

\begin{dfi}\label{def:subdivision}
  A {\em subdivision} of a surface $\s$ is a partition, $\mathcal{P}$,
  of $\s$ into three finite  collections of disjoint subsets: the {\em
    vertices},  {\em edges},  and {\em  faces}, which  are  denoted by
  $V_{\mathcal{P}}(\s)$,           $E_{\mathcal{P}}(\s)$,          and
  $F_{\mathcal{P}}(\s)$,  respectively,   and  satisfy  the  following
  conditions:
\begin{itemize}
\item[(S1)] every vertex is a point,

\item[(S2)] every edge is an open arc,

\item[(S3)] every face is an open disk, and

\item[(S4)] the boundary of every face is a closed path of edges and vertices.
\end{itemize}
\end{dfi}

Condition (S4) is based on the notion of ``closed path'' on a surface,
which can be formalized  as follows: let $\mathbb{S}^1 = \{ (  x , y )
\in \E^2 \mid  x^2 + y^2 = 1  \}$ be the circumference of  a circle of
unit radius centered at the origin.   We define a {\em simple path} in
$\mathbb{S}^1$ as a partition of $\mathbb{S}^1$ into a finite sequence
of isolated points  and open arcs. Then, condition  (S4) is equivalent
to  the following (refer  to Figure~\ref{fig:conditions4}):  for every
face $\tau$  in $\mathcal{P}$, there  exists a simple path,  $\pi$, in
$\mathbb{S}^1$    and    a    continuous    mapping,    $g_{\tau}    :
\mathbb{\overline{B}}^2     \rightarrow     \overline{\tau}$,    where
$\mathbb{\overline{B}}^2$  and $\overline{\tau}$  are the  closures of
$\mathbb{B}^2$   and   $\tau$,   such   that   $g_{\tau}$   (i)   maps
$\mathbb{B}^2$ homeomorphically  onto $\tau$, (ii) maps  each open arc
of  $\pi$ homeomorphically onto  an edge  of $\mathcal{P}$,  and (iii)
maps each isolated  point of $\pi$ to a  vertex of $\mathcal{P}$.  So,
condition  (S4) implies  that the  images of  the isolated  points and
edges  of $\pi$ under  $g_{\tau}$, taken  in the  order in  which they
occur around  $\mathbb{S}^1$, constitute  a closed, connected  path of
vertices and  edges of $\mathcal{P}$,  whose union is the  boundary of
$\tau$.  Note  that this  path need not  be simple, as  $g_{\tau}$ may
take two  or more distinct  points or open  arcs of $\pi$ to  the same
vertex or edge of $\mathcal{P}$, respectively.

\begin{figure}[htb!]
\begin{center}
\includegraphics[height=1.4in]{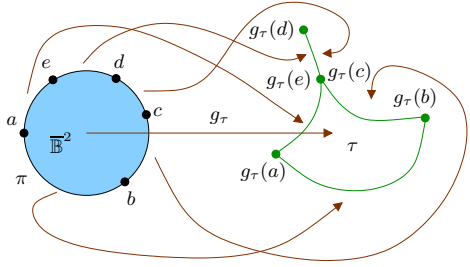}
\caption{\label{fig:conditions4} Illustration of condition (S4) of Definition~\ref{def:subdivision}.}
\end{center}
\end{figure}

Since an open disk cannot be entirely covered by a finite number of
vertices and edges, every vertex and edge in $\mathcal{P}$ must be
incident on some face of $\mathcal{P}$. In fact, using condition (S4),
it is possible to show that (i) every edge of $\mathcal{P}$ is
entirely contained in the boundary of some face of $\mathcal{P}$, (ii)
every vertex is incident on an edge, and (iii) every edge of
$\mathcal{P}$ is incident on two (not necessarily distinct) vertices
of $\mathcal{P}$. These vertices are called the {\em endpoints} of the
edge. If they are the same, then the edge is a {\em loop}, and its
closure is homeomorphic to the circumference of a circle of unit
radius, $\mathcal{S}^1 = \{ ( x , y ) \in \E^2 \mid x^2 + y^2 = 1
\}$. 

We can  define an  undirected graph, $G$,  and a  one-to-one function,
$\imath :  G \rightarrow \s$ from  the collection of  all vertices and
edges of $\mathcal{P}$. Let $G = ( V  , E )$ be a graph such that $V =
\{ v_1 , \ldots, v_h \}$ and $E  = \{ e_1 , \ldots, e_l \}$, where $h$
and $l$ are the number  of vertices and edges of $\mathcal{P}$. Define
a one-to-one mapping, $\imath : G \rightarrow \s$, such that each $v_j
\in V$ and each $e_k \in E$ is associated with a distinct vertex and a
distinct edge of $V_{\mathcal{P}}(  \s )$ and $E_{\mathcal{P}}( \s )$,
respectively,  for   $j  =  1,  \ldots,   h$  and  $k   =  1,  \ldots,
l$. Furthermore, if  $v$ and $u$ are the two  vertices in $V$ incident
on edge $e$ in $E$, then $i( v )$ and $i( u )$ are the two vertices of
$V_{\mathcal{P}}( \s )$ incident on $\imath( e )$ in $E_{\mathcal{P}}(
\s  )$. Note  that  $i(G)^{c} =  \s -  \imath(  G)$ are  the faces  of
$\mathcal{P}$.  We say  that $G$ is the {\em  graph of $\mathcal{P}$},
and that $i$ is the {\em embedding}  of $G$ on $\s$.  Graph $G$ can be
viewed as the combinatorial structure of $\mathcal{P}$, while $\imath$
can be viewed as a ``geometric realization'' of $G$ on the surface.

Two subdivisions of the same  surface are {\em isomorphic} if and only
if their graphs are isomorphic. Since $\mathcal{P}$ is fully described
by $G$ and $\imath$,  it is customary to call the pair,  $( G , \imath
)$,  the subdivision itself.   Triangulations of  a given  surface are
specialized subdivisions that  adequately capture the practical notion
of triangle meshes:

\begin{dfi}\label{def:triangulation}
  A {\em  triangulation} of  a surface  $\s$ is a  subdivision $(  G ,
  \imath )$  such that each  face of $\imath(  G )^{c}$ is  bounded by
  exactly three  distinct vertices (resp.  edges) from  $\imath( G )$.
  Furthermore,  any two  edges of  a  triangulation have  at most  one
  common  endpoint,  and  every  vertex  of a  triangulation  must  be
  incident on at least three edges.
\end{dfi}

The  following  lemmas  state  two  important  properties  of  surface
triangulations:

\begin{lem}\label{lem:tritwofaces}
  Every edge  of a  surface triangulation is  incident on  exactly two
  faces.
\end{lem}
\begin{proof}
See~\ref{sec:proofs}.
\end{proof}

\begin{lem}[\cite{GS85}]\label{lem:graphofsub}
  If  $G$  is  the graph  of  a  surface  triangulation, then  $G$  is
  connected.
\end{lem}

Recall that the genus, $g$, of a surface $\s$ is the maximum number of
disjoint, closed, and simple  curves, $\alpha_1, \ldots, \alpha_g$, on
$\s$ such that the set $\s  - ( \alpha_1 \cup \cdots \cup \alpha_g)$ is
connected.  Up to  homeomorphisms, there is only one  surface of genus
$g$~\cite{GX13}. For  any subdivision, $( G  , \imath )$,  of $\s$, we
have $n_v -  n_e + n_f = 2 \cdot  ( 1 - g )$,  where $n_v$, $n_e$, and
$n_f$ are  the number of vertices, edges,  and faces of $(  G , \imath
)$~\cite{GT01}.  If $( G , \imath  )$ is a triangulation of $\s$, then
$3 \cdot n_f = 2 \cdot n_e$, which implies that $2 \cdot n_v - n_f = 4
( 1 - g  )$ and $3 \cdot n_v - n_e  = 6 ( 1 - g )$,  and thus $n_e \in
\Theta( n_v + g)$ and $n_f \in \Theta( n_v + g)$. Here, we assume that
$\s$ is such that $g \in  \mathcal{O}( n_v )$, which implies that $n_e
\in \Theta( n_v  )$ and $n_f \in  \Theta( n_v )$.  As we  see in
Section~\ref{sec:algo},  our algorithm  requires only  the graph  of a
triangulation  as its  input.  Hence,  we simply refer  to  a given
triangulation by $\T$.

Let $\T$ be any triangulation of a surface $\s$. Then, every vertex of
$\T$ is incident on at least  three edges. Since every edge of $\T$ is
incident     on     exactly     two     faces     of     $\T$     (see
Lemma~\ref{lem:tritwofaces}), every vertex of $\T$ must be incident on
at least three  faces of $\T$ as well.   Furthermore, for every vertex
$v$ of $\T$, the edges $e$ and faces $\tau$ of $\T$ containing $v$ can
be arranged as cyclic  sequence $e_1, \tau_1, e_2, \ldots, \tau_{k-1},
e_k,  \tau_k$,  in  the  sense  that  $e_j$  is  the  common  edge  of
$\tau_{j-1}$ and  $\tau_j$, for  all $j$,  with $2 \le  j \le  k$, and
$e_1$  is  the common  edge  of $\tau_1$  and  $\tau_k$,  with $k  \ge
3$~\cite{GX13}. The set
\[
v \cup e_1 \cup \tau_1  \cup  e_2 \cup  \cdots  \cup \tau_{k-1} \cup e_k \cup \tau_k \subset \s \, ,
\]
is  called  the  {\em  star  of   $v$  in  $\T$}  and  is  denoted  by
$\textit{st}( v , \T  )$ (see Figure~\ref{fig:starlink}). It turns out
that  $\textit{st}(  v  , \T  )$  is  homeomorphic  to an  open  disk.
Furthermore, the boundary of $\textit{st}(  v , \T )$ in $\s$ consists
of  the  boundary edges  of  $\tau_1,  \ldots,  \tau_k$ that  are  not
incident on $v$, as well as the endpoints of $\tau_1, \ldots, \tau_k$,
except for $v$  itself.  This point set, denoted  by $\textit{lk}( v ,
\T )$, is a  simple, closed curve on $\s$ called the  {\em link of $v$
  in  $\T$} (see  Figure~\ref{fig:starlink}).  If  $e$ is  an  edge of
$\T$, then the  {\em star, $\textit{st}( e , \T)$, of  $e$ in $\T$} is
the set $e  \cup \tau \cup \sigma$, where $\tau$  and $\sigma$ are the
two  faces  of  $\T$  incident  on  $e$.   In  turn,  the  {\em  link,
  $\textit{lk}( e , \T)$, of $e$ in $\T$} is the set consisting of the
two vertices, $x$ and $y$, such that $x$ is incident on $\tau$ and $y$
is incident  on $\sigma$, but none of  $x$ and $y$ is  incident on $e$
(see Figure~\ref{fig:starlink}).

Let $\tau$ and $e$ be a face and an edge of $\T$, respectively.  Since
every  vertex of  $\T$ is  incident  on at  least three  triangulation
edges,  i.e., since  each vertex  of $\T$  has {\em  degree}  at least
three, if $u$, $v$, and $w$  are the boundary vertices of $\tau$, then
we  can uniquely  identify $\tau$  by enumerating  these  vertices. In
particular, we denote $\tau$  by $[ u , v , w  ]$. Similarly, since no
two edges of a triangulation share  the same two endpoints, if $u$ and
$v$ are the  two endpoints of edge $e$, then  we can uniquely identify
$e$ by enumerating its two  endpoints, and therefore we can denote $e$
by $[ u , v ]$.

\begin{figure}[htbp]
\begin{center}
\includegraphics[height=1.5in]{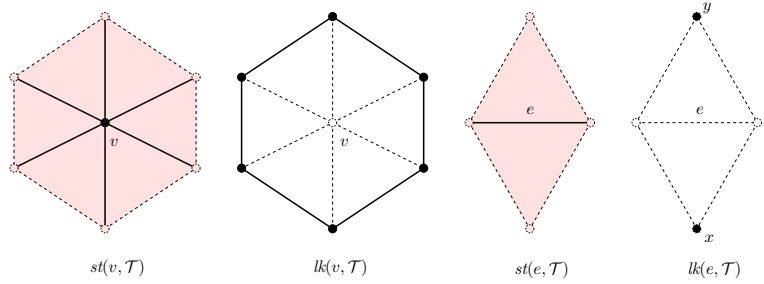}
\caption{\label{fig:starlink}  The star  and  the link  of vertex  $v$
  (left) and the star and the link of edge $e$ (right).}
\end{center}
\end{figure}

\begin{dfi}\label{def:contraction}
  Let $\T$ be a triangulation of a surface, $\s$, and let $e = [ u , v
  ]$ be  an edge of  $\T$.  The {\em  contraction} of $e$  consists of
  merging  $u$  and $v$  into  a  new vertex  $w$,  such  that $w  \in
  \textit{st}(  u ,  \T )  \cup  \textit{st}( v  , \T  )$, edges  $e$,
  $[v,x]$ and $[v,y]$, and  faces $[u,v,x]$ and $[u,v,y]$ are removed,
  edges of the form $[ u , p ]$  and $[ v , q ]$ are replaced by $[w ,
  p]$ and $[ w , q ]$, and faces of  the form $[ u , r , s ]$ and $[ v
  , t ,  z ]$ are replaced  by $[ w ,  r , s ]$ and  $[ w , t  , z ]$,
  where $x$ and $y$ are the vertices in the link, $\textit{lk}( e , \T
  )$, of $e$  in $\T$, $p, q \not\in  \{ x , y \}$, $r,s  \neq v$, and
  $t, z  \neq u$.  If the result  is a triangulation of  $\s$, then we
  denote   it   by   $\T-uv$,    and   call   the   contraction   {\em
    topology-preserving} and $e$ a {\em contractible} edge.
\end{dfi}

Figure~\ref{fig:contraction}   illustrates    the   edge   contraction
operation. 

\begin{figure}[htbp]
\begin{center}
\includegraphics[height=1.2in]{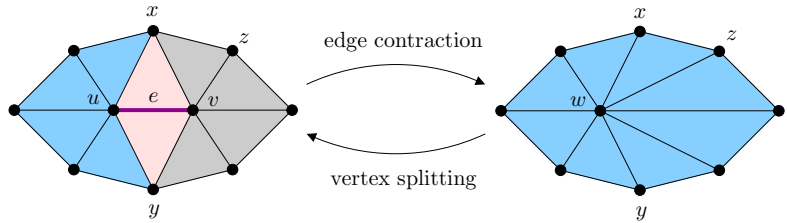}
\caption{\label{fig:contraction} Contraction  of edge $[u,v]$  and its
  inverse: splitting of vertex $w$.}
\end{center}
\end{figure}

Dey, Edelsbrunner,  Guha, and Nekhayev~\cite{DEGN99}  gave a necessary
and  sufficient  condition,  called   the  {\em  link  condition},  to
determine whether an edge  of a surface triangulation is contractible.
Let $e  = [ u  , v ]$  be any edge  of a surface  triangulation, $\T$.
Then, edge $e$ is contractible  if and only if the following condition
holds:
\begin{equation}\label{eq:linkcondition}
\textit{lk}( e , \T  ) = \textit{lk}( u , \T )  \cap \textit{lk}( v , \T) \, .
\end{equation}
In other words,  edge $e$ is contractible if and only  if the links of
$u$  and $v$  in $\T$  have  no common  vertices, except  for the  two
vertices of  the link of  $e$ in $\T$.   The link condition  is purely
combinatorial. In fact,  we can easily test edge  $e$ against the link
condition by considering  the graph, $G_{\T}$, of $\T$  only. In fact,
an equivalent  characterization of the link condition,  which uses the
notion  of critical  cycle on  the graph  of a  triangulation (defined
below), had been given before by Barnette in~\cite{BAR82}.

If edge $e$ passes the test, then the graph of $\T - uv$ can be easily
obtained from $G_{\T}$ by merging  its vertices $\imath^{-1}( u )$ and
$\imath^{-1}(  v  )$  into  a   new  vertex  $w$,  by  removing  edges
$\imath^{-1}( e )$, $\imath^{-1}( [ v  , x] )$ and $\imath^{-1}( [ v ,
y ] )$, and by replacing every  edge of the form $\imath^{-1}( [ u , r
] )$ and $\imath^{-1}( [ v , s  ] )$ with with $( w , \imath^{-1}( r )
)$  and $(  w ,  \imath^{-1}( s  ) )$,  respectively, where  $\imath :
G_{\T} \rightarrow \s$  is the embedding of $G_{\T}$  in $\s$, $x$ and
$y$ are the vertices  in the link, $\textit{lk}( e , \T  )$, of $e$ in
$\T$, and $r, s \not\in \{ x , y \}$.  We can prove that the resulting
graph is  embeddable in $\s$, and thus  the fact that we  can define a
triangulation  (i.e., $\T  - uv$)  from the  resulting graph  does not
depend on the surface geometry \cite{GX13}.

A {\em $\ell$-cycle}  in a triangulation $\T$ consists  of a sequence,
$e_1,\ldots,e_\ell$, of $\ell$ edges of $\T$ such that $e_j$ and $e_k$
share an endpoint in $\T$ if and only if $| j - k | = 1$ or $| j - k |
= \ell - 1$,  for all $j, k = 1,\ldots,\ell$, with  $j \neq k$.  Since
the  two endpoints of  a triangulation  edge cannot  be the  same, and
since  no two  edges  of a  triangulation  can have  two endpoints  in
common, a triangulation can only have $\ell$-cycles, for $\ell \ge 3$.
Furthermore,   each  cycle   can  be   unambiguously   represented  by
enumerating  the   vertices  of  its  edges  rather   than  the  edges
themselves.   In   particular,   if   $e_1,\ldots,e_\ell$   define   a
$\ell$-cycle   in   $\T$,   then   we   denote  this   cycle   by   $(
v_1,\ldots,v_\ell )$, where $v_j$ is  the common vertex of edges $e_j$
and $e_{j+1}$, for  each $j = 1, \ldots, \ell-1$,  and $v_\ell$ is the
common vertex of edges $e_1$  and $e_\ell$.  A $\ell$-cycle of $\T$ is
said to be  {\em critical} if and  only if $\ell = 3$  and its (three)
edges do  not belong to the  boundary of the  same triangulation face.
For instance,  cycle $( u ,  v , z  )$ is critical in  both (partially
shown) triangulations in  Figure~\ref{fig:critical}. Observe that edge
$[u,v]$ fails  the link condition  (see Eq.~\ref{eq:linkcondition}) in
both triangulations, and hence it is non-contractible in both.

Every  genus-$0$ surface (i.e.,  a surface  homeomorphic to  a sphere)
admits a triangulation with four  vertices, six edges, and four faces.
We denote  this triangulation by $\T_4$.   Figure~\ref{fig:t4} shows a
planar drawing of  the graph of $\T_4$.  Note  that every $3$-cycle of
$\T_4$ consists of (three) edges that  bound a face of $\T_4$.  So, no
$3$-cycle of $\T_4$ is critical.   Note also that every edge of $\T_4$
fails the link condition, and thus is a non-contractible edge. So,
$\T_4$ is a ``minimal'' triangulation in the sense that no edge of
$\T_4$ is contractible. In fact, $\T_4$ is the only triangulation of a
genus-$0$ surface satisfying this property. For a surface of arbitrary
genus, we have:

\begin{thm}[Lemma 3 in \cite{SCH91}]\label{the:linkcond}
  Let $\s$  be a surface, and  let $\T$ be any  triangulation of $\s$.
  Then, an edge $e$ of $\T$ is  a contractible edge if and only if $e$
  does  not belong  to any  critical  cycle of  $\T$ and  $\T$ is  not
  (isomorphic to) the triangulation $\T_4$.
\end{thm}

\begin{figure}[htb!]
\begin{center}
\includegraphics[height=1.5in]{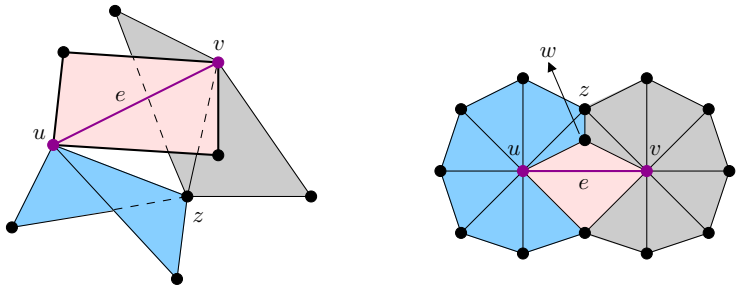}
\caption{\label{fig:critical} Edges $[u,z]$, $[z,v]$,  and $[ v , u ]$
  define critical cycles in both triangulations.}
\end{center}
\end{figure}
\begin{figure}[htb!]
\begin{center}
\includegraphics[height=1.4in]{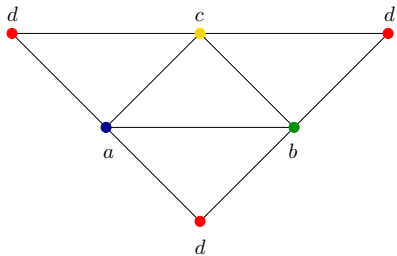}
\caption{\label{fig:t4} A planar drawing of the graph of
  $\T_4$. Vertices labeled $d$ are the same vertex.}
\end{center}
\end{figure}

\begin{dfi}\label{def:trapped}
  Let $\T$  be a  triangulation of a  surface $\s$,  and let $v$  be a
  vertex of $\T$.  If all  edges incident on $v$ are non-contractible,
  then $v$ is called {\em trapped}; else it is called {\em loose}.
\end{dfi}

\begin{dfi}\label{def:irreducible}
  Let $\s$ be  a surface, and let $\T$ be a  triangulation of $\s$. If
  every edge of $\T$ is  non-contractible, then $\T$ is called {\em
    irreducible};  else it is called {\em reducible}.
\end{dfi}

A triangulation is irreducible if and  only if all of its vertices are
trapped. To the  best of our knowledge, the best  known upper bound on
the  size  of  irreducible  triangulations  was  given  by  Jaret  and
Wood~\cite{JW10}:

\begin{thm}[\cite{JW10}]\label{the:itsize}
  Let $\s$ be a compact surface with empty boundary whose Euler genus,
  $h$, is positive,  and let $\T$ be any  irreducible triangulation of
  $\s$. Then, the number of vertices, $n_v$, of $\T$ is such that $n_v
  \le 13 \cdot h  - 4$. If $\s$ is also orientable,  then we have that
  $g =  2h$, where $g$  is the  genus of $\s$,  and hence $n_v  \le 26
  \cdot g - 4$.
\end{thm}

The largest  known irreducible triangulation of  an orientable surface
of   genus  $g$   has  $\lfloor   \frac{17}{2}  g   \rfloor$  vertices
(see~\cite{SUL06c}).

\section{Related work}\label{sec:rwork}
An irreducible triangulation, $\T^{\prime}$,  of a surface $\s$ can be
obtained   by  applying   a  sequence   of   topology-preserving  edge
contractions to a given triangulation,  $\T$, of $\s$. Such a sequence
can  be  found  by  repeatedly  searching  for  a  contractible  edge.
Whenever a contractible edge is found, it is contracted and the search
continues.   If  no  contractible  edge  is found,  then  the  current
triangulation is already an irreducible  one, and the search ends. The
\textit{link condition test} (defined by Eq.~\ref{eq:linkcondition} in
Section~\ref{sec:back}) can be used to decide whether an examined edge
is contractible.  While this approach  is quite simple, it can be very
time-consuming in the worst-case.

Indeed, if  an algorithm to  compute $\T^{\prime}$ relies on  the link
condition test to compute  an irreducible triangulation, then its time
complexity  is basically  determined  by two  factors:  (1) the  total
number  of  times  the link  condition  test  is  carried out  by  the
algorithm, and (2) the time spent with each test.  Bounding the number
of link condition tests is challenging because \textit{the contraction
  of an edge can  make a previously non-contractible edge contractible
  and vice-versa}.  Moreover, if no special data  structure is adopted
by the  algorithm, then  the time  to test  an edge $e  = [  u ,  v ]$
against the  link condition is  in $\Theta( d_u  \cdot d_v )$,  in the
worst case, where $d_u$ and $d_v$  are the degrees of vertices $u$ and
$v$ in the current triangulation.

Consider  the triangulation of  a sphere  in Figure~\ref{fig:badcase},
which is cut open in two separate pieces.  There are exactly $n_v = 3m
+   2$   vertices   in   this   triangulation,   namely:   $x$,   $y$,
$v_0,\ldots,v_{m-1}$,  $w_0,\ldots,w_{m-1}$, and $u_0,\ldots,u_{m-1}$.
For each $i \in \{0,\ldots,m-1\}$, edges $[ v_i , v_{(i+1) \mod m} ]$,
$[  v_i,x]$,  or  $[  v_i,y]$  are  all  non-contractible,  while  the
remaining ones  are all  contractible.  If all  non-contractible edges
happen to be tested against the link condition before any contractible
edge is tested,  then the time for testing  all non-contractible edges
against  the  link  condition  is  $\Omega(  n_f^2  )$,  as  $n_f  \in
\Theta(n_v)$ by assumption, and there are as many as $2m$ edges of the
forms $[  v_i,x]$ and $[v_i ,  y ]$, where  each of them is  tested in
$\Theta( m )$ time because
\[
d_x = m = d_y \, .
\]

Schipper devised a more efficient algorithm by reducing the time spent
on each link condition test~\cite{SCH91}. For each vertex $u$ in $\T$,
his algorithm maintains a  dictionary $D_u$ containing all vertices in
$\textit{lk}(  u ,  K  )$,  where $K$  is  the current  triangulation.
Determining if an edge  $[ u , v ]$ in $K$  is contractible amounts to
verifying if  $D_v$ contains a vertex  $w$ in $\textit{lk}( u  , K )$,
with $w \neq v$ and $w \not\in \textit{lk}( [u,v] , K )$, which can be
done in $\mathcal{O}(  d_u \lg d_v )$ time, where  $d_u$ and $d_v$ are
the degrees of $u$ and $v$, respectively. He proved that if $K$ is not
irreducible then $K$ contains a contractible edge incident on a vertex
of degree  at most  $6$.  To  speed up the  search for  a contractible
edge,  the edge  chosen to  be tested  against the  link  condition is
always incident  on a  vertex of lowest  degree.  To  efficiently find
this  edge, his  algorithm  also maintains  a  global dictionary  that
stores all vertices of $K$  indexed by their current degree.  However,
this heuristic does not prevent the same (non-contractible) edge of $K$
from being  repeatedly tested against the  link condition.  Schipper's
algorithm runs in $\mathcal{O}( n_f \lg n_f  + g \lg n_f + g^4 )$ time
and requires $\mathcal{O}(n_f)$ space.

\begin{figure}[htb!]
\begin{center}
\includegraphics[height=1.8in]{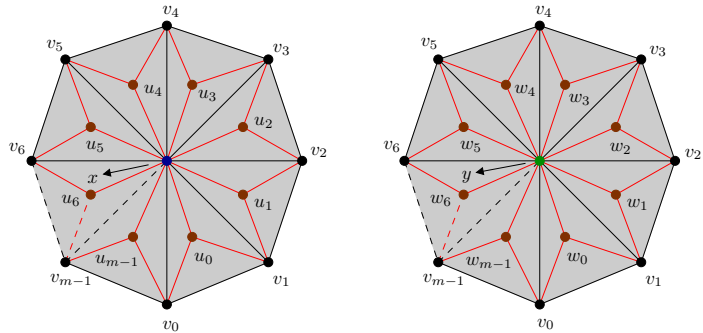}
\caption{\label{fig:badcase} A reducible  triangulation of the sphere
  cut open into two pieces.}
\end{center}
\end{figure}

Our algorithm  allows us to more efficiently  compute $\T^{\prime}$ by
testing each edge of $K$ against  the link condition at most once, and
by reducing the worst-case time complexity for the link condition test
even   further.     By   using    a   time   stamp    mechanism   (see
Section~\ref{sec:counting}  for  details), our  algorithm  is able  to
efficiently  determine if a  previously tested,  non-contractible edge
becomes   contractible   as   a   result  of   an   edge   contraction
(\textit{without  testing the edge  against the  link condition  for a
  second time}).  Our algorithm runs in $\mathcal{O}( g^2 + g \, n_f)$
time if  $g$ is  positive, and  it is linear  in $n_f$  otherwise.  In
either case, the space requirements are linear in $n_f$.

Edge contraction is a key operation for several \textit{mesh
  simplification} algorithms~\cite{LRCVWH03}.  The goal of these
algorithms is not to compute an irreducible triangulation, but
rather to decrease the level-of-detail (LOD) of a given triangulation
by reducing  its number  of vertices,  edges, and  triangles.  In
general,   contracted    edges   are   chosen    according   to   some
application-dependent   criterion,   such   as  preserving   geometric
similarity  between  the  input  and the  final  triangulation.

Garland  and Heckbert~\cite{GAR97} show  how to  efficiently combine
the edge  contraction operation with a quadric-based  error metric for
geometric   similarity.   Furthermore,   together  with   its  inverse
operation,  vertex  splitting,  the  edge contraction  operation  also
allows for  the construction of  powerful hierarchical representation
schemes  for  storing,   transmitting,  compressing,  and  selectively
refining   very  large   triangulations~\cite{HOP96,VEL01}.   However,
topology preservation is  not always desirable in the  context of mesh
simplification  applications, and to  the best  of our  knowledge, the
greedy algorithm  proposed by Cheng, Dey, and  Poon in~\cite{CDP04} is
the  only  simplification algorithm  whose  time  complexity has  been
analyzed.  

The  algorithm in~\cite{CDP04}  builds a  topology-preserving surface
triangulation  hierarchy  of  $\mathcal{O}(  n_v  + g^2  )$  size  and
$\mathcal{O}( \lg n_v + g )$  depth in $\mathcal{O}( n_v + g^2 )$ time
whenever  $n_v \ge  9182 g  - 222$  and  $g >  0$. Each  level of  the
hierarchy  is constructed  by  identifying and  contracting  a set  of
independent contractible edges in the triangulation represented by the
previous level.  A similar result for genus-$0$ surface triangulations
has been known for a long time~\cite{KIR83}, although the construction
of  the hierarchy  is  not  based on  edge  contractions. In  general,
however, we are not aware of  any attempts to bound the number of link
condition  tests  in  the  mesh  simplification  literature.   If
incorporated by simplification  algorithms, this distinguishing feature
of our algorithm, i.e., carrying out link condition tests faster, can
increase their overall simplification speed. 

\section{Algorithm}\label{sec:algo}
Our  algorithm takes  as input  a triangulation  $\T$ of  a surface 
$\s$  of  genus  $g$,   and  outputs  an  irreducible  triangulation 
$\T^{\prime}$ of the same surface.  The key idea behind our algorithm
is to iteratively  choose a vertex $u$ (rather than  an edge) from the
current triangulation, $K$, and then {\em process} $u$, 
%% i.e., 
which involves contracting (contractible) edges incident on $u$ until
no edge incident on $u$ is contractible, i.e., until $u$ becomes a
trapped vertex.  It was shown in~\cite{SCH91} that once vertex $u$
becomes trapped, it cannot become a loose vertex again as the result
of a topology-preserving edge contraction. 

\begin{lem}[\cite{SCH91}]\label{lem:trapped}
  Let $\T$ be  a surface triangulation, $v$ a  trapped vertex of $\T$,
  and $e$ a contractible edge of  $\T$.  If $e$ is contracted in $\T$,
  then $v$ remains trapped in $\T-e$.
\end{lem}

When the currently processed vertex $u$ becomes trapped (or if $u$ is
already trapped when it is chosen by the algorithm), another vertex
from the current triangulation is chosen and processed by the
algorithm until all vertices are processed, at which point the
algorithm ends.  Since all vertices in the output triangulation
$\T^{\prime}$ have been processed by the algorithm, and since all
edges contracted by our algorithm are contractible,
Lemma~\ref{lem:trapped} ensures that all vertices of $\T^{\prime}$ are
trapped.  It follows that triangulation $\T^{\prime}$ is irreducible.
It is worth noting that our algorithm requires no knowledge about the
embedding of $\T$, as all operations carried out by the algorithm are
purely topological, and hence they act on $G_{\T}$ only.

When contracting a  contractible edge $e = [ u ,  v ]$, our algorithm
does  not  merge  vertices  $u$  and  $v$  into  a  {\em  new}  vertex
$w$. Instead, either $u$ or $v$ is chosen to play the role of $w$, and
the other  vertex is merged  into the fixed  one. If $u$ is  the fixed
vertex, then  we say that $v$  \textit{is identified with}  $u$ by the
contraction  of $e$ (see  Figure~\ref{fig:contraction2}). When  $v$ is
identified with $u$  during the contraction of $e$,  every edge of the
form $[ v , z ]$ in $\T$ is replaced with an edge of the form $[ u , z
]$ in $\T - uv$, where $z  \in \textit{lk}( v, \T )$ and $z \not\in \{
u , x , y \}$, and $x$ and $y$ are the vertices in $\textit{lk}( e, \T
)$. We denote  the set $\{ u  , x , y \}$  by $\pd_{uv}$, and
the   set  $\{   z  \in   \textit{lk}(  v,   \T  )   \mid   z  \not\in
\pd_{uv} \}$ by $\po_{uv}$.

We assume  that $\T$  and all triangulations  resulting from  the edge
contractions  executed by  our algorithm  are stored  in  an augmented
doubly-connected edge list  (DCEL) data structure~\cite{BCKO08}, which
is briefly discussed in Section~\ref{sec:dcel}. A detailed description
of          the         algorithm         is          given         in
Sections~\ref{sec:procverts}-\ref{sec:procedges}. Section~\ref{sec:genus0}
discusses   the  particular  case   of  triangulations   of  genus-$0$
surfaces. Finally, Section~\ref{sec:comp}  analyzes the time and space
complexities of the algorithm.

\begin{figure}[htbp]
\begin{center}
\includegraphics[height=1.2in]{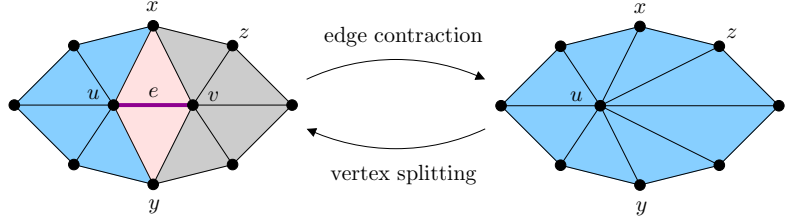}
\caption{\label{fig:contraction2} The  contraction of $e  = [ u ,  v
  ]$ in which $v$ is identified with $u$.}
\end{center}
\end{figure}

\subsection{Processing vertices}\label{sec:procverts}
To support the efficient processing  of vertices, the vertex record of
the DCEL  is augmented with six  attributes: $d$, $p$,  $n$, $c$, $o$,
and $t$,  where $d$, $c$,  $o$, and $t$  store integers, $p$  stores a
Boolean  value,  and  $n$  is  a  pointer  to  a  vertex  record  (see
Table~\ref{tab:attributes}).   We  denote each  attribute  $a$ of  a
vertex  $v$  of  the DCEL  by  $a(v)$.   The  value of  each  vertex
attribute is  defined with respect  to the vertex $u$  being currently
processed by the algorithm.  When $u$ is chosen to be processed by the
algorithm, its  attributes and all attributes of  its {\em neighbors},
i.e., the vertices in $\textit{lk}( u , K )$, where $K$ is the current
triangulation,  are  initialized  by  the  algorithm.   As  edges  are
contracted during the  processing of $u$, the attribute  values of the
neighbors of $u$ may change,  while other vertices become neighbors of
$u$ and have their attribute  values initialized.  If a vertex of $\T$
never  becomes a neighbor  of $u$  during the  processing of  $u$, its
attribute values do not change while $u$ is processed.

\begin{table}[htbp]
\begin{center}\small
\begin{tabular}{|c|l|} \hline
\textbf{Attribute} & \multicolumn{1}{c|}{\textbf{Description}} \\ \hline\hline
$d$ & the degree of $v$ \\
$p$ & indicates whether $v$ has already been processed by the algorithm \\
$n$ & indicates whether $v$ is a neighbor of $u$ \\
$c$ & number of critical cycles containing edge $[ u , v ]$ in $K$ \\
$o$ & time at which $v$ becomes a neighbor of $u$ \\
$t$  & time at which edge $[u,v]$ is removed from $\lue$ \\ \hline
\end{tabular}
\caption{\label{tab:attributes}  Attributes of  a vertex  $v$ during
  the processing of a vertex $u$.}
\end{center}
\end{table}

The  algorithm starts  by creating  a queue  $Q$ of  {\em unprocessed}
vertices, and by initializing the  attributes $d$, $p$, $n$, $c$, $o$,
and $t$ of each vertex $u$ of $\T$ (see Algorithm \ref{alg:continit}).
In  particular, for  each  vertex $u$  in  $\T$, its  degree $d_u$  is
computed and  stored in  $d(u)$, its attribute  $p(u)$ is set  to {\em
  false}, its attribute $n(u)$ is assigned the {\em null} address, and
its attributes $c(u)$, $o(u)$, and  $t(u)$ are assigned $0$, $-1$, and
$-1$, respectively.   Finally, a pointer to  the record of  $u$ in the
DCEL is inserted into $Q$.

After the initialization stage, the algorithm starts contracting edges
of $\T$ (see Algorithm \ref{alg:contractions}).  Each edge contraction
produces a new triangulation from the one to which the contraction was
applied.  The algorithm stores the currently modified triangulation in
a  variable, $K$.  Here,  we do  not distinguish  between $K$  and the
triangulation  stored in  it.   Initially,  $K$ is  set  to the  input
triangulation $\T$ and  the vertices in $Q$ are the  ones in $\T$. Let
$u$ be the vertex at the front of $Q$. The algorithm uses the value of
$p( u  )$ to decide whether  $u$ should be  processed.  In particular,
$p(u)$ is {\em  false} if and only if $u$  belongs to $K$ \textit{and}
$u$ has not been processed yet (i.e.,  $u$ is in $Q$).  If $p( u )$ is
{\em  true} when  $u$  is removed  from  $Q$, then  $u$ is  discarded.
Otherwise,   the   algorithm  processes   $u$,   i.e.,  it   contracts
(contractible) edges incident  on $u$ until $u$ is  trapped (see lines
5-36  of Algorithm  \ref{alg:contractions}). When  vertex  $u$ becomes
trapped, we say that $u$ has been {\em processed} by the algorithm.

{\setstretch{1.0}
\begin{algorithm}[ht!]\small
\caption{\textsc{Initialization}($\T$)}\label{alg:continit}
\begin{algorithmic}[1]
        \STATE $Q \leftarrow \emptyset$ \COMMENT{$Q$ is a queue of vertices}
        \FOR{each vertex $u$ in $\T$}
               \STATE $d(u) \leftarrow 0$
               \FOR{each  $v$ in $\textit{lk}( u , \T )$}
                       \STATE $d(u) \leftarrow d(u) + 1$
               \ENDFOR
               \STATE $p(u) \leftarrow \textit{false}$
               \STATE $n(u) \leftarrow \textit{nil}$
               \STATE $c(u) \leftarrow 0$
               \STATE $o(u) \leftarrow -1$
               \STATE $t(u) \leftarrow -1$
               \STATE insert a pointer to $u$ into $Q$.
        \ENDFOR
	\RETURN $Q$
\end{algorithmic}
\end{algorithm}}

Two doubly-connected linked lists, $\lue$  and $\lte$, are used by the
algorithm to  store edges incident  with $u$ during the  processing of
$u$.  The  former is  the list of  {\em unprocessed edges},  while the
latter is  the list of {\em  tested edges}. At any  given time, $\lue$
stores the edges incident on $u$ that have not been tested against the
link condition yet, while $\lte$ stores the edges incident on $u$ that
have  been  tested  against  the  link condition  before,  during  the
processing of  $u$, and  failed the test.  List $\lue$  is initialized
with all edges  $[ u , v ]$ of  $K$ such that $p( v  )$ is {\em false}
(lines  7-19  of  Algorithm \ref{alg:contractions}),  while  list
$\lte$   is    initially   empty    (see   line   20    of   Algorithm
\ref{alg:contractions}).

To process  $u$, the  algorithm removes one  edge, $[  u , v  ]$, from
$\lue$ at  a time and determines whether  $[ u , v  ]$ is contractible
(lines 23-30 of Algorithm \ref{alg:contractions}).  If so, $[ u , v ]$
is  contracted; else  it is  inserted  into $\lte$.  Once list  $\lue$
becomes  empty, the algorithm  considers list  $\lte$ (lines  31-33 of
Algorithm  \ref{alg:contractions}).  List  $\lte$  contains all  edges
incident  on $u$  that have  been  tested against  the link  condition
during the processing of edges in $\lue$ and failed the test. However,
while in  $\lte$, an  edge may have  become contractible again  as the
result of the contraction of another edge in $\lue$.  If so, Procedure
\textsc{ProcessEdgeList}$()$   in   Algorithm~\ref{alg:processlistS}
will find and contract this edge.

\begin{algorithm}[ht!]\small
\caption{\textsc{Contractions}($ \T , Q $)}\label{alg:contractions}
\begin{algorithmic}[1]
        \STATE $S \leftarrow \emptyset$ \COMMENT{$S$ is a stack for maintaining edge contraction information}
        \STATE $K \leftarrow \T$
        \STATE $\textit{ts} \leftarrow 0$ 
        \WHILE {$Q \neq \emptyset$}
               \STATE remove a vertex $u$ from $Q$ \COMMENT{vertex $u$ is chosen to be processed}
               \IF {\textit{not} $p(u)$}
                        \STATE $\lue \leftarrow \emptyset$ \COMMENT{$\lue$ is the list of unprocessed edges}
                        \FOR {each  $v$ in $\textit{lk}(u,K)$}
                                \STATE $n(v) \leftarrow u$ \COMMENT{mark $v$ as a neighbor of $u$}
                                \STATE $o(v) \leftarrow \textit{ts}$ \COMMENT{set the time at which $v$ is found to be a neighbor of $u$}
                                \STATE $t(v) \leftarrow -1$ \COMMENT{indicates that $[ u , v ]$ has not been tested yet}
                               \IF {\textit{not}  $p(v)$} 
                                      \IF {$d(v) = 3$}
                                           \STATE insert $[u,v]$ at the front of $\lue$
                                      \ELSE
                                           \STATE insert $[u,v]$ at the rear of $\lue$
                                      \ENDIF
                                \ENDIF \COMMENT{inserts $[ u , v ]$ into $\lue$ whenever $p(v)$ is {\em false}}
                        \ENDFOR \COMMENT{$\lue$ stores all vertices in $\textit{lk}(u,K)$ that have not been processed yet}
                        \STATE $\lte \leftarrow \emptyset$ \COMMENT{$\lte$ is the list of tested edges}
                        \REPEAT
                             \WHILE {$\lue \neq \emptyset$}
                                   \STATE \text{remove edge $e = [ u , v ]$ from $\lue$}
                                   \STATE $t(v) \leftarrow \textit{ts}$
                                   \IF {$d(v) = 3$}
                                        \STATE{\textsc{ProcessVertexOfDegreeEq3}($e$, $K$, $S$, $\lue$, $\lte$, \textit{ts} )}
                                   \ELSE
                                        \STATE{\textsc{ProcessVertexOfDegreeGt3}($e$, $K$, $S$, $\lue$, $\lte$, \textit{ts} )}
                                  \ENDIF
                             \ENDWHILE \COMMENT{processes all edges in $\lue$}
                             \IF {$\lte \neq \emptyset$}
                                  \STATE \textsc{ProcessEdgeList}$( K , S, \lue , \lte , \textit{ts} )$ \COMMENT{process contractible edges in $\lte$}
                             \ENDIF
                       \UNTIL{$\lue = \emptyset$}
                       \STATE{$p(u) \leftarrow \textit{true}$} 
               \ENDIF \COMMENT{vertex $u$ is now processed}
        \ENDWHILE
	\RETURN $( K , S )$
\end{algorithmic}
\end{algorithm}

Recall that  if an edge  $[ u ,  v ]$ in  $K$ is contracted,  then $u$
becomes incident on edges  of the form $[ u , z ]$  in $K - uv$, where
$z$  is a vertex  in $\po_{uv}$  (see Figure~\ref{fig:contraction2}).
These \textit{new} edges are always inserted into $\lue$, as they have
not  been  processed  yet.   Hence,  the contraction  of  an  edge  by
Algorithm~\ref{alg:processlistS}  may  cause  the insertion  of  new
edges into $\lue$.  If so,  list $\lue$ becomes nonempty and its edges
are processed  again.  Otherwise, list  $\lue$ remains empty,  and the
processing of $u$  ends with the value of $p(u)$ set  to {\em true}. A
key feature of  our algorithm is its ability  to determine which edges
from  $\lte$ become  contractible,  after the  contraction of  another
edge, without  testing those edges  against the link  condition again.
To  do so,  the algorithm  relies on  a \textit{time  stamp} mechanism
described in detail in Section~\ref{sec:counting}.

\subsection{Testing edges}\label{sec:testedges}
To decide if an edge $[ u  , v ]$ removed from $\lue$ is contractible,
the link condition test is applied to $[u,v]$, except when the degree
$d_v$     of     $v$    is     $3$     (see     lines    25-29     of
Algorithm~\ref{alg:contractions}). If $d_v  = 3$, then $[ u  , v ]$ is
always contractible,  unless the  degree $d_u$ of  $u$ is  also $3$,
which is  the case if and  only if the current  triangulation $K$ is
$\T_4$.

\begin{pro}\label{prop:contractible1}
  Let $K$  be a surface  triangulation, and let  $v$ be any  vertex of
  degree $3$ in  $K$.  If $K$ is (isomorphic to)  $\T_4$, then no edge
  of $K$ is contractible. Otherwise, every edge of $K$ incident on $v$
  is a contractible edge in $K$.
\end{pro}

\begin{proof}
  Let $v$ be  any vertex of $K$ whose degree, $d_v$,  is equal to $3$.
  Then, $\textit{lk}(  v , K  )$ contains exactly three  vertices, say
  $u$,  $x$, and  $y$  (see Figure~\ref{fig:vertexofdegree3}).   Since
  there  are  exactly  two  faces  incident  on  $[  u  ,  v  ]$  (see
  Lemma~\ref{lem:tritwofaces}), the other vertices  of these two faces are
  $x$ and  $y$, else $v$ would  have degree greater than  $3$.  So, we
  get $\textit{lk}( [ u , v ] , K ) = \{ x , y \}$. We claim that $[ u
  , v  ]$ is contractible  if and only  if $K$ is not  (isomorphic to)
  $\T_4$.  Suppose that $K$ is not isomorphic to $\T_4$. Then, face $[
  u , x ,  y ]$ is not in $K$, which means that  $\textit{lk}( u , K )
  \cap \textit{lk}(  v , K )  = \{ x ,  y \}$.  Conversely,  if $K$ is
  isomorphic to  $\T_4$ then face  $[ u ,  x , y  ]$ is in  $K$, which
  implies that $\textit{lk}( u , K )  \cap \textit{lk}( v , K ) = \{ x
  ,  y ,  [ x  , y  ] \}$.   By the  link condition,  $[ u  , v  ]$ is
  contractible if and only if  $K$ is not isomorphic to $\T_4$.  Since
  every vertex of $\T_4$ has degree $3$, the claim follows.
\end{proof}

\begin{figure}[htbp]
\begin{center}
\includegraphics[height=1.5in]{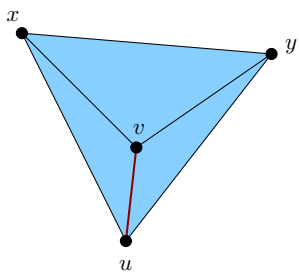}
\caption{\label{fig:vertexofdegree3} A  vertex, $v$, of  degree $3$ in
  $K$. Edge $[ u , v ]$ is  non-contractible if and only if $[ u , x ,
  y ] \in K$.}
\end{center}
\end{figure}

Proposition~\ref{prop:contractible1} implies that if $d_v = 3$, we can
decide  whether $[  u  , v]$  is  contractible by  determining if  the
current triangulation  $K$ is  isomorphic to $\T_4$.   Testing whether
$K$ is  isomorphic to $\T_4$ amounts to  checking if $d_u =  d_v = 3$,
which     can    be    done     in    constant     time.     Procedure
\textsc{ProcessVertexOfDegreeEq3}$()$           in           Algorithm
\ref{alg:processdegree3} is executed if $d_v$ is equal to $3$ (line 26
of Algorithm~\ref{alg:contractions}).  If $d_u$  is also equal to $3$,
then  $K$  is isomorphic  to  $\T_4$ and  nothing  is  done (line  2).
Otherwise,          procedure         \textsc{Contract}$()$         in
Algorithm~\ref{alg:contract}  is invoked  to  contract $[u,v]$.   This
procedure is  discussed in detail  in Section~\ref{sec:counting} along
with  lines  4-24  of  Algorithm~\ref{alg:processdegree3},  which  are
related to the time stamp mechanism for counting critical cycles.

\begin{algorithm}[htb!]\small
\caption{\textsc{ProcessVertexOfDegreeEq3}( $ e , K , S,  \lue , \lte ,  \textit{ts}$ )}\label{alg:processdegree3}
\begin{algorithmic}[1]
\STATE get the vertices $u$ and $v$ of $e$ in $K$
\IF{$d(u) \neq 3$}
    \STATE \textsc{Contract}( $e, K, S, \lue , \lte , \textit{ts}$ )
    \COMMENT{contract edge $e=[u,v]$}
    \STATE let $x$ and $y$ be the vertices in $\textit{lk}( e , K )$
    \IF{$t(x) \neq -1$ \textit{and} $t(y) \neq -1$}
        \STATE $c(x) \leftarrow c(x) - 1$ \COMMENT{edge $[u,x]$ is in $\lte$; a critical cycle containing it is gone}
        \STATE $c(y) \leftarrow c(y) - 1$ \COMMENT{edge $[u,y]$ is in $\lte$; a critical cycle containing it is gone}
        \IF{$c(x) = 0$}
             \STATE move $[u,x]$ to the front of $\lte$ \COMMENT{$[u,x]$ is now contractible}
        \ENDIF 
        \IF{$c(y) = 0$}
             \STATE move $[u,y]$ to the front of $\lte$ \COMMENT{$[u,y]$ is now contractible}
        \ENDIF
    \ELSIF{$t(x) \neq -1$ \textit{and} $t(x) \ge o(y)$}
        \STATE $c(x) \leftarrow c(x) - 1$ \COMMENT{edge $[u,x]$ is in $\lte$; a  critical cycle containing it is gone}
        \IF{$c(x) = 0$}
             \STATE move $[u,x]$ to the front of $\lte$ \COMMENT{$[u,x]$ is now contractible}
        \ENDIF 
    \ELSIF{$t(y) \neq -1$ \textit{and} $t(y) \ge o(x)$} 
        \STATE $c(y) \leftarrow c(y) - 1$  \COMMENT{edge $[u,y]$ is in $\lte$; a critical cycle containing it is gone}
        \IF{$c(y) = 0$}
             \STATE move $[u,y]$ to the front of $P$ \COMMENT{$[u,y]$ is now contractible}
        \ENDIF
    \ENDIF \COMMENT{update the value of $c(x)$ and $c(y)$ after contracting $[u,v]$}
\ENDIF \COMMENT{if $K$ is not isomorphic to $\T_4$}
\end{algorithmic}
\end{algorithm}

If  $d_v >  3$  when line  25  of Algorithm~\ref{alg:contractions}  is
reached, then $[ u , v ]$  is tested against the link condition. As we
pointed out in Section~\ref{sec:rwork}, if no special care is taken or
no special data  structure is adopted, the  test $[ u , v  ]$ can take
$\Theta(  d_u  \cdot  d_v  )$  time. To  reduce  the  worst-case  time
complexity of the link condition  test, our algorithm makes use of the
$n$ attribute.  During the processing of $u$,  we set $n( w ) = u$ for
every vertex $w$ in $K$ with $[ u , w ] \in K$.

Since $d_v > 3$, $K$ cannot be  isomorphic to $\T_4$.  So, edge $[ u ,
v ]$  is non-contractible  if and only  if $[ u  , v  ]$ is part  of a
critical cycle  in $K$ (see Figure~\ref{fig:neighbor}) ,  i.e., if and
only if  $u$ and $v$ have a  common neighbor $z$ such  that $z \not\in
\textit{lk}( [u,v]  , K)$ (i.e.,  $z \in \po_{uv}$). Conversely,  if $u$
and $v$ do  not have a common neighbor other than  the two vertices in
$\textit{lk}( [u,v] , K)$, then they  cannot be part of a $3$-cycle in
$K$. By examining  the $n$ attribute of the  vertices in $\po_{uv}$, our
algorithm can determine if $u$ and $v$ have  a common neighbor in
$\po_{uv}$, which can be done in $\mathcal{O}( d_v )$ time.

\begin{figure}[htb!]
\begin{center}
\includegraphics[height=1.2in]{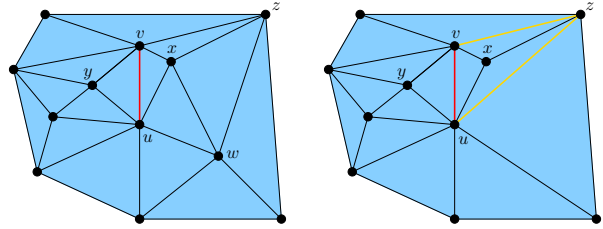}
\caption{\label{fig:neighbor} Vertex  $z$ is a neighbor  of vertex $u$
  in the right triangulation, but not in the left one.}
\end{center}
\end{figure}

Procedure           \textsc{ProcessVertexOfDegreeGt3}$()$           in
Algorithm~\ref{alg:processdegreegt3}   is  the  one   responsible  for
testing $[  u , v ]$  against the link  condition when $d_v >  3$
(line  28 of Algorithm  \ref{alg:contractions}).  This  procedure
tests edge $[ u , v ]$ against the link condition, which amounts to counting
the number  of critical cycles in $K$  containing $[ u ,  v ]$. Rather
than merely checking the value of the $n$ attribute of all vertices in
$\po_{uv}$,  Algorithm~\ref{alg:processdegreegt3}  computes  the
number $c(v)$  of critical cycles in  $K$ that contain edge $[  u , v ]$.
To that end, Algorithm~\ref{alg:processdegreegt3} (lines 2-12)
counts the number of vertices $z$ in  $\po_{uv}$ such that $n( z ) = u$,
which is precisely the number  of critical cycles in $K$ containing $[
u , v  ]$.  If $c(v)$ equals zero, then edge $[ u  , v ]$ is
contracted. Otherwise, edge $[ u ,  v]$ is inserted into $\lte$, as it
has been  tested against  the link condition  and has failed  the test
(lines 13-17 of Algorithm \ref{alg:processdegreegt3}).

\begin{algorithm}[htb!]\small
\caption{\textsc{ProcessVertexOfDegreeGt3}( $ e , K , S,  \lue , \lte ,  \textit{ts}$ )}\label{alg:processdegreegt3}
\begin{algorithmic}[1]
\STATE get the vertices $u$ and $v$ of $e$ in $K$
\FOR {each $z$ in $\textit{lk}( v , K )$}
    \IF {$z \in \po_{uv}$ \textit{and} $n( z ) = u$}
         \STATE $c( v ) \leftarrow c( v ) + 1$ \COMMENT{$(u,v,z)$ is a critical cycle in $K$; increment $c(v)$}
         \IF {$t( z ) \neq -1$ \textit{and} $t( z ) < o( v )$}
             \STATE $c( z ) \leftarrow c( z ) + 1$ \COMMENT{found a critical cycle in $K$ containing $[u,z]$}
             \IF {$c(z) = 1$}
                  \STATE move $[ u , z ]$ to the rear of $\lte$
                  \COMMENT{$c(z)$ was zero before}
             \ENDIF
        \ENDIF
    \ENDIF \COMMENT{updates the number, $c(z)$, of critical cycles in $K$ containing $[u,z]$}
\ENDFOR \COMMENT{computes the number, $c(v)$, of critical cycles in $K$ containing $[u,v]$}
\IF {$c( v ) = 0$}
    \STATE \textsc{Contract}$( e , K , S , \lue , \lte, \textit{ts} )$
    \COMMENT{$[u,v]$ in $K$ is contractible}
\ELSE
    \STATE insert $[ u , v ]$ at the rear of $\lte$ \COMMENT{edge $[u,v]$ is non-contractible in $K$, as $c(v) > 0$}
\ENDIF
\end{algorithmic}
\end{algorithm}

Lines 5-10 of Algorithm  \ref{alg:processdegreegt3} are related to the
counting of critical  cycles containing edge $[ u , z  ]$, for each $z
\in    \po_{uv}$,     in    triangulation    $K     -    uv$.     See
Section~\ref{sec:counting}  for  further  details. Furthermore,  while
edge     $[     u     ,     v     ]$    is     being     tested     by
Algorithm~\ref{alg:processdegreegt3}, the  degree $d_v$ of  $v$ in $K$
may not be  the same as the degree $d_v^{\prime}$ of  $v$ in the input
triangulation $\T$. In fact, during  the processing of any vertex $w$,
the degree  of $w$  can only  increase or remain  the same,  while the
degree  of any  other vertex  can only  decrease or  remain  the same.
Hence, we get $d_v \le d_v^{\prime}$,  and we can say that the time to
test  $[ u  , v  ]$  against the  link condition  is in  $\mathcal{O}(
d_v^{\prime}  )$. In  general, the  overall time  spent with  the link
condition test during the processing of $u$ is given by
\[
\sum_{w \in W_u} \mathcal{O}( d_w^{\prime} )
\, ,
\]
where $W_u$ is the set of all vertices $w$ of $\T$ such that $[ u , w
]$ is an edge tested  against the link condition during the processing
of $u$, and  $d_u^{\prime}$ and $d_w^{\prime}$ are the  degrees of $u$
and $w$  in $\T$, respectively. 

\subsection{Counting critical cycles}\label{sec:counting}
Let $[ u , v ]$ be a contractible edge in $K$ during the processing of
$u$,  and refer  to  Figure~\ref{fig:createcc}.  If  $[  u ,  v ]$  is
contracted, then every  $\ell$-cycle containing $[ u , v  ]$ in $K$ is
shortened  and transformed  into  a $(  \ell  - 1  )$-cycle in  $K-uv$
containing $u$.  Thus, every $4$-cycle  containing $[ u  , v ]$  in $K$
gives rise to a $3$-cycle in $K-uv$ containing vertex $u$.

\begin{figure}[htb!]
\begin{center}
\includegraphics[width=4.5in]{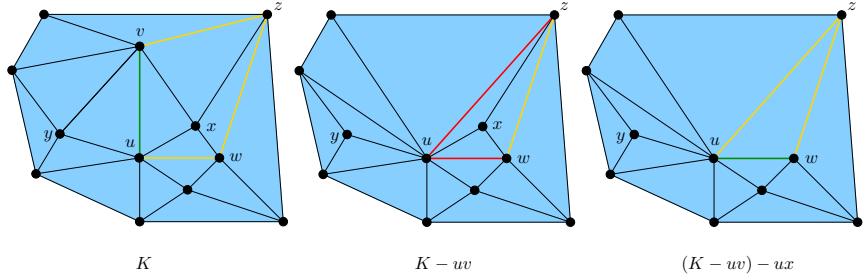}
\caption{\label{fig:createcc} Cycle  $(u,z,w)$ is critical  in $K-uv$,
  and non-critical in $(K-uv)-ux$.}
\end{center}
\end{figure}

Observe that a contractible edge in $K$ may become non-contractible in
$K-uv$.  In  particular, if a  newly created $3$-cycle,  which results
from an edge contraction, does not  bound a face in $K-uv$, then every
edge that belongs to it  is non-contractible in $K-uv$.  For instance,
if edge $[ u , v ]$  is contracted in triangulation $K$ on the left of
Figure~\ref{fig:createcc}, then $( u , v , z , w )$ gives rise to $( u
, z  , w)$  in $K-uv$, which  is critical.  Observe also that  an edge
contraction can  make a critical  cycle non-critical in  the resulting
triangulation.  For  instance, if edge  $[ u ,  x ]$ is  contracted in
triangulation  $(K- uv)$  in Figure~\ref{fig:createcc},  then critical
cycle $( u , z , w )$ in $K-uv$ becomes non-critical in $(K- uv)-ux$.

In  general,  if  the contraction  of  an  edge  $[  u  , v  ]$  in  a
triangulation  $K$  identifies  a   degree-$3$  vertex  $v$  with  the
currently processed vertex,  $u$, then the cycle defined  by the three
edges of $\textit{lk}( v  , K )$ bounds a face in  $K - uv$, and hence
it cannot  be critical in $K -  uv$. Conversely, if $C$  is a critical
cycle in $K$ but  not in $K - uv$, then $C$ must bound  a face in $K -
uv$. But, this  is only possible if a vertex $z$  of $K$ is identified
with a vertex of $C$ by the edge contraction that produced $K-uv$ from
$K$. Thus, vertex $z$ must be  $v$, vertex $u$ must belong to $C$, and
$C$ must consist of the edges in $\textit{lk}( v , K )$.  Moreover, if
a  critical cycle  $C$ in  $K$  becomes non-critical  in $K  - uv$,  a
non-contractible edge in $K$ may become contractible in $K - uv$.

\begin{pro}\label{prop:contractible2}
  Let $K$  be a surface triangulation,  and let $f$  be a contractible
  edge  of  $K$.   If  a  non-contractible edge  $e$  of  $K$  becomes
  contractible in  $K-f$, then  $f$ must be  incident on  a degree-$3$
  vertex $v$  of $K$ and  $e$ must belong  to $\textit{lk}( v ,  K )$.
  Moreover,  $e$ belongs  to a  single  critical cycle  in $K$,  which
  consists of  the edges  in $\textit{lk}(  v , K  )$, and  this cycle
  becomes non-critical in $K-f$.
\end{pro}
\begin{proof}
  By assumption, edge $f$ is contractible in triangulation $K$. So, we
  can conclude that $K$ cannot be (isomorphic to) $\T_4$. Thus, if $e$
  is a  non-contractible edge in $K$,  then $e$ belongs  to a critical
  cycle,  say $C$,  in $K$.   Moreover, since  $e$ is  contractible in
  $K-f$, we can also conclude  that $C$ is non-critical in $K-f$. But,
  this means that  $f$ is incident on  a vertex, $u$, in $C$  and on a
  degree-$3$ vertex, $v$,  in $K$ such that $C$  consists of the edges
  in $\textit{lk}( v , K  )$.  Also, the contraction of $f$ identifies
  $v$  with  $u$.   We claim  that  $C$  is  the only  critical  cycle
  containing $e$ in $K$.  In fact, if $e$ belonged to another critical
  cycle, say $C^{\prime}$, in $K$,  then $C^{\prime}$ would have to be
  non-critical in $K  - f$; else $e$ would  remain non-contractible in
  $K -  f$. But, if  $C^{\prime}$ were non-critical  in $K -  f$, then
  $C^{\prime}$ would have to consist of the edges of $\textit{lk}( v ,
  K )$ as well. Thus, $C^{\prime} = C$, i.e., $C$ is the only critical
  cycle containing $e$ in $K$.
\end{proof}

Proposition~\ref{prop:contractible2}  allows us  to  devise, for  each
edge $e$ that has been tested against the link condition, a time stamp
mechanism to keep track of the  number of critical cycles to which $e$
belongs.  Recall  that all  such  edges $e$  are  stored  in the  list
$\lte$. The idea is quite simple.  Whenever a contractible edge $[ u ,
v ]$, with $d_v = 3$,  is contracted, the algorithm checks whether the
critical cycle  counter of $x$  and $y$ must  be \textit{decremented},
where $x$ and $y$ are the two vertices  of $\textit{lk}( [ u , v ] , K
)$. From  Proposition~\ref{prop:contractible2}, we know that $[  u , x
]$ and  $[ u  , y ]$  are the  only edges incident  on $u$  that could
become contractible in $K - uv$ (if they are non-contractible edges in
$K$). In  turn, if  $d_v >  3$ then the  algorithm checks  whether the
critical  cycle counter  of  all vertices  involved  in newly  created
$3$-cycles  of $K-uv$  must be  \textit{incremented}. This  is because
contractible edges  in $K$ may become non-contractible  in $K-uv$, but
not       the      other       way      around       according      to
Proposition~\ref{prop:contractible2}.  Furthermore,  the newly created
critical cycles must contain a new neighbor of $u$ in $K - uv$ (i.e, a
vertex in $\po_{uv}$).

The  time stamp mechanism  relies on  a global  \textit{time counter},
\textit{ts},  and  on  the  $o$  and $t$  attributes.   The  value  of
\textit{ts} is set to zero before any vertex of $\T$ is ever processed
(line 3  of Algorithm~\ref{alg:contractions}). Moreover,  the value of
\textit{ts}  is updated if  and only  if an  edge is  contracted. More
specifically,  the  value of  \textit{ts}  is  incremented  by one  by
Algorithm~\ref{alg:contract}   immediately  before  the   actual  edge
contraction occurs.

\begin{algorithm}[htbp!]\small
\caption{\textsc{Contract}( $ e , K , S,  \lue , \lte , \textit{ts}$ )}\label{alg:contract}
\begin{algorithmic}[1]
\STATE get the vertices $u$ and $v$ of $e$
\STATE get the vertices $x$ and $y$ of $\textit{lk}( e , K )$
\STATE $\textit{ts} \leftarrow \textit{ts} + 1$
\FOR { each $z$ in $\textit{lk}( v , K )$ }
     \IF { $z \in \po_{uv}$ }
          \STATE $n( z ) \leftarrow u$
          \STATE $c( z ) \leftarrow 0$
          \STATE $o( z ) \leftarrow \textit{ts}$
          \STATE $t( z ) \leftarrow -1$
     \ENDIF \COMMENT{vertex $z$ will become a neighbor of $u$ in $K-uv$}
\ENDFOR \COMMENT{initializes the $n$, $c$, $o$, and $t$ attributes of the new neighbors of $u$}
\STATE $p( v ) \leftarrow true$ \COMMENT{prevents $v$ from being selected for processing}
\STATE push a record with $v$, $[ u , v ]$, $[ v , x]$, $[ v , y
]$, $[u,v,x]$, and $[u,v,y]$ onto $S$
\STATE $\textit{temp} \leftarrow \emptyset$ \COMMENT{\textit{temp} is a temporary list of edges $[u,z]$ such that $z \in \po_{uv}$} 
\STATE \textsc{Collapse}$( e , K , \textit{temp})$ \COMMENT{updates the DCEL}
\STATE $d( x ) \leftarrow d( x ) - 1$ \COMMENT{updates the degree of $x$}
\STATE $d( y ) \leftarrow d( y ) - 1$ \COMMENT{updates the degree of $y$}
\STATE $d( u ) \leftarrow d( u ) + d( v ) - 4$ \COMMENT{updates the degree of $u$}
\IF { $d( x ) = 3$ \textbf{and} \textbf{not} $p(x)$ \textbf{and} $t(x)=-1$}
     \STATE move $[ u , x ]$ to the front of $\lue$
\ENDIF
\IF { $d( y ) = 3$ \textbf{and} \textbf{not} $p(y)$ \textbf{and} $t(y)=-1$}
     \STATE move $[ u , y ]$ to the front of $\lue$
\ENDIF
\FOR { each $[ u , z ]$ in $\textit{temp}$ }
     \IF { \textit{not} $p( z )$ }
          \IF { $d( z ) = 3$ }
               \STATE insert $[ u , z ]$ at the front of $\lue$
          \ELSE
               \STATE insert $[ u , z ]$ at the rear of $\lue$
          \ENDIF
     \ELSE    
          \STATE get the vertices $x$ and $y$ of $\textit{lk}( [ u , z ] , K )$
          \FOR { each $w$ in $\textit{lk}( z , K )$ }
                \IF { $w \not\in \pd_{uz}$ \textit{and}  $n( w ) = u$ \textit{and} $t(w) \neq  -1$}
                     \STATE $c( w ) \leftarrow c( w ) + 1$
                     \COMMENT{increment $c(w)$ to account for $(u,w,z)$}
                     \IF { $c(w) = 1$}
                          \STATE{move $[ u , w ]$ to the rear of $\lte$}
                          \COMMENT{$[u,w]$ is now non-contractible}
                     \ENDIF
                \ENDIF
         \ENDFOR
     \ENDIF
\ENDFOR \COMMENT{updates $c(z)$ if $[u,z]\in \lte$ and inserts $[u,z]$ in $\lue$ otherwise}
\end{algorithmic}
\end{algorithm}

The $o$ and $t$ attributes of every vertex $u$ of $\T$ are each set to
$-1$ during  the initialization stage  (Algorithm \ref{alg:continit}).
During the processing of a vertex  $u$, the value of the $o$ attribute
of a  vertex $v$ is changed  to \textit{ts} if  and only if $v$  is or
becomes a neighbor of $u$, i.e., right  before $[ u , v ]$ is inserted
into list  $\lue$ because $v$  is already a  neighbor of $u$  when the
processing     of      $u$     begins     (see      line     10     of
Algorithm~\ref{alg:contractions}) or because $v$ becomes a neighbor of
$u$ as the result of an  edge contraction during the processing of $u$
(in line  8 of Algorithm~\ref{alg:contract}).  The value  of $o(v)$ is
changed only once  during the processing of $u$,  and after the change
is  made $o(v)$  can  be  viewed as  the  \textit{time} the  algorithm
discovers that  $v$ is  in $\textit{lk}(  u, K )$.   In turn,  the $t$
attribute  of a  vertex $v$  may be  changed at  most once  during the
processing  of  $u$.   The  value  of $t(v)$  is  set  to  \textit{ts}
immediately before $[  u , v ]$ is removed from  list $\lue$ (see line
24 of  Algorithm~\ref{alg:contractions}).  Hence, after  the change is
made, $t(v)$ can be viewed  as the \textit{time} the algorithm decides
whether $[ u , v ]$ is contractible.

Before we describe  the time stamp mechanism, we  state two invariants
regarding list  $\lue$ and $\lte$, which  will also help  us prove the
correctness of the algorithm:

\begin{pro}\label{prop:listLUE}
  Let $u$  be any  vertex of $\T$  processed by the  algorithm.  Then,
  during the processing of vertex $u$, the conditions regarding $\lue$
  below are (loop)  invariants of the while and  repeat-until loops in
  lines    22-30    and     21-34,    respectively,    of    Algorithm
  \ref{alg:contractions}:
\begin{itemize}
\item[(1)] every edge $[ u , w ]$ in $\lue$ is an edge of the
  current triangulation, $K$;

\item[(2)] if  $[ u ,  w ]$ is  an edge in  $\lue$ such that  $d_w$ is
  greater than $3$, then edge $[ u  , w ]$ cannot precede an edge $[ u
  , z ]$ in $\lue$ such that $d_z$ is equal to $3$;

\item[(3)] the value of $p(z)$ is \textit{false}, for every vertex $z$
  such that $[ u , z ]$ is in $\lue$;

\item[(4)] the value of $o(z)$ is no longer $-1$, for every vertex $z$ such that
  $[ u , z ]$ is in $\lue$;

\item[(5)] the value of $t(z)$ is $-1$, for every vertex $z$ such that
  $[ u , z ]$ is in $\lue$;

\item[(6)] the value of $c(z)$ is $0$, for every vertex $z$ such that
  $[ u , z ]$ is in $\lue$; and

\item[(7)]  no  edge  in  $\lue$  has been  tested  against  the  link
  condition before.
\end{itemize}
\end{pro}
\begin{proof}
See~\ref{sec:correctness}.
\end{proof}

\begin{pro}\label{prop:listLTE}
  Let $u$  be any  vertex of $\T$  processed by the  algorithm.  Then,
  during the processing of  $u$, the conditions regarding $\lte$ below
  are (loop)  invariants of the  while and repeat-until loop  in lines
  22-30 and 21-34, respectively, of Algorithm \ref{alg:contractions}:
\begin{itemize}
\item[(1)] lists $\lue$ and $\lte$ have no edge in common;

\item[(2)] if $[ u , z ]$ is an edge in $\lte$ then $t( z ) \ge o( z 
  ) > -1$; and

\item[(3)] every edge in $\lte$  was tested against the link condition
  exactly once and failed.
\end{itemize}
\end{pro}
\begin{proof}
See~\ref{sec:correctness}.
\end{proof}

Let $[ u ,  v ]$ be an edge removed from  $\lue$ during the processing
of vertex $u$, and let $K$  be the current triangulation at that time.
Suppose that  $[ u ,  v ]$ is  contractible. The time  stamp mechanism
distinguishes two cases: $d_v > 3$ and $d_v = 3$.

\textbf{Case $\bm{d_v  > 3}$.}  If $d_v$  is greater than  $3$ in $K$,
then Algorithm~\ref{alg:processdegreegt3} is executed  on $[ u , v ]$,
and Algorithm~\ref{alg:contract} is invoked  in line 14 to contract $[
u , v ]$ (refer to triangulation $K$ in Figure~\ref{fig:createcc}). As
we pointed out before, the contraction of $[ u , v ]$ may give rise to
one or more  critical cycles in $K - uv$.  So,  for every neighbor $z$
of $v$  in $K$ that  becomes a new  neighbor of $u$  in $K -  uv$, the
algorithm determines if  $u$ and $z$ have a  common neighbor, $w$.  If
so, then $( u ,  v , z , w )$ is a $4$-cycle  in $K$, shortened by the
contraction of $[ u, v ]$, that gave rise to critical cycle $( u , z ,
w )$ in $K-uv$.  If edge $[ u ,  w ]$ is in $\lte$, then $c( w )$ must
be incremented by $1$ to account for the newly created critical cycle,
$( u , z , w )$, in $K - uv$.  Otherwise, nothing needs to be done, as
either  $[ u  ,  w  ]$ is  still  in $\lue$  or  vertex  $w$ has  been
processed.

Lines 4-11 of \textsc{Contract}$()$ (see Algorithm~\ref{alg:contract})
visit all neighbors  $z$ of $v$ in $K$ that  become neighbors of $u$
in  $K -  uv$. Procedure  \textsc{Collapse}$()$, invoked  in  line 15,
contracts  $[  u  , v  ]$,  updates  the  DCEL,  and returns  a  list,
\textit{temp},  with the  new neighbors  $z$ of  $u$ in  $K-uv$. Lines
16-18 update the degrees of the vertices $x$, $y$, and $u$, where $x$
and $y$  are the  two vertices  in $\textit{lk}( [u,v]  , K  )$. Lines
19-24  ensure that  Proposition~\ref{prop:listLUE}(2) holds,  and lines
25-43 process the new neighbors $z$  of $u$ that were placed in list
\textit{temp}.   If   $p(z)$  is  \textit{true}   then  the  algorithm
determines whether  the contraction  of $[  u , v  ]$ in  $K$ produced
critical cycles  in $K - uv$  involving $[ u ,  z ]$.  If  this is the
case, then the  critical cycle counter of the third  vertex $w$ of the
cycle is updated accordingly.  If $p(z)$ is \textit{false} then $[ u ,
z]$ is inserted into $\lue$ in lines 27-31.

Suppose that $p(z)$  is {\em true}.  To determine  whether $u$ and $z$
share the same neighbor, $w$, in $K-uv$, or equivalently, to determine
the occurrence of a new critical cycle in $K - uv$ involving edge $[ u
, z  ]$, \textsc{Contract}$()$ compares $n(  w )$ with  $u$, for every
vertex  $w$ in  $\textit{lk}(  z ,  K -  uv  )$ such  that $w  \not\in
\pd_{uz}$ (see  lines 28-35 of Algorithm~\ref{alg:contract}).
If $n(  w ) = u$,  then $( u ,  z , w )$  is a critical cycle  in $K -
uv$. Otherwise, $( u , z , w )$ is {\em not} a cycle in $K - uv$. This
verification takes $\Theta( d_z )$-time,  where $d_z$ is the degree of
$z$ in  $K - uv$.  Since $z$  is a previously processed  vertex, it is
possible that  $d_z$ is greater  than the degree  of $z$ in  the input
triangulation, $\T$. The value of $c(w)$ must be incremented by $1$ to
account for  $( u ,  z ,  w )$ whenever  $[ u ,  w ]$ belongs  to list
$\lte$. Line  35 of Algorithm~\ref{alg:contract}  checks if $n( w  ) =
u$, $w \not\in \pd_{uz}$, and $t(  w ) \neq -1$. If the first
two conditions are true,  then $( u , z , w )$  is a critical cycle in
$K    -     uv$.    If    the    third    is     also    true,    then
Propositions~\ref{prop:listLUE} and \ref{prop:listLTE} tell us that $[ u
, w  ]$ is  in $\lte$.  Accordingly,  \textsc{Contract}$()$ increments
the critical  cycle counter, $c(w)$, of $w$  by $1$ in line  36 if and
only if the logical expression in line 35 evaluates to \textit{true}.

Suppose now  that $p(z)$ is {\em  false}.  Then, \textsc{Contract}$()$
simply  inserts  $[  u  ,  z  ]$  into  $\lue$  (see  lines  27-31  of
Algorithm~\ref{alg:contract}).  \textit{Our algorithm  need  not check
  whether $[ u , z ]$ is part  of a critical cycle in $K - uv$ at this
  point}. This verification is postponed to the moment at which $[ u ,
z    ]$    is    removed     from    $\lue$,    in    line    23    of
Algorithm~\ref{alg:contractions}, with  vertex $z$ labeled  as $v$. If
$[ u , v  ]$ is part of a critical cycle,  then $v$ cannot have degree
$3$, which means that $[ u , v ]$ is tested against the link condition
in  lines 2-12  of Algorithm~\ref{alg:processdegreegt3}.   During this
test, if $[ u , v ]$ is found to be part of a critical cycle, then the
third    vertex   involved    in   the    cycle   (labeled    $z$   in
Algorithm~\ref{alg:processdegreegt3}) may have its $c$ attribute value
incremented.   Indeed,  the value  of  $c(z)$  is  incremented by  $1$
whenever (a) $[ u , z ]$  is in $\lte$ (i.e., $t(z) \neq -1$), and (b)
$[ u , z ]$ was inserted in $\lte$ before $v$ became a neighbor of $u$
(i.e,  $t(z)   <  o(v)$).   Condition  (b)  is   necessary  to  ensure
correctness of  the counting  process.  Otherwise, $c(  z )$  could be
incremented twice for the  same cycle, $( u , v ,  z )$: one time when
edge $[  u , z  ]$ is tested  against the link condition,  and another
time when edge $[  u , v ]$ is tested against  the link condition. For
an example, let $[ u , r ]$ be an edge of $K$ such that $[ u , r ]$ is
part of a critical cycle, $( u , r , s )$, of $K$ by the time $[ u , r
]$     is     removed     from     $\lue$    in     line     23     of
Algorithm~\ref{alg:contractions}   (see   Figure~\ref{fig:createcc2}).
Suppose that $d_r > 3$ and  $p( s )=\textit{false}$.  Then, when $[ u,
r  ]$ is  given as  input to  Algorithm~\ref{alg:processdegreegt3}, we
have two possibilities:

\begin{figure}[htbp]
\begin{center}
\includegraphics[height=1.5in]{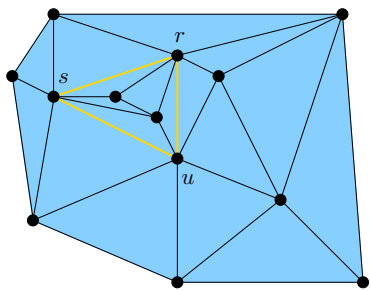}
\caption{\label{fig:createcc2} A  critical cycle  $( u ,  r , s  )$ in $K$.}
\end{center}
\end{figure}

\begin{itemize}
\item[(i)] $t( s ) < o( r )$:  edge $[ u , s ]$ was tested against the
  link condition  {\em before} $r$ becomes  a neighbor of  $u$. Thus, by
  the time $[ u , s ]$  was tested against the link condition, edge $[
  u , r]$ was not an edge of the current triangulation.  Consequently,
  line 6 of Algorithm~\ref{alg:processdegreegt3} could not be executed
  to increment $c( r )$ by $1$ (with $r$ labeled $z$) to account for a
  critical cycle that did not exist  at the time. For the same reason,
  line  4  of  Algorithm~\ref{alg:processdegreegt3}  cannot  increment
  $c(s)$ by $1$ to account for the same cycle either (with $s$ labeled
  $v$).   However,  when $[  u  ,  r ]$  is  tested  against the  link
  condition,  $c(  r  )$  is   incremented  by  $1$  in  line  $4$  of
  Algorithm~\ref{alg:processdegreegt3}  (with   $r$  labeled  $v$)  to
  account for $( u ,  r , s )$.  In addition, since $[ u  , s ]$ is in
  $\lte$, we have $t( s ) \neq  -1$. By hypothesis, we also know that $t(
  s ) <  o( r )$.  So, line  6 of Algorithm \ref{alg:processdegreegt3}
  is executed to increment  $c(s)$ by $1$ to account for $(  u , r , s
  )$ for the first time as well (with $s$ labeled $z$).

\item[(ii)] $t( s ) \ge o( r  )$: vertex $r$ was already a neighbor of
  $u$ when  $[ u , s]$ was  tested against the link  condition. So, we
  have two cases: (a) $[ u , s ]$ is tested against the link condition
  before $[ u ,  r ]$, and (b) $[ u , r ]$  is tested against the link
  condition  before $[  u ,  s ]$.   If (a)  holds, then  $c( s  )$ is
  incremented by $1$ to account  for $( u , r , s )$ when  $[ u , s ]$
  is  the  input  edge, $e$,  of  Algorithm~\ref{alg:processdegreegt3}
  (with  $s$ labeled  $v$). However,  the  value of  $c( r  )$ is  not
  incremented by $1$ to account for the same cycle, as line $6$ is not
  executed. The reason is that $[ u , r ]$ is still in $\lue$. So, $t(
  r ) = -1$, which implies that condition $t( z ) \neq -1$ fails (with
  $z =  r$) in line $5$  of Algorithm~\ref{alg:processdegreegt3}. When
  $[ u , r ]$ is tested against the link condition, the value of $c( r
  )$ is incremented by $1$ to account for $( u , r , s )$ in line 4 of
  Algorithm~\ref{alg:processdegreegt3} (with $r$ labeled $v$). At this
  point, $c(s)$ is not incremented to account  for $( u , r , s )$ for
  the {\em second} time, as condition  $t(z) < o(v)$ fails for $z = s$
  and $v = r$. If (b) holds,  then the situation is similar to (a); we
  just have to interchange the roles of $r$ and $s$. Therefore, to
  account for $(u,r,s)$, the values of   $c( r )$ and $c( s )$ are
  incremented by $1$ only once. 
\end{itemize}

\textbf{Case $\bm{d_v = 3}$.} If $v$ is a degree-$3$ vertex and $K$ is
not       (isomorphic      to)       $\T_4$,       then      procedure
\textsc{ProcessVertexOfDegreeEq3}$()$
(Algorithm~\ref{alg:processdegree3})  is  invoked   in  line  26  of
Algorithm~\ref{alg:contractions} to  contract $[ u,  v ]$. Let  $C$ be
the critical cycle of $K$ consisting of the edges in $\textit{lk}( v ,
K )$, i.e., $[ u , x ]$, $[ x , y ]$, and $[ u, y]$, where $x$ and $y$
are  the two  vertices of  $\textit{lk}( [  u  , v  ] ,  K )$.   Then,
Proposition~\ref{prop:contractible2} tells us  that the contraction of
edge $[ u , v ]$ makes $C$ non-critical in $K-uv$. If edge $[ u , x ]$
(resp.  $[  u , y ]$)  belongs to $\lte$, then  the value of  $c( x )$
(resp.  $c( y  )$) must be decremented by $1$ to  account for the fact
that one critical cycle in $K$ containing $[ u , x ]$ (resp.  $[ u , y
]$) is no longer critical in $K - uv$.

The value of $c( x )$ (resp.   $c( y )$) should only be decremented if
$c( x )$  (resp.  $c( y )$) was previously  incremented to account for
the           critical           cycle           that           became
non-critical. Algorithm~\ref{alg:processdegree3}  uses the values of
the $o$ and $t$  attributes of $x$ and $y$ to decide  whether $c( x )$
and $c( y )$ should be decremented as follows:

\begin{itemize}
\item If,  immediately after  the contraction of  $[u,v]$ in  $K$, the
  values of $t(x)$ and $t(y)$ are  both different from $-1$, then $[ u
  , x ]$ and $[  u , y ]$ are both in $\lte$, and  $c( x )$ and $c(y)$
  were incremented  by $1$  to account for  the existence of  $C$ when
  either $[  u ,  x ]$  or $[  u , y  ]$ was  tested against  the link
  condition  in line  4 of  Algorithm~\ref{alg:processdegreegt3} (with
  $x$ or $y$  labeled $v$). Both $x$ and $y$  are vertices with degree
  greater than $3$ in $K$, as it was  the case when $[ u , x ]$ and $[
  u , y  ]$ were removed from $\lue$ and then  tested against the link
  condition.   From the case  $d_v >  3$, we  know that  $c( x  )$ and
  $c(y)$ were incremented by $1$ to account for $C$ exactly once.  So,
  to account  for the fact that  $C$ is no longer  critical in $K-uv$,
  both $c(x)$ and $c(y)$ are  decremented by $1$ after the contraction
  of  $[ u  , v  ]$ in  line 3  of Algorithm~\ref{alg:processdegree3},
  which is done right after by lines 6 and 7.
 
\item If, immediately after the contraction of $[ u , v ]$, $t(x) \neq
  -1$ and $t(  y ) = -1$, then only  $[ u , x ]$  is in $\lte$. Vertex
  $y$ cannot be  trapped, as edge $[  v , y ]$ is  contractible in $K$
  (see Proposition \ref{prop:contractible1}) and no edge incident on a
  trapped  vertex can be  contractible (see  Lemma \ref{lem:trapped}).
  Thus, vertex $y$ has not been processed yet, which means that $[ u , y
  ]$ is still in $\lue$.  Moreover, the value of $c(x)$ is incremented
  to account  for $C$ if  and only if  $[ u ,  x ]$ was  inserted into
  $\lte$ after $y$ became a neighbor of  $u$ (i.e., $n( y ) = u$).  In
  fact,    if    $n(     y    )    =    u$    then     line    4    of
  Algorithm~\ref{alg:processdegreegt3} is executed for  $v = x$ and $z
  = y$, incrementing $c(x)$ by $1$ to account for $C$. Also, since $t(
  y )  = -1$, line  6 of Algorithm~\ref{alg:processdegreegt3}  is {\em
    not} executed for $z = y$,  and hence $c(y)$ is not incremented by
  $1$  to account for  $C$ while  $[u,x]$ is  tested against  the link
  condition.   Conversely, if $[  u ,  x ]$  was inserted  into $\lte$
  before $y$  became a neighbor  of $u$, then  $y$ is not a  vertex in
  $\po_{ux}$,      which       means      that      line       4      of
  Algorithm~\ref{alg:processdegreegt3} is not executed for $v = x$ and
  $z =  y$. Thus,  the value of  $c(x)$ is  not incremented by  $1$ to
  account for  $C$. This is consistent  with the fact that  $C$ is not
  even a cycle in the current triangulation by the time $[ u , x ]$ is
  tested against the link condition.

  When  $[ u  ,  x ]$  is inserted  into  $\lte$ after  $y$ becomes  a
  neighbor of $u$, we must have $o( x ) \ge o( y )$, as $[ u , y ]$ is
  still in list $\lue$ (i.e., $t(y) = -1$) and $[ u , x ]$ was removed
  from $\lue$ before $[ u , y ]$.  Since $t( w ) \ge o( w )$ for every
  vertex $w$ such that $[ u , w ]$ is in $\lte$, we must have that $t(
  x ) \ge o( y )$.  If $[  u , x ]$ is inserted into $\lte$ before $y$
  becomes a neighbor  of $u$, then $t( x )  < o( y )$, as  $t( x )$ is
  the time at  which $[ u ,  x ]$ is removed from  $\lue$ and inserted
  into  $\lte$,  while $o(y)$  is  the time  at  which  $y$ becomes  a
  neighbor of $u$. So, whenever $t(x) \neq  -1$, $t( y ) = -1$ and $t(
  x ) \ge o( y )$, the value  of $c(x)$ (but not the one of $c(y)$) is
  decremented by $1$ in line 15 of Algorithm~\ref{alg:processdegree3},
  right after the contraction of $[u,v]$ in line 3, to account for the
  fact that $C$ is no longer critical in $K-uv$.

\item  If $t(x)  = -1$  and $t(  y )  \neq -1$  immediately  after the
  contraction of $[u,v]$, then we have the same case as before, except
  that the roles of $x$ and $y$ are interchanged.

\item If both $t(x)$ and $t(y)$ are equal to $-1$, then neither $[ u ,
  x ]$ nor  $[ u , y ]$ are  in $\lte$, and thus there  is no need for
  updating $c(x)$ and $c(y)$ (as  none of them were incremented by $1$
  to account for $C$). Furthermore, since $C$ is no longer critical in
  $K-uv$, the values of $c(x)$ and $c(y)$ cannot be incremented by $1$
  to account for  $C$ when $[u,x]$ and $[u,y]$  are tested against the
  link condition, which  is also consistent with the  fact that $C$ is
  not critical in  $K-uv$. In fact, cycle $C$ may not  even be a cycle
  in $K$ when $[u,x]$ and $[u,y]$ are tested.
\end{itemize}

To  illustrate all cases  above, consider  the triangulation  $K_1$ in
Figure~\ref{fig:createcc3}.  Note that vertex $x$ is a neighbor of $u$
in $K_1$,  but vertex $y$ is  not.  Suppose that edge  $[ u ,  w ]$ is
contracted, making  $y$ a neighbor  of $u$ and  yielding triangulation
$K_2$ in Figure~\ref{fig:createcc3}.  Next, suppose that edge $[ u , z
]$     is    contracted,     yielding    triangulation     $K_3$    in
Figure~\ref{fig:createcc3}. Finally, since  $v$ is a degree-$3$ vertex
in $K_3$, edge $[ u , v ]$ is  contracted, which makes $( u , x , y )$
a non-critical cycle in  $K_3 - uv$. After $[ u ,  v ]$ is contracted,
the   values  of   $c(   x  )$   and   $c(  y   )$   are  updated   by
Algorithm~\ref{alg:processdegreegt3}.   To illustrate how  the updates
are carried  out by the  algorithm, consider the  following scenarios:
(i) $[  u , x ]$ is  removed from $\lue$ before  edge $[ u ,  w ]$ is,
(ii) $[  u , x ]$ is  removed from $\lue$ after  edge $[ u ,  w ]$ is,
(iii) $[ u ,  y ]$ is removed from list $\lue$ before $[  u , z ]$ is,
and (iv) $[ u , y ]$ is removed from list $\lue$ after $[u,z]$ is.

\begin{figure}[htb!]
\begin{center}
\includegraphics[width=4.7in]{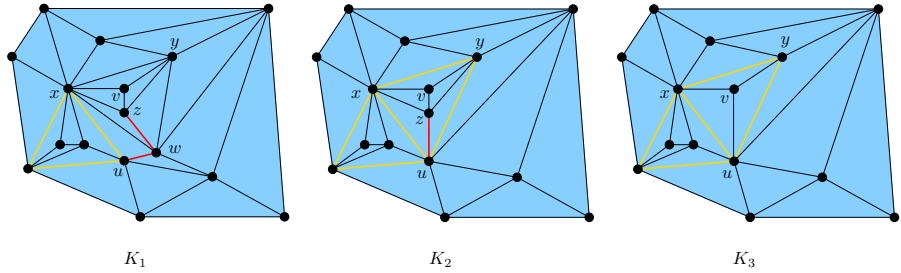}
\caption{\label{fig:createcc3} $K_2$  (resp.  $K_3$) is  obtained from
  $K_1$ (resp.  $K_2$) by contracting $[u,w]$ (resp. $[u,z]$). }
\end{center}
\end{figure}

Suppose that (i)  and (iii) hold. Then, both  $[u,x]$ and $[u,y]$ have
been  tested  against  the  link  condition by  the  time  $[u,v]$  is
considered for contraction in $K_3$. So, $t(x) \neq -1$ and $t(y) \neq
-1$ immediately after  the contraction of $[ u, v  ]$, as both $[u,x]$
and $[u,y]$ have  already been removed from list  $\lue$ (and inserted
into list $\lte$). Since $y$ is  not a neighbor of $u$ when $[u,x]$ is
tested against the link condition, the values of $c(x)$ and $c(y)$ are
not  incremented to  account for  critical cycle  $(u,x,y)$  in $K_3$.
Indeed, both $c(x)$  and $c(y)$ are incremented by  $1$ to account for
$(u,x,y)$  in $K_3$  while $[  u ,  y ]$  is tested  against  the link
condition.    This  is   done  by   lines   4  and   6  of   Algorithm
\ref{alg:processdegreegt3},  with $x$  and  $y$ labeled  $z$ and  $v$,
respectively, as  $t(x) \neq -1$  and $t(  x ) <  o( y )$.   After the
contraction  of   $[u,v]$  in  $K_3$,  both  $c(x)$   and  $c(y)$  are
decremented    by   $1$   in    lines   6    and   7    of   Algorithm
\ref{alg:processdegreegt3}, which accounts for the fact that $(u,x,y)$
is no longer critical in $K_3 - uv$.

If (iv) holds  instead, then only $[u,x]$ has  been tested against the
link condition  by the time  $[u,v]$ is considered for  contraction in
$K_3$.   So, $t(x)  \neq -1$  and $t(y)  = -1$  immediately  after the
contraction of $[  u, v ]$, as $[u,y]$ is still  in list $\lue$. Since
$y$ is not  a neighbor of $u$ when $[u,x]$ is  removed from $\lue$ and
tested against the link condition, we get $t(x) < o(y)$.  This implies
that   $c(x)$   is  not   decremented   by   $1$,   in  line   15   of
Algorithm~\ref{alg:processdegree3}), to account for the fact that $( u
, x , y )$ is not critical  in $K_3 - uv$. This is consistent with the
fact that  $c(x)$ is not incremented  by $1$ to  account for critical
cycle $( u  , x , y  )$ when $[ u ,  x ]$ was tested  against the link
condition.

Suppose that  (ii) holds.  Then, vertex  $y$ is already  a neighbor of
$u$ when $[  u , x ]$ is removed from  list $\lue$.  Furthermore, edge
$[u,x]$ is  removed from $\lue$ before  $[u,y]$ is, as  both edges are
inserted at  the rear of $\lue$  and $[u,x]$ is  inserted first. Since
$t(y) = -1$ and $t(y) < o(  x )$ by the time $[u,x]$ is tested against
the link condition,  both $c(x)$ and $c(y)$ are  incremented by $1$ in
lines 4  and 6 of  Algorithm \ref{alg:processdegreegt3}, respectively,
during the test.  If (iii) also  holds, then we get $t(x) \neq -1$ and
$t(y)  \neq  -1$  by the  time  $[u,v]$  is  tested against  the  link
condition. So, after the contraction  of $[u,v]$ in $K_3$, both $c(x)$
and   $c(y)$  are   decremented  by   $1$  in   lines  6   and   7  of
Algorithm~\ref{alg:processdegree3},  respectively, to account  for the
fact that  $(u,x,y)$ is not a critical  cycle in $K_3 -  uv$.  If (iv)
holds instead, then since $v$  is a degree-$3$ vertex, edge $[u,v]$ is
removed  from $\lue$  before $[u,y]$  is. This  means that  $[u,y]$ is
still  in $\lue$  after the  contraction  of $[u,v]$.   Thus, $t(y)  =
-1$. But,  since $y$ was  a neighbor of  $u$ when $[u,x]$  was removed
from $\lue$, we  get $t( x ) \ge  o( y )$. So, both  $c(x)$ and $c(y)$
are    decremented    by    $1$     in    lines    6    and    7    of
Algorithm~\ref{alg:processdegree3},  respectively, to account  for the
fact that $(u,x,y)$  is not critical in $K_3 -  uv$.  Thus, the values
of  $c(x)$ and  $c(y)$ are  consistently updated  by the  algorithm in
cases (ii) and (iii).

Finally,      suppose       that      triangulation      $K_2$      in
Figure~\ref{fig:createcc3}  is   the  initial  triangulation   in  the
processing  of $u$,  which means  that the  algorithm finds  that both
vertices  $x$  and  $y$  are   neighbors  of  $u$  in  lines  8-19  of
Algorithm~\ref{alg:contractions}, and  thus $o(x) = o(y)  \neq -1$. If
edge $[u,z]$ is  removed from $\lue$ before both  $[u,x]$ and $[u,y]$,
then  $[u,x]$ and  $[u,y]$ are  in list  $\lue$ immediately  after the
contraction of $[u,z]$. Since $v$ is a degree-$3$ vertex, edge $[u,v]$
is  inserted at the  front of  $\lue$, which  implies that  $[u,v]$ is
removed from $\lue$  before any of $[u,x]$ and  $[u,y]$ is.  So, after
the contraction  of $[u,v]$,  we get  $t(x) = t(y)  = -1$.   Thus, the
values  of  $c(x)$  and   $c(y)$  are  not  decremented  in  Algorithm
\ref{alg:processdegree3} to account for the fact that $(u,x,y)$ is not
a critical  cycle in  $K-uv$.  This is  consistent with the  fact that
none of  $c(x)$ and  $c(y)$ have been  incremented yet.   This example
shows that, to  consistently update the values of  the $c$ attributes,
our algorithm need not increment  counters every time a critical cycle
arises.

\subsection{Processing edges}\label{sec:procedges}
From  Proposition~\ref{prop:listLUE}, each  edge that  belongs  to list
$\lue$, during the processing of vertex $u$, is an edge of the form $[
u , z  ]$ such that $c( z ) =  0$, $o(z) \neq -1$, $t( z  ) = -1$, and
$p( z ) = \textit{false}$ (i.e, $z$ is in $Q$ and thus it has not been
processed  yet).   Furthermore, every  edge  in  $\lue$ is  eventually
removed from  $\lue$ during the execution  of the while  loop in lines
22-30 of Algorithm~\ref{alg:contractions}. Once an edge $[ u , z ]$ is
removed  from  $\lue$,  there  are  $3$ possibilities:  it  is  either
contracted, inserted into list $\lte$, or ignored.

If $[  u , z ]$ is  contracted, then it is  removed from triangulation
$K-uz$ and $p(z)$  is set to {\em true},  which prevents the algorithm
from trying to process vertex $z$ after it is removed from $Q$ in line
5 of Algorithm~\ref{alg:contractions}. If $[ u , z ]$ is inserted into
$\lte$, then  $[ u , z ]$  has been tested against  the link condition
and        found        to        be        non-contractible        by
Algorithm~\ref{alg:processdegreegt3}. Moreover,  immediately before $[
u   ,   z   ]$   is   inserted   into   $\lte$   (see   line   16   of
Algorithm~\ref{alg:processdegreegt3}), the value of  $c( z )$ is equal
to the number  of critical cycles containing  $[ u , z ]$  in $K$, and
the value  of $t( z )$ is  the time at which  $[ u , z  ]$ was removed
from $\lue$.  If $[  u , z ]$ is ignored, i.e.,  if the degree, $d_u$,
of $u$ and the  degree, $d_z$, of $z$ are both equal  to $3$ (see line
25    of    Algorithm~\ref{alg:contractions}    and    line    2    of
Algorithm~\ref{alg:processdegree3}),  then   $K$  is  (isomorphic  to)
$\T_4$, which means that $[ u , z ]$ is not contractible.

List $\lue$ will eventually be empty after finitely many iterations of
the  while loop  in lines  22-30  of Algorithm~\ref{alg:contractions}.
This  is  because   there  are  finitely  many  edges   in  the  input
triangulation $\T$, each edge  contraction yields a triangulation with
three fewer edges,  no vertex is created by the  algorithm, and no edge
removed  from  $\lue$ is  inserted  into  $\lue$  again.  So,  let  us
consider  the moment  at which  $\lue$ becomes  empty and  line  31 of
Algorithm~\ref{alg:contractions} is  reached. For every edge $[  u , z
]$ in the current triangulation, $K$, we distinguish two cases: (1) $[
u , z ] \not\in \lte$ and (2) $[ u , z ] \in \lte$.

If $[ u ,  z ] \not\in \lte$, then let us consider  the value of $p( z
)$. If  $p( z )$ is {\em  true}, then vertex $z$  was processed before
$u$ is removed from  $Q$. So, edge $[ u , z  ]$ is never inserted into
$\lue$. Since $z$  was processed before, it is  trapped, which implies
that $[ u , z ]$ is non-contractible. If $p( z )$ is {\em false}, then
$z$ is still in $Q$, and edge  $[ u , z]$ was ignored by the algorithm
after being  removed from list  $\lue$. So, triangulation $K$  must be
(isomorphic  to) $\T_4$,  which  implies  that all  edges  of $K$  are
non-contractible edges.

If $[ u , z ] \in \lte$, then  $[ u , z ]$ has been tested against the
link  condition  after being  removed  from  $\lue$  and found  to  be
non-contractible at  the time. List  $\lte$ is a temporary  holder for
this kind of edge. Every time  list $\lue$ becomes empty and the while
loop  in lines  22-30 of  Algorithm~\ref{alg:contractions}  ends, list
$\lte$  is  examined   by  procedure  \textsc{ProcessEdgeList}$()$  in
Algorithm~\ref{alg:processlistS},  which  is  invoked  by line  32  of
Algorithm~\ref{alg:contractions}  whenever $\lte$  is  nonempty.  This
procedure  checks  whether  an  edge  $[  u ,  v]$  in  $\lte$  became
contractible (after being  inserted into $\lte$). If so,  at least one
contractible edge in $\lte$ is contracted. The contraction of $[ u , v
]$ can generate  edges in $K - uv$  that are not in $K$.   This is the
case whenever $\po_{uv} \neq  \emptyset$, and the edges are precisely
the ones of the form $[u ,  w]$ in $K-uv$, with $w \in \po_{uv}$ (see
Figure~\ref{fig:contraction2}).

\begin{algorithm}[htb!]\small
\caption{\textsc{ProcessEdgeList}($K , S , \lue , \lte , \textit{ts}$)}\label{alg:processlistS}
\begin{algorithmic}[1]
     \WHILE {$\lte \neq \emptyset$}
            \STATE \text{let $e = [ u , v ]$ be the edge at the front of $\lte$} 
             \IF {$d(v) = 3$}
                   \STATE \text{remove edge $e = [ u , v ]$ from $\lte$}
                   \STATE \textsc{ProcessVertexOfDegreeEq3}$( e , K , S , \lue , \lte , \textit{ts} )$
             \ELSE
                 \STATE{\textbf{break}}\COMMENT{the edge at the front
                   of $\lte$ is not incident on a degree-$3$ vertex}
            \ENDIF
      \ENDWHILE \COMMENT{contract edges incident on degree-$3$ vertices}
      \IF {$\lte \neq \emptyset$}
           \STATE \text{let $e = [ u , v ]$ be the edge at the front of $\lte$} 
           \IF {$c(v) = 0$}
               \STATE \text{remove edge $e$ from $\lte$}\COMMENT{the degree of $v$ is greater than $3$}
               \STATE \textsc{Contract}( $e, K, S, \lue , \lte , \textit{ts}$
               )\COMMENT{since $c(v) = 0$, edge $e$ is contractible}
          \ENDIF
          \ENDIF\COMMENT{the  first edge  of $\lte$  is incident  on a
            vertex with degree greater than $3$}
\end{algorithmic}
\end{algorithm}

To efficiently find a contractible edge in $\lte$ or find out that one
does not  exist, our  algorithm always moves  every edge  $[ u ,  z ]$
whose value of  $c(z)$ is $0$ to the front  of $\lte$.  In particular,
every time that the value of $c(z)$ is decremented, for any vertex $z$
such that $[ u , z ]$ is in $\lte$, the algorithm verifies if $c( z )$
becomes $0$. If so,  edge $[ u , z ]$ is moved  to the front of $\lte$
(see       lines      8-13,       16-18,       and      21-23       of
Algorithm~\ref{alg:processdegree3}). In addition,  every time that the
value of $c(z)$ is incremented, for any  vertex $z$ such that $[ u , z
]$ is  in $\lte$, the algorithm verifies  if $c( z )$  becomes $1$. If
so, edge $[ u  , z ]$ is moved to the rear  of $\lte$ (see lines 37-39
of      Algorithm~\ref{alg:contract}     and     lines      7-9     of
Algorithm~\ref{alg:processdegreegt3}).  So,  the  following  invariant
regarding $\lte$ also holds:

\begin{pro}\label{prop:listLTE2}
  Let $u$  be the currently  processed vertex of the  algorithm. Then,
  the following  property regarding list $\lte$ is  a (loop) invariant
  of  the while and  repeat-until loops  in lines  22-30 and  21-34 of
  Algorithm \ref{alg:contractions}: no edge $[ u , z ]$ in $\lte$ such
  that $c(  z ) > 0$ can  precede an edge $[  u , w ]$  in $\lte$ such
  that $c( w )$ is equal to $0$.
\end{pro}
\begin{proof}
See~\ref{sec:correctness}.
\end{proof}

From Proposition~\ref{prop:listLTE2},  it suffices to  check the value
of $c( z )$, where $[ u , z  ]$ is the edge at the front of $\lte$, to
find  out whether  $\lte$ contains  a contractible  edge,  which takes
constant time.  Of course, the  correctness of this test relies on the
premise that $c(w)$  is indeed equal to the  number of critical cycles
in $K$  containing edge $[  u , w ]$,  for every edge  $[ u , w  ]$ is
$\lte$. The following states that this premise is valid:

\begin{pro}\label{pro:ccequality}
  Let  $u$ be any  vertex of  $\T$ processed  by the  algorithm. Then,
  whenever  line 32  of  Algorithm~\ref{alg:contractions} is  reached,
  during the processing of  $u$, we have that for every edge  $[ u , w
  ]$ in list $\lte$,  the value of $c( w )$ is  the number of critical
  cycles in $K$ containing edge $[ u , w  ]$, where $K$ is the current
  triangulation at the time.
\end{pro}
\begin{proof}
See~\ref{sec:correctness}.
\end{proof}

If  list  $\lue$ is  empty  after Algorithm~\ref{alg:processlistS}  is
executed, then  no edge  in $\lte$ is  contractible, which  also means
that  no edge  incident  on $u$  is  contractible. So,  vertex $u$  is
trapped and the processing of  $u$ ends.  Otherwise, the while loop in
lines 22-30  of Algorithm~\ref{alg:contractions} is  executed again to
process the edges in $\lue$.  It  is worth noting that {\em no edge is
  tested  against the  link condition  more than  once}.  Furthermore,
since $\lte$ is  a doubly-connected linked list, moving  an element of
$\lte$ from any position to the front or rear of $\lte$ can be done in
constant time if we have a pointer to the element.  With that in mind,
we included a pointer in the  edge record of our augmented DCEL to the
edge record  of $\lte$, which makes  it possible to access  an edge in
$\lte$ from the DCEL record of the edge in constant time.

\subsection{Updating the DCEL}\label{sec:dcel}
%
%Our   algorithm   stores  the   input   triangulation,   $\T$,  in   a
%Doubly-Connected Edge List  (DCEL) data structure~\cite{BCKO08}, which
%is augmented with vertex attributes $p$, $n$, $c$, $o$, and $t$ and an
%edge attribute  (i.e., a pointer to  a node in $\lte$).   Our DCEL has
%four records: one for vertices,  one for edges, one for triangles, and
%one for half-edges.  A {\em half-edge}  can be seen as one ``side'' of
%an edge, and it provides us  with a convenient way of representing the
%two possible orientations  of an edge.  Every edge  has two half-edges
%associated with it,  one for each ``side'', which are  said to be {\em
%  mates}  and   have  opposite   orientations,  as  shown   in  Figure
%\ref{fig:halfedge}.

%Half-edges enable us to consistently orient the boundary edges of each
%triangle of a triangulation. Indeed, the boundary of a triangle can be
%represented  by   a  sequence  of  three  half-edges   with  the  same
%orientation: the triangle bounded by the half-edge lies to the left of
%the half-edge  for an observer walking  along the edge. 

%\begin{figure}[htbp]
%\begin{center}
%\includegraphics[height=1.5in]{./halfedge}
%\caption{\label{fig:halfedge} Edge $e = [ u , v ]$ is incident on faces $\tau$ and
%  $\sigma$; $h_1$ and $h_2$  are the two half-edges of $e$.}
%\end{center}
%\end{figure}

Our   algorithm   stores  the   input   triangulation,   $\T$,  in   a
Doubly-Connected Edge List  (DCEL) data structure~\cite{BCKO08}, which
is augmented with vertex attributes $p$, $n$, $c$, $o$, and $t$ and an
edge attribute  (i.e., a pointer to  a node in $\lte$).   Our DCEL has
four records: one for vertices,  one for edges, one for triangles, and
one  for half-edges.  For each  half-edge,  $h$, that  bounds a  given
triangle of  the triangulation represented  by the DCEL, there  is the
{\em  next} half-edge  and the  {\em previous}  half-edge on  the same
boundary. The  destination vertex of $h$  is the origin  vertex of its
next  half-edge, while  the origin  vertex of  $h$ is  the destination
vertex of its previous half-edge.

The  record of  a  vertex $v$  stores  a pointer,  \textit{he}, to  an
arbitrary  half-edge whose  origin vertex  is $v$.   It  also contains
fields corresponding  to the $p$,  $n$, $c$, $o$, and  $t$ attributes.
The edge record of an  edge $e$ contains two pointers, \textit{h1} and
\textit{h2},  one for  each  half-edge  of $e$.   It  also contains  a
pointer to a record of $\lte$. The face record of a face $\tau$ stores
a  pointer,  \textit{he},  to  one  of the  three  half-edges  of  its
boundary. The half-edge record of  a half-edge $h$ contains a pointer,
\textit{or},  to its  origin vertex;  a pointer,  \textit{pv},  to its
previous half-edge;  a pointer, \textit{nx}, to its  next half-edge; a
pointer,  \textit{eg},  to  its  corresponding edge;  and  a  pointer,
\textit{fc}, to the face it belongs to.

Our DCEL also has  a procedure, called \textsc{Mate}$()$, that returns
the mate of  a given half-edge $h$ by comparing a  pointer to $h$ with
the pointers $\textit{h1(eg(h))}$ and $\textit{h2(eg(h))}$ of the edge
$eg(h)$   to    which   $h$   belongs.     If   $h$   is    equal   to
$\textit{h1(eg(h))}$,  then  \textsc{Mate}$()$  returns a  pointer  to
$\textit{h2(eg(h))}$.  Otherwise,  \textsc{Mate}$()$ returns a pointer
to $\textit{h1(eg(h))}$.

Procedure    \textsc{Collapse}$()$   in   Algorithm~\ref{alg:collapse}
implements   the  edge  contraction   operation  in   a  triangulation
represented  by our  DCEL. If  $e =  [  u ,  v ]$  is the  edge to  be
contracted   during   the   processing   of  a   vertex,   $u$,   then
\textsc{Collapse}$()$  removes edges  $[u,v]$,  $[v,x]$, and  $[v,y]$,
along with faces $[ u , v , x ]$ and $[ u, v , y ]$, where $x$ and $y$
are the two vertices in $\textit{lk}( [ u , v ] , K )$, and $K$ is the
current triangulation. In addition, \textsc{Collapse}$()$ replaces all
edges of the form $[  v , z ]$, where $z \not\in \{ u  , x , y \}$, by
edges   of   the   form   $[   u   ,  z   ]$.    Each   operation   in
\textsc{Collapse}$()$ takes constant time,  but there are $d_v-3$ edge
replacements.  So, the time  complexity of \textsc{Collapse}$()$ is in
$\Theta( d_v )$.

\begin{algorithm}[htb!]\small
\caption{\textsc{Collapse}($e, K , \textit{temp}$)}\label{alg:collapse}
\begin{algorithmic}[1]
     \STATE get the vertices $u$ and $v$ of $e$
     \STATE $h \leftarrow \textbf{if}~\textit{or}(\textit{h1}(e)) \neq u~\textbf{then}~\textit{or}(\textit{h1}(e))~\textbf{else}~\textit{or}(\textit{h2}(e))$
     \STATE $x \leftarrow \textit{or}( \textit{pv}( h ) )$
     \STATE $y \leftarrow \textit{or}( \textit{pv}( \textsc{Mate}( h ) ) )$
     \STATE $h1 \leftarrow \textsc{Mate}(\textit{nx}( h ))$ 
     \COMMENT{$h1$ starts at $x$ and ends at $v$}
     \STATE $h2 \leftarrow \textit{pr}( \textsc{Mate}( h ))$
     \COMMENT{$h2$ starts at $v$ and ends at $y$}
     \STATE $e1 \leftarrow \textit{eg}( h1 )$ 
     \COMMENT{$e1$ points to edge $[v,x]$}
     \STATE $e2 \leftarrow \textit{eg}( h2 )$
     \COMMENT{$e2$ points to edge $[v,y]$}
     \STATE $h3 \leftarrow \textit{nx}( h1 )$
     \WHILE{$h3 \neq h2$} 
            \STATE $\textit{or}( h3 ) \leftarrow u$
            \STATE{insert $h3$ into \textit{temp}}
            \STATE $h3 \leftarrow \textit{nx}( \textsc{Mate}( h3 ))$ 
     \ENDWHILE
     \COMMENT{replace $v$ by $u$} 
     \STATE $f \leftarrow \textit{eg}( \textit{pv}( h ) )$ 
     \COMMENT{$f$ is a pointer to edge $[ u , x ]$ in $K$} 
     \STATE $g \leftarrow \textit{eg}( \textit{nx}( \textsc{Mate}( h ) ) )$ 
     \COMMENT{$g$ is a pointer to edge $[ u , y ]$ in $K$} 
     \STATE $\textit{h1}(f) \leftarrow \textsc{Mate}( \textit{pv}( h ) )$
     \STATE $\textit{h2}(f) \leftarrow h1$
     \COMMENT{the half-edge starting at  $u$ and ending at $x$ is now a mate of $h1$}
     \STATE $\textit{h1}(g) \leftarrow \textsc{Mate}( \textit{nx}( \textsc{Mate}( h ) ) )$
     \STATE $\textit{h2}(g) \leftarrow h2$
     \COMMENT{the half-edge starting at  $y$ and ending at $u$ is now a mate of $h2$}
     \STATE $\textit{he}( u ) \leftarrow h2$
     \COMMENT{makes sure $u$ points to a half-edge in the final triangulation} 
     \STATE{remove edges $e$, $e1$, and $e2$, and triangles $\textit{fc}( \textit{h1}(e))$ and $\textit{fc}(\textit{h2}(e))$}
\end{algorithmic}
\end{algorithm}

\subsection{Genus-$0$ surfaces}\label{sec:genus0}
In this  section, we make  some observations about our  algorithm with
regard to triangulations of surfaces  of genus $0$, as it takes linear
time   in  $n_f$   to  produce   an  irreducible   triangulation  from
$\T$. Recall that $n_f$ is the number of triangles in $\T$.

We start by noticing that if  list $\lte$ is nonempty when the loop in
lines 21-34  of Algorithm~\ref{alg:contractions} ends  (i.e., when the
processing of vertex $u$ ends), then every  edge $[ u , v ]$ in $\lte$
is   a   non-contractible   edge   in   the   triangulation   at   the
time. Otherwise, the $c$  attribute of $v$ would be zero and  $[ u , v
]$ would have been  contracted.  From Lemma~\ref{lem:trapped}, we know
that $[ u , v]$ can no  longer be contracted, as vertex $u$ is trapped
after  being processed.  Nevertheless,  it is  still possible  that an
edge, $[  w ,  v ]$, in  the link  of $u$ is  contracted after  $u$ is
processed,  causing $[  u ,  v  ]$ to  be removed  from the  resulting
triangulation, as  shown in Figure \ref{fig:trapped}.   Of course, the
triangulation immediately before the contraction cannot be (isomorphic
to) $\T_4$.

\begin{figure}[htb!]
\begin{center}
\includegraphics{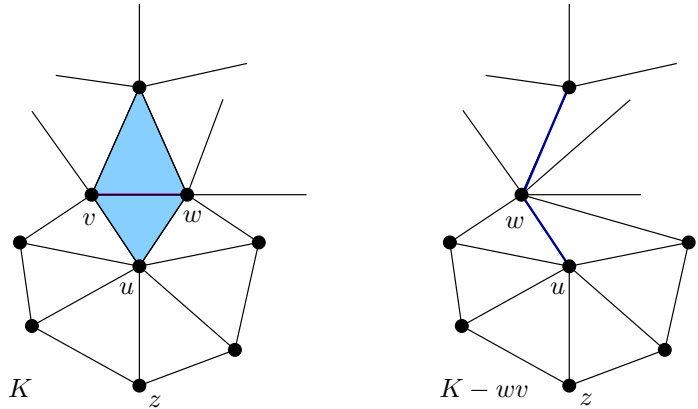}
\caption{\label{fig:trapped} Vertex $u$ is trapped, but its link may
 be modified by the contraction of a link edge.}
\end{center}
\end{figure}

Since $u$ is trapped before the contraction  of $[ v , w ]$, edge $[ u
, v]$ must  belong to a critical cycle, say  $C$, in the triangulation
$K$ to which  the contraction is applied. Similarly, edge $[  u , w ]$
must also belong  to a critical cycle, say  $C^{\prime}$, in $K$. Both
$C$ and $C^{\prime}$  can have at most one edge in  common. If they do
have an edge in common, then the contraction of $[ v , w ]$ identifies
$C$  and  $C^{\prime}$  in  the  resulting triangulation,  $K  -  wv$.
Otherwise, $C$ gives rise to  another critical cycle containing $[ u ,
w ]$ in $K-wv$.  In  either case, no critical cycle containing $[u,w]$
becomes non-critical in  $K-wv$. Thus, if $e$ is  any edge incident on
$u$  when  the  contraction  stage  ends, then  every  critical  cycle
containing $e$  immediately after $u$  is processed belongs to  or has
been merged into a critical cycle in triangulation $\T^{\prime}$.

When $\s$  is a  genus-$0$ surface, the  contraction stage  produces a
triangulation  $\T^{\prime}$   isomorphic  to  $\T_4$.    No  edge  of
$\T^{\prime}$ is contractible, of course,  but none of them can belong
to  a critical cycle  either.  So,  from our  previous remark,  we can
conclude that if $u$ is the first vertex removed from $Q$ in line 5 of
Algorithm~\ref{alg:contractions},  then list $\lte$  must be  empty by
the  time vertex  $u$ is  processed. Otherwise,  every edge  in $\lte$
would be  part of a critical  cycle in the  current triangulation, say
$L$, at the time. But, since exactly three of those edges must be part
of $\T^{\prime}$, the critical  cycles containing these three edges in
$L$ would also belong to $\T^{\prime}$.  However, this is not possible
as  $\T_4$  has  no  critical  cycles.  Hence,  after  vertex  $u$  is
processed, no edge incident on $u$ is part of a critical cycle in $L$.
Since $u$ is trapped, all  those edges must be non-contractible. Thus,
$L$ must  be isomorphic  to $\T_4$, and  hence {\em  all} contractions
occur during the processing of the first vertex $u$ removed from $Q$.

\subsection{Complexity}\label{sec:comp}
This section analyzes the time and space complexities of the algorithm
described in the previous sections. A key feature of this algorithm is
the fact  that it  tests an  edge against the  link condition  at most
once. If an  edge is ever tested against the  link condition and found
to be non-contractible, the edge is stored in an auxiliary list (i.e.,
$\lte$) and a critical cycle counter  is assigned to the edge by the
algorithm to keep  track of the number of critical cycles  containing
the edge.

It turns out that maintaining the critical cycle counters of all edges
in $\lte$  is cheaper  than repeatedly testing  them against  the link
condition. In  particular, the time to update  critical cycle counters
is constant  in lines 5-24  of Algorithm~\ref{alg:processdegree3}, can
be charged  to the time  spent with the  link condition test  in lines
2-11 of  Algorithm~\ref{alg:processdegreegt3}, and is  in $\sum_{z \in
  \mathcal{J}_u^v}   \mathcal{O}(  \rho_z   )$  in   lines   25-41  of
Algorithm~\ref{alg:contract}, where  $\mathcal{J}_u^v$ denotes the set
of all  vertices $z$ in  $\T$ such that  $z \in \T^{\prime}$,  $z$ has
been processed  before $u$,  $z$ becomes a  neighbor of $u$  after the
contraction of edge $[ u , v  ]$, and $\rho_z$ is the degree of $z$ in
the triangulation resulting from the  contraction.  As we see later in
the proof of  Theorem~\ref{thm:complexity}, if $\mathcal{C}_u$ denotes
the set  of all vertices  $v$ in $\T$  such that edge $[  u , v  ]$ is
contracted during the processing of $u$, then 
\[
\sum_{u   \in  \T^{\prime}}   \sum_{v  \in   \mathcal{C}_u}   \left  (
  \mathcal{O}(  d_v  )  +  \sum_{z \in  \mathcal{J}_u^v}  \mathcal{O}(
  \rho_z ) \right )
\]
is an  upper bound for  the total time  to test all  edges $[ u  , v]$
against  the link  condition plus  the time  to update  critical cycle
counters in lines 25-41  of Algorithm~\ref{alg:contract}, for every $u
\in \T^{\prime}$ and every  $v \in \mathcal{C}_u$, where $\T^{\prime}$
is   the  irreducible   triangulation  produced   by   the  algorithm.
Furthermore,  the above bound  can be  written as  $\mathcal{O}(n_v) +
\mathcal{O}( g^2 )$  if the genus $g$ of the surface  on which $\T$ is
defined is positive.

\begin{thm}\label{thm:complexity}
  Given a triangulation $\T$ of a surface $\s$, our algorithm computes
  an irreducible triangulation  $\T^{\prime}$ of $\s$ in $\mathcal{O}(
  g^2 +  g \cdot n_f  )$ time  if the genus  $g$ of $\s$  is positive,
  where $n_f$ is the number of  faces of $\T$.  %% Otherwise, 
  If $g=0$, the time to
  compute $\T^{\prime}$ is  linear in $n_f$. In both  cases, the space
  required by the algorithm is linear in $n_f$.
\end{thm}

\begin{proof}
  Let $n_v$ and $n_e$ be the number of vertices and edges of the input
  triangulation  $\T$.   The  initialization  of the  algorithm  (see
  Algorithm~\ref{alg:continit})      takes     $\mathcal{O}(     n_v)$
  time. Indeed,  each iteration  of the outer  for loop in  lines 2-13
  takes  $\Theta( d_u  )$ time  steps, where  $d_u$ is  the  degree of
  vertex $u$ in  $\T$, as line 3 and lines  7-12 require constant time
  each, and the inner for  loop in lines 4-6 takes $\Theta(d_u)$ time.
  Since
\[
\sum_{u \in  \T} d_u = 2 \cdot n_e \, ,
\]
the total time taken  by the outer for loop in lines  2-13 is given by
$\sum_{u \in {\T}} \Theta( d_u ) \in \Theta( n_e )$. So,
\begin{equation}\label{eq:tcomp1}
\Theta( n_e ) + \sum_{u \in \T} t_u
\end{equation}
is the total time complexity of the algorithm, where $t_u$ is the time
taken to process  vertex $u$ in the outer while loop  in lines 4-37 of
Algorithm~\ref{alg:contractions}.  If  $u$  does  not  belong  to  the
irreducible triangulation,  $\T^{\prime}$, produced by  the algorithm,
then $t_u$ is in $\Theta( 1 )$, as $p( u )$ is {\em false} immediately
after    $u$    is    removed    from    $Q$   in    line    $5$    of
Algorithm~\ref{alg:contractions}.  Consequently, we  can  re-write the
expression in Eq.~(\ref{eq:tcomp1}) as follows:
\begin{align}\label{eq:tcomp2}
  \Theta( n_e ) + \left ( \sum_{u \in \T, u \not\in \T^{\prime}} t_u
  \right ) + \left ( \sum_{u \in \T^{\prime}} t_u \right ) &= \Theta(
  n_e ) + \Theta( n_v - n_v^{\prime} ) + \left ( \sum_{u \in
      \T^{\prime}} t_u \right ) \\ \nonumber
&= \Theta( n_e ) + \mathcal{O}( n_v ) + \left ( \sum_{u \in \T^{\prime}} t_u \right ) \, ,
\end{align}
where $n_v^{\prime}$  is the number of vertices  in $\T^{\prime}$, and
the meaning of `$=$' is that  \textit{every function of the set on the
  left of `$=$' is also a function of the set on the right of `$=$'}.

We now restrict our attention to the time, $t_u$, to process vertex $u
\in  \T^{\prime}$, which  is the  time to  execute the  lines  7-35 of
Algorithm   \ref{alg:contractions}.   Let   $u$  be   any   vertex  of
$\T^{\prime}$. After $u$ is removed  from $Q$ in line $5$ of Algorithm
\ref{alg:contractions}, the for loop  in lines 8-19 is executed.  This
loop iterates exactly $\rho_u$ times,  where $\rho_u$ is the degree of
$u$ in the current triangulation,  $K$, i.e., the triangulation at the
moment that $u$ is removed from $Q$.  Since $u$ has not been processed
yet, we must have that $\rho_u  \le d_u$, where $d_u$ is the degree of
$u$ in  $\T$.  This  is because  the degree of  $u$ can  only decrease
or remain the same 
before  $u$  is  processed.   This  is  also the  case  after  $u$  is
processed. So, the total time spent  within the for loop in lines 8-19
of Algorithm~\ref{alg:contractions}  is in $\mathcal{O}(  d_u )$.  The
repeat-until loop  in lines 21-34  of Algorithm~\ref{alg:contractions}
is executed next,  and the time taken by this  loop is proportional to
the  time spent to  process all  edges removed  from lists  $\lue$ and
$\lte$. Thus, time $t_u$ can be bounded from above by
\begin{equation}\label{eq:tu}
\mathcal{O}(   d_u  ) + q_u \, ,
\end{equation}
where $q_u$  denotes the time to  process all edges  ever removed from
lists $\lue$ and $\lte$.

Let $\mathcal{C}_u$ be the subset of vertices of $\T$ such that $v \in
\mathcal{C}_u$ if  and only if edge $[  u , v ]$  is contracted during
the processing of  $u$. Let $\mathcal{N}_u$ be the  subset of vertices
of $\T$ such that  $v \in \mathcal{N}_u$ if and only if  edge $[ u , v
]$ belongs to $\T_u$, where $\T_u$ is the triangulation resulting from
the  processing of  $u$.  The  set $\mathcal{C}_u  \cup \mathcal{N}_u$
consists of all vertices of  $\T$ that are adjacent to $u$ immediately
before  $u$  is  processed  or  become  adjacent  to  $u$  during  the
processing of $u$.  Note that  if $v$ is in $\mathcal{N}_u$, then edge
$[ u , v ]$ is non-contractible,  as $u$ is trapped in $\T_u$ and thus
no  edge  incident  on  $u$  can  become  contractible  after  $u$  is
processed.  However, recall from  Section~\ref{sec:genus0} that $[ u ,
v   ]$  may   still   be  removed   from   the  final   triangulation,
$\T^{\prime}$. This  is the case whenever  an edge in the  link of $u$
and incident on $v$ is contracted, identifying $v$ with another vertex
in the  link of $u$  (see Figure~\ref{fig:trapped}).  So, a  vertex in
$\mathcal{N}_u$ is  not necessarily in $\T^{\prime}$.   Note also that
$\mathcal{C}_u  \cap \mathcal{N}_u  =  \emptyset$, as  each vertex  in
$\mathcal{C}_u$ is  eliminated during the processing of  $u$ and hence
cannot belong to $\T_u$.

To find  an upper bound  for $q_u$, we  distinguish two cases:  $v \in
\mathcal{C}_u$ and  $v \in  \mathcal{N}_u$.  If $v  \in \mathcal{C}_u$
then  edge  $[  u ,  v  ]$  is  contracted  during the  processing  of
$u$. Otherwise, we  know that $v \in \mathcal{N}_u$ and edge  $[ u , v
]$ is not contracted during the processing of $u$ (i.e., it is an edge
in $\T_u$).   Let $\mathcal{A}_u = \{  v \in \mathcal{N}_u  \mid v \in
\T^{\prime}  ~\text{and   $v$  was  processed  before   $u$}  \}$  and
$\mathcal{B}_u =  \mathcal{N}_u -  \mathcal{A}_u$. In what  follows we
show that
\begin{itemize}
\item[(a)]  For every  vertex $v  \in \mathcal{C}_u$,  the  time required  to
  process edge  $[ u , v  ]$ is in
\[
\mathcal{O}( d_v ) +  \sum_{z \in \mathcal{J}_u^v} \mathcal{O}( \rho_z
) \, ,
\]
where $\mathcal{J}_u^v$  denotes the set  of all vertices $z$  in $\T$
such that $z \in \T^{\prime}$,  $z$ has been processed before $u$, $z$
becomes a neighbor  of $u$ after the  contraction of $[ u ,  v ]$, and
$\rho_z$ is the degree of  $z$ in the triangulation resulting from the
contraction.

\item[(b)] For  every vertex $v \in \mathcal{A}_u$,  the time required
  to process edge $[ u , v ]$ is constant.

\item[(c)] For  every vertex $v \in \mathcal{B}_u$,  the time required
  to process edge $[ u , v  ]$ is in $\mathcal{O}( d_v )$.
\end{itemize}

Let  $v$ be  any vertex  in $\mathcal{C}_u$.   From the  definition of
$\mathcal{C}_u$, we  know that $[  u , v  ]$ is contracted  during the
processing of  $u$. In addition,  this contraction occurs  during the
execution of  either (i) line  26, (ii) line  28, or (iii) line  32 of
Algorithm \ref{alg:contractions}.

If (i)  holds, then $v$ is  a degree-$3$ vertex  in the triangulation,
$K$,  immediately  before  the  contraction.   Furthermore,  the  time
required to process $[  u , v ]$ is proportional to  the time spent by
the       execution      of       procedure      \textsc{Contract}$()$
(Algorithm~\ref{alg:contract})  on $[  u ,  v  ]$ and  $K$. Since  the
degree $\rho_v$ of $v$  in $K$ is $3$, the for loop  in lines 25-43 of
Algorithm~\ref{alg:contract}  is not executed,  as the  temporary list
\textit{temp}         returned         by        \textsc{Collapse}$()$
(Algorithm~\ref{alg:collapse}) in  line 15 is empty. The  for loop in
lines 4-11 takes $\Theta( \rho_v )$ time, and so does the execution of
Algorithm~\ref{alg:collapse}).      The     remaining     lines     of
Algorithm~\ref{alg:contract}  take constant time  each.  So,  the time
required to process edge $[ u , v ]$ in case (i) is in $\Theta( \rho_v
) = \Theta( 1 )$.

If (ii) holds, then the degree $\rho_v$ of $v$ in the triangulation
$K$ immediately before the contraction of $[ u , v ]$ is greater than
$3$.  Line 28 of Algorithm~\ref{alg:continit} invokes the procedure in
Algorithm~\ref{alg:processdegreegt3}.  The  for loop in  lines 2-12 of
Algorithm~\ref{alg:processdegreegt3} iterates $\Theta( \rho_v )$ times
to test  $[ u  , v  ]$ against the  link condition.   Since $v$  is in
$\lue$, $v$ is  still in $Q$.  Thus, $\rho_v \le  d_v$, where $d_v$ is
the degree of $v$ in $\T$. Thus,  the time spent by the for loop is in
$\mathcal{O}(  d_v  )$.  Since  $[   u  ,  v  ]$  was  contracted  (by
assumption),   line  14  of   Algorithm~\ref{alg:processdegreegt3}  is
executed and Algorithm~\ref{alg:contract} is invoked to contract $[ u , v ]$.

Lines 1-24 of  Algorithm~\ref{alg:contract} execute in $\Theta( \rho_v
)$  time, including  the time  for executing  \textsc{Collapse}$()$ in
line  15.   The  for  loop  in lines  25-43  of  \textsc{Contract}$()$
iterates $\rho_v - 3$ times, which is the length of list \textit{temp}
(i.e.,  the number of  neighbors of  $v$ that  become adjacent  to $u$
after  the contraction  of  $[ u  ,  v ]$).   Lines  26-32 execute  in
constant time  each, while  the total time  required to  execute lines
33-41 is in $\Theta( \rho_z )$, where $z$ is a vertex in $\textit{lk}(
v , K )$ whose degree in $K$ is $\rho_z$.  So, the total time required
by Algorithm~\ref{alg:contract}  on input  $[ u, v  ]$ and $K$  can be
bounded above by
\[
\mathcal{O}( d_v ) + \sum_{z  \in \mathcal{J}_u^v} \Theta( \rho_z ) \, .
\]

Note  that  $\rho_z \ge  d_z^{\prime}$,  where  $d_z^{\prime}$ is  the
degree of $z$  is $\T^{\prime}$, as some edges in the  link of $z$ may
still be contracted before  the final triangulation $\T^{\prime}$ is
obtained.  In any case, the total time required to process edge $[ u ,
v ]$ in case (ii) is bounded above by
\[
\mathcal{O}( d_v ) + \mathcal{O}( d_v ) + \sum_{z \in \mathcal{J}_u^v}
\Theta( \rho_z )  = \mathcal{O}( d_v ) +  \sum_{z \in \mathcal{J}_u^v}
\Theta( \rho_z ) \, .
\]

If (iii)  holds, then edge  $[ u  , v ]$  was tested against  the link
condition  before, and  then inserted  into $\lte$  after  failing the
test. Furthermore, $[  u , v]$ is contracted either in  line 5 or line
13 of Algorithm~\ref{alg:processlistS}. From our discussion about case
(ii), we know that  the cost for testing $[ u ,  v ]$ against the link
condition is in $\mathcal{O}( d_v )$, where $d_v$ is the degree of $v$
in $\T$.   In turn,  the cost  for updating the  $c$ attribute  of any
vertex $z$ such that $[u , z ]$ is an edge in $\lte$ can be charged to
the cost of the contraction or  link condition test of another edge in
$\lue$ or $\lte$, as remarked below:

\begin{rmk}\label{rmk:costcc}
  If $[u , z ]$ is in $\lte$, then the value of $c(z)$ can be updated
  by either line 6 of Algorithm~\ref{alg:processdegreegt3}, line 30 of
  Algorithm~\ref{alg:contract}, or lines 6-7, 15, and 20 of
  Algorithm~\ref{alg:processdegree3}.  However, these lines are always
  executed to contract or test an edge.
\end{rmk}

If    $[   u    ,    v   ]$    is    contracted   in    line   5    of
Algorithm~\ref{alg:processlistS}, then  our discussion about  case (i)
tells us that the cost for contracting $[ u , v ]$ is constant, as $v$
has degree $3$  at the time. Consequently, the  total time required to
process $[  u , v ]$ belongs  to $\mathcal{O}( d_v )$,  where $d_v$ is
the  degree of $v$  in $\T$,  which is  equal to  or greater  than the
degree of  $v$ at the  time $[ u  , v ]$  was tested against  the link
condition (immediately before it is  inserted into list $\lte$). If $[
u , v ]$ is contracted in line 13 of Algorithm~\ref{alg:processlistS},
then the degree of $v$ immediately  before the contraction of $[ u , v
]$ is greater  than $3$. So, from our discussion  about case (ii), the
total time required to process $[ u , v ]$ is in
\begin{equation}\label{eq:caseabound}
\mathcal{O}( d_v ) + \sum_{z  \in \mathcal{J}_u^v} \Theta( \rho_z ) \, .
\end{equation}

From   cases   (i),   (ii),   and   (iii),  we   can   conclude   that
Eq.~\ref{eq:caseabound}  is an upper  bound for  the time  required to
process every  edge $[ u ,  v ]$, with $v  \in \mathcal{C}_u$, proving
claim (a).

Now,  let  $v$ be  a  vertex  in  $\mathcal{N}_u$.  By  definition  of
$\mathcal{N}_u$, we  know that $[ u ,  v ]$ is in  $\T_u$, which means
that $[ u , v ]$ is non-contractible. If $p( v )$ is {\em true} by the
time $u$ is removed from $Q$, then $v \in \mathcal{A}_u$ and edge $[ u
,   v   ]$   is  not   inserted   into   $\lue$   (see  line   12   of
Algorithm~\ref{alg:contractions}). Thus,  the time  to process $[  u ,
v]$ is constant, which proves claim (b). If $p( v )$ is {\em false} by
the time $u$ is removed from  $Q$, then $v \in \mathcal{B}_u$ and edge
$[  u ,  v ]$  is inserted  into  $\lue$. Since  $[ u  , v  ]$ is  not
contracted during the processing of $u$,  the algorithm found $[ u , v
]$ to be non-contractible immediately after  removing $[ u , v ]$ from
$\lue$. So, either $[  u , v ]$ failed the link  condition test or the
current triangulation at the time  was (isomorphic to) $\T_4$. If $[ u
, v ]$ is tested against the link condition, then the time required to
process $[  u , v ]$  is in $\mathcal{O}(  d_v )$, where $d_v$  is the
degree of $v$  in $\T$.  While $[ u  , v]$ is in $\lte$,  the cost for
updating the $c$  attribute of $[ u ,  v ]$ is charged to  the cost of
the contraction  or link condition test  of another edge  in $\lue$ or
$\lte$ (see  Remark~\ref{rmk:costcc}). If  $[ u ,  v ]$ is  not tested
against the link condition, then the time required to process $[ u , v
]$ is constant. So, claim (c) holds.

From our discussion above, we get
\begin{align}\label{eq:tuandqu}
  t_u \in \mathcal{O}(   d_u  ) + q_u &= \mathcal{O}(   d_u  ) \\
  \nonumber &+ \sum_{v \in \mathcal{C}_u} \left ( \mathcal{O}( d_v ) +
    \sum_{z \in  \mathcal{J}_u^v} \Theta( \rho_z ) \right  ) + \sum_{v
    \in  \mathcal{A}_u}  \Theta(  1  ) +  \sum_{v  \in  \mathcal{B}_u}
  \mathcal{O}( d_v ) \, ,
\end{align}
and thus
\begin{align}\label{eq:tubig}
  \sum_{u   \in  \T^{\prime}}   t_u  \in   \sum_{u   \in  \T^{\prime}}
  \mathcal{O}(  d_u   )  &+   \sum_{u  \in  \T^{\prime}}   \sum_{v  \in
    \mathcal{C}_u}  \left   (  \mathcal{O}(   d_v  )  +   \sum_{z  \in
      \mathcal{J}_u^v}  \Theta(  \rho_z )  \right  )  \\ \nonumber  &+
  \sum_{u \in \T^{\prime}} \left  ( \sum_{v \in \mathcal{A}_u} \Theta(
    1 ) + \sum_{v \in \mathcal{B}_u} \mathcal{O}( d_v ) \right ) \, .
\end{align}

Equation~(\ref{eq:tubig})  can be  rewritten to  get rid  of $\rho_z$,
$\mathcal{A}_u$, and $\mathcal{B}_u$.  Indeed,  we know that if $z$ is
a vertex in $\mathcal{J}_u^v$, then  the degree $\rho_z$ of $z$ in the
triangulation $\T_u$ may be  greater than $d_z^{\prime}$, which is the
degree of  $z$ in $\T^{\prime}$.  However, $\rho_z  - d_z^{\prime}$ is
equal to the  number of edges of the link of  $z$ that were contracted
after $\T_u$  was obtained.  So, we  can charge the  cost of exploring
$\rho_z    -     d_z^{\prime}$    edges    in     lines    34-41    of
Algorithm~\ref{alg:contract}  to  the  contraction  of the  $\rho_z  -
d_z^{\prime}$ edges of the link of $z$. More specifically, if $[ w , y
]$  is an  edge of  the link  of $z$  that got  contracted  during the
processing of $w$, which occurs after processing $u$, then the cost of
exploring $[ z , y  ]$ in lines 34-41 of Algorithm~\ref{alg:contract},
during the processing of $z$, can be absorbed by $\mathcal{O}( d_v
)$ in the term
\[
\sum_{v \in \mathcal{C}_w} \left ( \mathcal{O}( d_v ) + \sum_{z \in  \mathcal{J}_w^v} \Theta(  \rho_z  ) \right  ) 
\]
of
\begin{equation}\label{eq:qw}
q_w \in \sum_{v \in \mathcal{C}_w}  \left ( \mathcal{O}( d_v ) + \sum_{z
    \in  \mathcal{J}_w^v} \Theta(  \rho_z  ) \right  )  +   \sum_{v \in \mathcal{A}_w} 
    \Theta( 1 ) + \sum_{v \in \mathcal{B}_w} 
    \mathcal{O}( d_v ) \, .
\end{equation}
Thus, 
\begin{align}\label{eq:tubig2}
\sum_{u \in \T^{\prime}}  t_u \in \sum_{u \in \T^{\prime}}  \mathcal{O}(   d_u  ) &+ \sum_{u \in \T^{\prime}}  \sum_{v \in \mathcal{C}_u} \left (
    \mathcal{O}( d_v ) + \sum_{z  \in \mathcal{J}_u^v} \Theta( d_z^{\prime} ) 
  \right ) \\ \nonumber &+ \sum_{u \in \T^{\prime}} \left (  \sum_{v \in \mathcal{A}_u} 
    \Theta( 1 ) + \sum_{v \in \mathcal{B}_u} 
    \mathcal{O}( d_v )  \right )  \, .
\end{align}

Note that $|\mathcal{A}_u|  \le d_u^{\prime}$, where $d_u^{\prime}$ is
the degree of vertex $u$ in $\T^{\prime}$.  Then, we get
\[
\sum_{u \in \T^{\prime}}  \sum_{v \in \mathcal{A}_u} 1 \le \sum_{u \in
  \T^{\prime}} d_u^{\prime} = 2 n_e^{\prime} \quad \text{and} \quad \sum_{u \in \T^{\prime}}  \sum_{v \in \mathcal{B}_u} d_v \le n_v^{\prime} \cdot
\left ( \sum_{v
  \in \T} d_v \right ) =  2 n_v^{\prime} n_e \, ,
\]
where  $n_e^{\prime}$ is  the number  of edges  of  $\T^{\prime}$, and
consequently we can write Eq.~\ref{eq:tubig2} as follows:
\begin{equation}\label{eq:tubig3}
\sum_{u \in \T^{\prime}}  t_u \in \sum_{u \in \T^{\prime}}  \mathcal{O}(   d_u  ) + \sum_{u \in \T^{\prime}}  \sum_{v \in \mathcal{C}_u} \left (
    \mathcal{O}( d_v ) + \sum_{z  \in \mathcal{J}_u^v} \Theta( d_z^{\prime} ) 
  \right ) + \mathcal{O}( n_e^{\prime} ) + \mathcal{O}( n_v^{\prime}
  \cdot n_e )  \, .
\end{equation}

Since
\[
\sum_{u \in  \T^{\prime}} d_u \le \sum_{u \in  \T} d_u \, ,
\]
we can conclude that $\sum_{u  \in \T^{\prime}} \mathcal{O}( d_u )$ is
in   $\mathcal{O}(  n_e   )$.   Moreover,   $\mathcal{J}_u^{v_1}  \cap
\mathcal{J}_u^{v_2} = \emptyset$,  for any two  vertices $v_1$  and $v_2$  of $\T$
such that $[ u  , v_1 ]$ and $[ u , v_2  ]$ were contracted during the
processing of  $u$.  Indeed, a vertex $z$  is in $\mathcal{J}_u^{v_1}$
if  and  only  if it  became  adjacent  to  $u$  as  a result  of  the
contraction of $[  u , v_1 ]$.  So, vertex  $z$ cannot become adjacent
to vertex $u$ as a result of the contraction of edge $[ u , v_2 ]$. As
a result, the union set 
\[
\bigcup_{v \in \mathcal{C}_u} \mathcal{J}_u^v 
\]
is a  subset of  the set $V_u$  of all  vertices in $\textit{lk}(  u ,
\T^{\prime} )$. As a result, we get
\[
\sum_{v \in \mathcal{C}_u} \sum_{z \in \mathcal{J}_u^v} d_z^{\prime}
\le \sum_{z \in V_u} d_z^{\prime} \Longrightarrow \sum_{u \in
  \T^{\prime}} \sum_{z \in V_u} d_z^{\prime} \le n_v^{\prime} \cdot
\left ( \sum_{u \in
  \T^{\prime}} d_u^{\prime} \right ) = 2 n_v^{\prime} n_e^{\prime} \, ,
\]
where $d_u^{\prime}$ is the degree of $u$ in $\T^{\prime}$.

For  any two  vertices  $x$ and  $y$  in $\T^{\prime}$,  we know  that
$\mathcal{C}_{x} \cap \mathcal{C}_{y}  = \emptyset$. Also every vertex
$v$ in  $\T$ that is not  in $\T^{\prime}$ belongs to  exactly one set
$\mathcal{C}_u$, for some vertex $u$ in $\T^{\prime}$. So,
\[
\sum_{u \in  \T^{\prime}} \sum_{v \in \mathcal{C}_u} d_v \in \sum_{v \in
  \T, v \not\in \T^{\prime}} d_v \in \mathcal{O}( n_e ) \, ,
\]
which implies that
\begin{equation}\label{eq:tubig4}
\sum_{u \in \T^{\prime}}  t_u \in \mathcal{O}( n_e ) + \mathcal{O}( n_e
) + \mathcal{O}( n_v^{\prime} \cdot n_e^{\prime} ) + \mathcal{O}(
n_e^{\prime} ) + \mathcal{O}( n_v^{\prime} \cdot n_e ) \, .
\end{equation}
So, the total time required for our algorithm to produce $\T^{\prime}$
from $\T$ can be given by
\begin{equation}\label{eq:tcomp3}
  \Theta( n_e ) + \mathcal{O}(n_v) + \sum_{u \in \T^{\prime}} t_u = 
  \Theta( n_v ) +  \mathcal{O}( (n_v^{\prime})^2 ) + \mathcal{O}( n_v^{\prime} \cdot n_v ) \, .
\end{equation}

If $\s$ is a genus-$0$ surface,  then we know that $n_v^{\prime} = 4$,
which  means that  Eq.~(\ref{eq:tcomp3}) becomes  simply $\mathcal{O}(
n_v   )$.   Otherwise,    Theorem~\ref{the:itsize}   tells   us   that
$n_v^{\prime}  \le  26   \cdot  g  -  4$,  which   then  implies  that
Eq.~(\ref{eq:tcomp3}) can be  written in terms of  $g$ and $n_v$ as
\begin{equation}\label{eq:tcomp4}
  \Theta( n_v ) +  \mathcal{O}( g^2 ) + \mathcal{O}( g \cdot n_v ) \, .
\end{equation}
From our assumption that $n_f  \in \Theta( n_v )$, the time complexity
of our algorithm  is in $\mathcal{O}( g^2 + g \cdot  n_f )$ if surface
$\s$ has a positive genus, $g$.  Otherwise, it is in $\mathcal{O}( n_f
)$. As  for the space  complexity of the  algorithm, we note  that the
space required to store the augmented  DCEL is in $\Theta( n_v + n_f +
n_e )$. In turn, lists $\lue$ and $\lte$ require $\Theta( n_e )$ space
each.  Since  $n_e, n_f \in \Theta(  n_v )$, we can  conclude that the
overall space  required by  our algorithm on  input $\T$ is  linear in
$n_f$.
\end{proof}

\section{Experimental results}\label{sec:expresults}
We implemented  the algorithm described  in Section~\ref{sec:algo}, as
well   as   Schipper's   algorithm~\cite{SCH91}  and   a   brute-force
algorithm. The brute-force algorithm  carries out two steps. First, an
array   with  all  edges   of  the   input  triangulation,   $\T$,  is
shuffled. Second, each edge in the array is visited and tested against
the  link  condition.   If  an  edge  passes  the  test,  then  it  is
contracted. Otherwise, it is inserted  in an auxiliary array. Once the
former array is  empty, the edges in the auxiliary  array are moved to
former one, and the second step  is repeated. If no edge is contracted
during an execution  of the second step, then  the algorithm stops, as
no  remaining  edge is  contractible, which implies that  the output
triangulation $\T^{\prime}$ is irreducible.

As  we   pointed  out  in   Section~\ref{sec:rwork},  the  brute-force
algorithm  may  execute  $\Omega(n_f^2)$  link  condition  tests  (see
Figure~\ref{fig:badcase}), where  $n_f$ is the number  of triangles in
$\T$. In what  follows, we describe an experiment  in which we compare
the  implementations of  the three  aforementioned  algorithms against
triangulations typically  found in  graphics applications, as  well as
triangulations devised to provide  us with some insights regarding the
behavior of our algorithm and the one by Schipper~\cite{SCH91}.

\subsection{Experimental setup}\label{sec:setup}
All  algorithms were  implemented  in \texttt{C++}  and compiled  with
\texttt{clang  503.0.40} using  the \texttt{-O3}  option.  We  ran the
experiments on  an iMac running OSX  10.9.4 at 3.2 GHz  (Intel Core i3
--- 1 processor and 2 cores),  with 256KB of level-one data cache, 4MB
of level-two cache, and 8GB  of RAM.  The implementations are based on
the  same  data  structure   for  surface  triangulations  (i.e.,  the
augmented DCEL  described in Section~\ref{sec:dcel}), and  they do not
depend on any third-party libraries\footnote{\texttt{http://www.mat.ufrn.br/$\sim$mfsiqueira/Marcelo\_Siqueiras\_Web\_Spot/Software.html}.}.

Time  measurements  refer  to  the  time to  compute  the  irreducible
triangulations only  (i.e., we did not  take into account  the time to
read  in   the  triangulations   from  a  file   and  create   a  DCEL
representation in main memory). In particular, each implementation has
a  function,  named   \texttt{run()},  that  computes  an  irreducible
triangulation from  a given pointer  to the augmented  DCEL containing
the input triangulation.  We only  measured the time spent by function
\texttt{run()}.

To time and compare the  implementations, we considered four groups of
triangulations. The first group consists of small genus triangulations
typically found in  graphics papers (see Table~\ref{tab:group1}).  The
second  group consists  of 10  triangulations of  the  same genus-$0$,
brick-shaped     surface      with     $3,\!844$     cavities     (see
Figure~\ref{fig:brick}).  Each triangulation  has a distinct number of
triangles (see Table~\ref{tab:group2}).  The third group consists of 8
triangulations  of a  brick-shaped surface  with $3,\!500$  holes (see
Table~\ref{tab:group3}). This surface was obtained from the one in the
second group by replacing $3,\!500$ cavities with holes.  Finally, the
fourth group  consists of 10 triangulations of  surfaces with varying
genus (see  Table~\ref{tab:group4}).  The triangulations  have about
the same number of triangles, and the surfaces were also obtained from
the ones in the second group by replacing a certain number of cavities
with holes.

Triangulations  in  the  first  group  were  chosen  to  evaluate  the
performance of  the three  algorithms on data  typically used  by mesh
simplification  algorithms~\cite{LRCVWH03}. 

Recall  that the time  complexity of  both our  algorithm and  the one
given  by Schipper~\cite{SCH91}  is  dictated by  two parameters:  the
number of triangles and the genus of the input triangulation. When the
genus  is zero,  the time  upper bound  we derived  for  our algorithm
depends  solely   and  linearly  on  the  number   of  triangles  (see
Section~\ref{sec:genus0}).  Triangulations in  the  second group  were
chosen  to  evaluate  the  performance  of  the  three  algorithms  on
triangulations of genus $0$ surfaces.

\begin{table}[htbp]
\begin{center}
\begin{tabular}{|c|c|c|c|c|} \hline
\textbf{Triangulation} & \textbf{\# Vertices} & \textbf{\# Edges} &  \textbf{\# Triangles} & \textbf{\# Genus}  \\ \hline\hline
Armadillo         &  $171,\!889$                      &    $\hphantom{0,}\!515,\!661$ & $\hphantom{0,\!}342,\!774$ & $\hphantom{00}0$ \\
Botijo               &  $\hphantom{0}20,\!000$   & $\hphantom{0,\!0}60,\!024$  & $\hphantom{0,\!0}40,\!016$ & $\hphantom{00}5$ \\
Casting            &  $\hphantom{00}5,\!096$  & $\hphantom{0,\!0}15,\!336$  & $\hphantom{0,\!0}10,\!224$	& $\hphantom{00}9$ \\
Eros                 &  $197,\!230$                      & $\hphantom{0,}\!591,\!684$  & $\hphantom{0,\!}394,\!456$	& $\hphantom{00}0$ \\
Fertility            &  $\hphantom{0}19,\!994$     & $\hphantom{0,\!0}60,\!000$    & $\hphantom{0,\!0}40,\!000$ & $\hphantom{00}4$ \\
Filigree             &  $\hphantom{0}29,\!129$   & $\hphantom{0,\!0}87,\!771$    & $\hphantom{0,\!0}58,\!514$ & $\hphantom{0}65$ \\
Hand               &   $195,\!557$                     &  $\hphantom{0,\!}586,\!665$    &  $\hphantom{0,\!}391,\!110$ & $\hphantom{00}0$ \\
Happy Buddha	&  $543,\!652$                      & $1,\!631,\!574$                          & $1,\!087,\!716$ & $104$ \\
Iphigenia         &  $351,\!750$	              & $1,\!055,\!268$                          & $\hphantom{0,\!}703,\!512$ & $\hphantom{00}4$ \\
Socket             &  $\hphantom{000,\!}836$   &  $\hphantom{0,\!00}2,\!544$	& $\hphantom{0,\!00}1,\!696$	&  $\hphantom{00}7$ \\ \hline
\end{tabular}
\caption{\label{tab:group1}     Euler    characteristics     of    the
  triangulations in the first group.}
\end{center}
\end{table}

\begin{figure}[htbp]
\begin{center}
\includegraphics[width=2.5in]{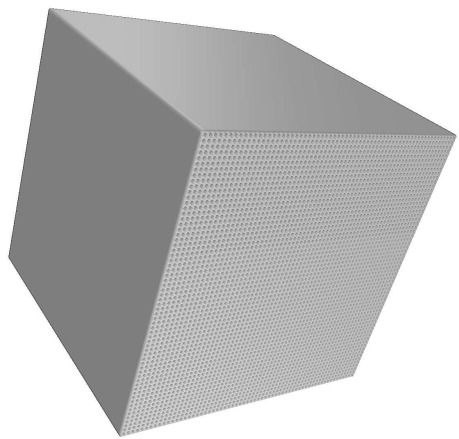}
\caption{\label{fig:brick} A brick-shaped surface with $3,\!844$  cavities.}
\end{center}
\end{figure}

Triangulations in the third and  fourth groups were chosen to evaluate
the influence of both parameters (i.e., genus and number of triangles)
separately.   Triangulations in the  third group  have the  same genus
(i.e, $3,\!500$), but their  numbers of triangles vary, which allowed
us  to evaluate  the influence  of the  number of  triangles  over the
performances  of our  algorithm  and Schipper's  algorithm.  In  turn,
triangulations  in the  fourth group  have  about the  same number  of
triangles, but  their genuses vary,  which allowed us to  evaluate the
influence  of the  genus over  the performances  of our  algorithm and
Schipper's algorithm.

\begin{table}[htb!]
\begin{center}
\begin{tabular}{|c|c|c|c|c|} \hline
\textbf{Triangulation} & \textbf{\# Vertices} & \textbf{\# Edges}  &  \textbf{\# Triangles} \\ \hline\hline
B0  &  $2,\!097,\!150$                      & $6,\!291,\!444$                     &  $4,\!194,\!296$ \\
B1  &  $1,\!097,\!150$                      & $3,\!291,\!444$                     &  $2,\!194,\!296$ \\
B2  &  $\hphantom{0,\!}597,\!150$  & $1,\!791,\!444$                     &  $1,\!194,\!296$ \\
B3  &  $\hphantom{0,\!}297,\!150$  & $\hphantom{0,\!}891,\!444$    &  $\hphantom{0,\!}594,\!296$ \\
B4  &  $\hphantom{0,\!}147,\!150$  & $\hphantom{0,\!}441,\!444$    &  $\hphantom{0,\!}294,\!296$ \\
B5  &  $\hphantom{0,\!0}72,\!150$  & $\hphantom{0,\!}216,\!444$    &  $\hphantom{0,\!}144,\!296$ \\
B6  &  $\hphantom{0,\!0}34,\!650$  & $\hphantom{0,\!}103,\!944$    &  $\hphantom{0,\!0}69,\!296$ \\
B7  &  $\hphantom{0,\!0}15,\!900$  & $\hphantom{0,\!0}47,\!694$    &  $\hphantom{0,\!0}31,\!796$ \\
B8  &  $\hphantom{0,\!00}8,\!400$  & $\hphantom{0,\!0}25,\!194$    &  $\hphantom{0,\!0}16,\!796$ \\
B9  &  $\hphantom{0,\!00}4,\!400$  & $\hphantom{0,\!0}13,\!194$    &  $\hphantom{0,\!00}8,\!796$ \\ \hline
\end{tabular}
\caption{\label{tab:group2}     Euler    characteristics     of    the
  triangulations in the second group.}
\end{center}
\end{table}

\begin{table}[htb!]
\begin{center}
\begin{tabular}{|c|c|c|c|c|} \hline
\textbf{Triangulation} & \textbf{\# Vertices} & \textbf{\# Edges}  &  \textbf{\# Triangles} \\ \hline\hline
C0  &  $2,\!104,\!150$                      & $6,\!333,\!444$                     &  $4,\!222,\!296$ \\
C1  &  $1,\!104,\!150$                      & $3,\!333,\!444$                     &  $2,\!222,\!296$ \\
C2  &  $\hphantom{0,\!}604,\!150$  & $1,\!833,\!444$                     &  $1,\!222,\!296$ \\
C3  &  $\hphantom{0,\!}354,\!150$  & $1,\!083,\!444$                     &  $\hphantom{0,\!}722,\!296$ \\
C4  &  $\hphantom{0,\!}179,\!150$  & $\hphantom{0,\!}558,\!444$ &  $\hphantom{0,\!}372,\!296$ \\
C5  &  $\hphantom{0,\!0}91,\!650$  & $\hphantom{0,\!}295,\!444$    &  $\hphantom{0,\!}197,\!296$ \\
C6  &  $\hphantom{0,\!0}41,\!650$  & $\hphantom{0,\!}145,\!944$    &  $\hphantom{0,\!0}97,\!296$ \\
C7  &  $\hphantom{0,\!0}16,\!650$  & $\hphantom{0,\!0}70,\!994$    &  $\hphantom{0,\!0}47,\!796$ \\ \hline
\end{tabular}
\caption{\label{tab:group3}     Euler    characteristics     of    the
  triangulations in the third group (their genus is $3,\!500$).}
\end{center}
\end{table}

\begin{table}[htb!]
\begin{center}
\begin{tabular}{|c|c|c|c|c|} \hline
\textbf{Triangulation} & \textbf{\# Vertices} & \textbf{\# Edges} &  \textbf{\# Triangles} & \textbf{\# Genus}  \\ \hline\hline
D0  &  $2,\!104,\!150$ & $6,\!333,\!444$ & $4,\!222,\!296$ & $3,\!500$ \\
D1  &  $2,\!103,\!150$ & $6,\!327,\!444$ & $4,\!218,\!296$ & $3,\!000$ \\
D2  &  $2,\!102,\!150$ & $6,\!321,\!444$ & $4,\!214,\!296$ & $2,\!500$ \\
D3  &  $2,\!101,\!150$ & $6,\!315,\!444$ & $4,\!210,\!296$ & $2,\!000$ \\
D4  &  $2,\!100,\!150$ & $6,\!309,\!444$ & $4,\!206,\!296$ & $1,\!500$ \\
D5  &  $2,\!099,\!150$ & $6,\!303,\!444$ & $4,\!202,\!296$ & $1,\!000$ \\
D6  &  $2,\!098,\!150$ & $6,\!297,\!444$ & $4,\!198,\!296$ & $\hphantom{0,}500$ \\
D7  &  $2,\!097,\!350$ & $6,\!292,\!644$ & $4,\!195,\!096$ & $\hphantom{0,}100$ \\
D8  &  $2,\!097,\!250$ & $6,\!292,\!044$ & $4,\!194,\!696$ & $\hphantom{0,\!0}50$ \\
D9  &  $2,\!097,\!170$ & $6,\!291,\!564$ & $4,\!194,\!376$ & $\hphantom{0,\!0}10$ \\ \hline
\end{tabular}
\caption{\label{tab:group4}     Euler    characteristics     of    the
  triangulations in the fourth group.}
\end{center}
\end{table}

\subsection{Results}\label{sec:results}
From now  on, we  denote our algorithm,  Schipper's algorithm  and the
brute-force  algorithm by  \textbf{RS},  \textbf{S}, and  \textbf{BF},
respectively. We initially ran \textbf{RS} and \textbf{S} exactly once
on each triangulation of the first group (see Table~\ref{tab:group1}),
while \textbf{BF} was executed  ten times on each triangulation (since
we randomized the input by shuffling the edges of the triangulations).
So, for  \textbf{BF}, we computed  and recorded the  average execution
times  over  the  ten runs  on  each  triangulation.   A plot  of  the
\textit{time}  (in seconds)  taken by  the three  algorithms  on every
input  triangulation versus  the \textit{number  of triangles}  of the
triangulations is shown in Figure \ref{fig:mixtimings}.

\begin{figure}[htb!]
\begin{center}
\includegraphics[width=3.2in]{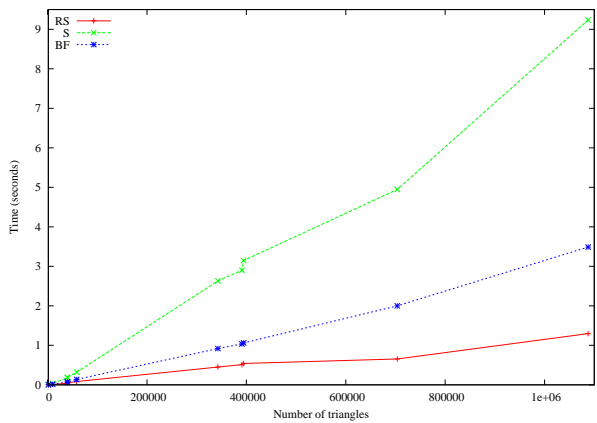}
\caption{\label{fig:mixtimings} Runtimes for the execution of \textbf{RS},
  \textbf{S}, and  \textbf{BF}  on the first group triangulations.}
\end{center}
\end{figure}

As we can see in Figure~\ref{fig:mixtimings}, the larger the number of
triangles,   the  larger   the   ratios  $t_{s}/t_{\textit{rs}}$   and
$t_{\textit{bf}}/t_{\textit{rs}}$,  where  $t_{\textit{rs}}$, $t_{s}$,
and $t_{\textit{bf}}$ are the  times taken by \textbf{RS}, \textbf{S},
and \textbf{BF}, respectively.  In particular, $t_{s}/t_{\textit{rs}}$
and $t_{\textit{bf}}/t_{\textit{rs}}$  are equal to  $2.13$ and $1.14$
for  the triangulation with  the smallest  number of  triangles (i.e.,
\texttt{Socket}), and equal to $7.55$ and $3.05$ for the triangulation
with      the     largest      number     of      triangles     (i.e.,
\texttt{Iphigenia}). Observe  that \textbf{BF} outperforms \textbf{S}.
In particular, $t_{\textit{s}}/t_{\textit{bf}}$ is always greater than
$1.9$  and  its  largest  value   is  $2.97$,  which  is  attained  on
triangulation \texttt{Eros}.

We repeated the experiment for  the triangulations in the second group
(see  Table~\ref{tab:group2}). All triangulations  in this  group have
genus-$0$. In particular, for every $i=1,\ldots,9$, triangulation B$i$
was  obtained from  triangulation B$i-1$  by a  simplification process
that approximately halved  the number of triangles of  B$i-1$.  A plot
of the  \textit{time} (in  seconds) taken by  \textbf{RS}, \textbf{S},
and  \textbf{BF}  on  triangulations   B0-B9  as  a  function  of  the
\textit{number of triangles} of  the triangulations is shown in Figure
\ref{fig:bricktimings}.

Note  that  \textbf{RS} outperforms  \textbf{S}  and \textbf{BF},  and
\textbf{BF}  outperforms   \textbf{S}.   However,  this   time,  ratio
$t_{s}/t_{\textit{rs}}$  gets  smaller  as  the  number  of  triangles
grows. In particular, $t_{s}/t_{\textit{rs}}$  is equal to $12.28$ for
triangulation  B9 and  equal to  $3.56$ for  triangulation  B0.  Ratio
$t_{\textit{bf}}/t_{\textit{rs}}$ presents  the same behavior,  but it
becomes noticeable  only for triangulations B0-B3, for which  
the numbers of triangles exceed $500,\!000$.

Triangulations  C0-C7 in  Table~\ref{tab:group3} have  a  fixed, large
genus (i.e, $3,\!500$). In addition,  C1-C7 were built as follows: for
every   $i=1,\ldots,7$,   triangulation   C$i$   was   obtained   from
triangulation  C$i-1$ by a  simplification process  that approximately
halved the number  of triangles of C$i-1$.  The  triangulations in the
third group were designed  to compare the performances of \textbf{RS},
\textbf{S},  and \textbf{BF}  on variable-size  triangulations  of the
same  fixed, large  genus surface  (as opposed  to the  same genus-$0$
surface  like we  did  before  for the  triangulations  in the  second
group). A plot of the \textit{time} (in seconds) taken by \textbf{RS},
\textbf{S}, and  \textbf{BF} on triangulations C0-C7 as  a function of
the  \textit{number of triangles}  of the  triangulations is  shown in
Figure \ref{fig:fixedgenustimings}.

\begin{figure}[htb!]
\begin{center}
\includegraphics[width=3.2in]{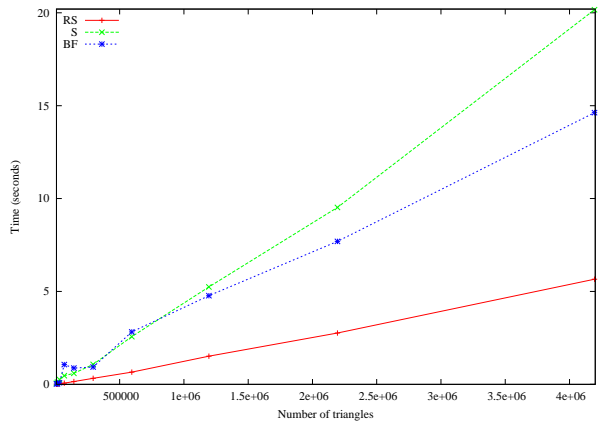}
\caption{\label{fig:bricktimings}  Runtimes for the execution of  \textbf{RS},
  \textbf{S}, and \textbf{BF} on the second group triangulations.}
\end{center}
\end{figure}

\begin{figure}[htb!]
\begin{center}
\includegraphics[width=3.5in]{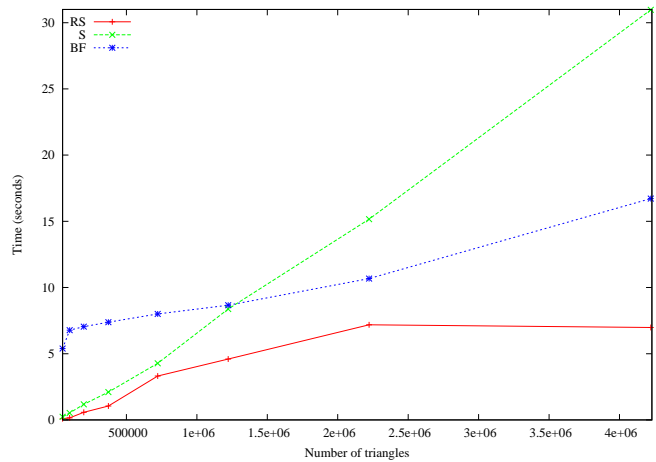}
\caption{\label{fig:fixedgenustimings}  Runtimes  for  the  execution  of
  \textbf{RS},  \textbf{S},   and  \textbf{BF}  on   the  third  group
  triangulations.}
\end{center}
\end{figure}

Once  again \textbf{RS} outperformed  \textbf{S} and  \textbf{BF}, but
\textbf{BF} outperforms \textbf{S}  only for C0 and C1,  which are the
triangulations  with the  largest number  of  triangles.  Furthermore,
contrary to the results obtained from the triangulations in the second
group,  ratio $t_{s}/t_{\textit{rs}}$  gets  larger as  the number  of
triangles    grows.    This    is    also   the    case   for    ratio
$t_{\textit{bf}}/t_{\textit{rs}}$,  but  the   behavior  can  only  be
noticed for  triangulations C0, C1, and  C2, which are  the ones whose
number of triangles is greater than $500,\!000$.

Finally,   we  ran   \textbf{RS},  \textbf{S},   and   \textbf{BF}  on
triangulations D0-D9  of the  fourth group. These  triangulations have
nearly the same number of  triangles, but their genus varies from $10$
to $3,\!500$ (see Table~\ref{tab:group4}). A plot of the \textit{time}
(in  seconds) taken  by  \textbf{RS}, \textbf{S},  and \textbf{BF}  on
triangulations  D0-D9  as a  function  of  the  \textit{genus} of  the
triangulations  is   shown  in  Figure  \ref{fig:varyinggenustimings}.
Observe that  \textbf{RS} outperforms \textbf{S}  and \textbf{BF}, and
\textbf{BF}  outperforms  \textbf{S}  for all  triangulations.   Ratio
$t_{s}/t_{\textit{rs}}$ gets  larger as the genus  grows.  Its maximum
value is $4.7$,  and is attained for triangulation  D0.  Unlike, ratio
$t_{\textit{bf}}/t_{\textit{rs}}$ gets  slightly smaller as  the genus
grows, and it is basically constant and about $2.4$ for triangulations
D0-D3 whose genuses are greater than $1,\!999$.
 
\begin{figure}[htbp]
\begin{center}
\includegraphics[width=3.5in]{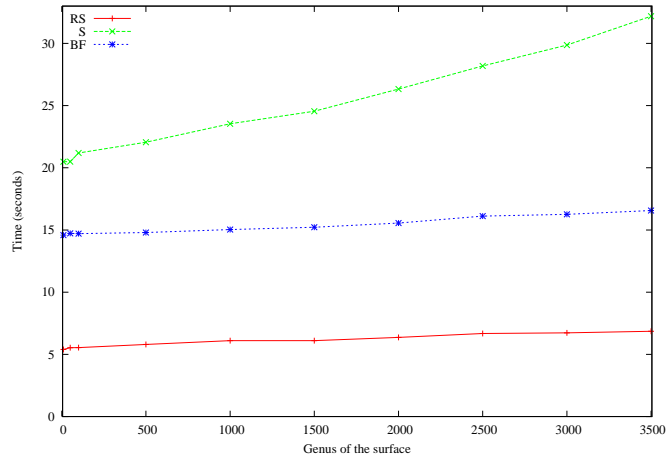}
\caption{\label{fig:varyinggenustimings} Runtimes for the execution of
  \textbf{RS},  \textbf{S},  and   \textbf{BF}  on  the  fourth  group
  triangulations.}
\end{center}
\end{figure}

\subsection{Discussion}\label{sec:discussion}
To properly analyze the  results in Section~\ref{sec:results}, we take
into account  the \textit{number of link condition  tests} carried out
by each algorithm, as well  as the \textit{number of edges tested more
  than once} by  \textbf{S} and \textbf{BF}.  We denote  the number of
link condition tests carried out by \textbf{RS} (resp.  \textbf{S} and
\textbf{BF})  by $\ell_{\textit{rs}}$ (resp.   $\ell_{\textit{s}}$ and
$\ell_{\textit{bf}}$), the  number of edges  tested more than  once by
\textbf{S}   (resp.   \textbf{BF})   and   by  $\epsilon_{s}$   (resp.
$\epsilon_{\textit{bf}}$). In particular,    $\ell_{\textit{bf}}$   and
$\epsilon_{\textit{bf}}$ are  the average values over the  ten runs of
\textbf{BF}.

Although \textbf{RS}  outperforms both \textbf{S}  and \textbf{BF} for
the triangulations in the second group, ratios $t_{s}/t_{\textit{rs}}$
and $t_{\textit{bf}}/t_{\textit{rs}}$ get  smaller as the number $n_f$
of triangles  of the input  triangulations grows.  The main  reason is
that  the probability  that a  randomly chosen  edge from  a genus-$0$
surface triangulation  is contractible increases as  $n_f$ gets larger
(and  the surface  is kept  fixed).  So,  ratios  $\ell_{\textit{s}} /
\ell_{\textit{rs}}$  and   $\ell_{\textit{bf}}  /  \ell_{\textit{rs}}$
decrease  as $n_f$  gets larger,  causing  $t_{s}/t_{\textit{rs}}$ and
$t_{\textit{bf}}/t_{\textit{rs}}$  to decay,  as we  can see  in Tables
\ref{tab:lctgenus0}-\ref{tab:lrgenus0}.     Note   also   that    $\ell_{\textit{bf}}   >
\ell_{\textit{s}}$         and        $\epsilon_{\textit{bf}}        >
\epsilon_{\textit{s}}$,   for  all  triangulations   B0-B9.  Moreover,
$\epsilon_{\textit{s}}$  is very  small  and does  not  scale up  with
$n_f$. Even so, we get  $t_{s} > t_{\textit{bf}}$.

\begin{table}[htbp]
\begin{center}
\begin{tabular}{|c|c|c|c|c|c|} \hline
\textbf{Triangulation} & $\bm{\ell_{\textit{rs}}}$ & $\bm{\ell_{\textit{s}}}$ & $\bm{\ell_{\textit{bf}}}$ & $\bm{\epsilon_{\textit{s}}}$ & $\bm{\epsilon_{\textit{bf}}}$ \\ \hline\hline
B0  &  $2,\!079,\!539$                       & $2,\!097,\!158$                      & $3,\!406,\!091.5$                     & 6  & $19,\!011.7$ \\
B1  &  $1,\!085,\!325$                       & $1,\!097,\!194$                      & $1,\!783,\!418.4$                     & 9  & $10,\!034.9$ \\
B2  &  $\hphantom{0,}\!588,\!335$   & $\hphantom{0,}\!597,\!196$  & $\hphantom{0,}\!968,\!511.9$  & 7  & $\hphantom{0}\!5,\!495.3$ \\
B3  &  $\hphantom{0,}\!289,\!520$   & $\hphantom{0,}\!297,\!174$  & $\hphantom{0,}\!479,\!536.2$  & 8  & $\hphantom{0}\!2,\!863.7$ \\
B4  &  $\hphantom{0,}\!145,\!338$   & $\hphantom{0,}\!147,\!193$  & $\hphantom{0,}\!232,\!984.4$  & 7  & $\hphantom{0}\!1,\!482.1$ \\
B5  &  $\hphantom{0,\!0}68,\!454$   & $\hphantom{0,\!0}72,\!190$  & $\hphantom{0,}\!111,\!673.8$  & 7  & $\hphantom{0,\!0}\!816.4$ \\
B6  &  $\hphantom{0,\!0}29,\!827$   & $\hphantom{0,\!0}34,\!675$  & $\hphantom{0,\!0}50,\!936.6$  & 6  & $\hphantom{0,\!0}\!458.9$ \\
B7  &  $\hphantom{0,\!0}14,\!085$   & $\hphantom{0,\!0}15,\!908$  & $\hphantom{0,\!0}18,\!838.2$  & 6  & $\hphantom{0,\!0}\!292.5$ \\
B8  &  $\hphantom{0,\!00}7,\!979$   & $\hphantom{0,\!00}8,\!408$  & $\hphantom{0,\!00}9,\!941.4$  & 6  & $\hphantom{0,\!0}\!159.4$ \\
B9  &  $\hphantom{0,\!00}4,\!063$   & $\hphantom{0,\!00}4,\!408$  & $\hphantom{0,\!00}5,\!246.4$  & 6  & $\hphantom{0,\!00}\!92.1$ \\\hline
\end{tabular}
\caption{\label{tab:lctgenus0} The number  of link condition tests and
  the  number  of  edges  tested  more than  once  obtained  from  the
  executions  of  \textbf{RS},  \textbf{S},  and  \textbf{BF}  on  the
  triangulations B0-B9 in the second group.}
\end{center}
\end{table}

\begin{table}[htbp]
\begin{center}
\begin{tabular}{|c|c|c|c|c|c|c|} \hline
\textbf{Triangulation}  & $\bm{\ell_{\textit{s}}} /
\bm{\ell_{\textit{rs}}}$ & $\bm{\ell_{\textit{bf}}} /
\bm{\ell_{\textit{rs}}}$ & $\bm{\ell_{\textit{s}}} /
\bm{\ell_{\textit{bf}}}$  & $\bm{t_{\textit{s}}} / \bm{t_{\textit{rs}}}$ & $\bm{t_{\textit{bf}}} / \bm{t_{\textit{rs}}}$ & $\bm{t_{\textit{s}}} / \bm{t_{\textit{bf}}}$ \\ \hline\hline
B0     & $1.01$    & $1.64$   & $0.62$ & $\hphantom{0}3.56$ & $\hphantom{0}2.58$   & $1.38$ \\
B1     & $1.01$    & $1.65$   & $0.62$ & $\hphantom{0}3.45$ & $\hphantom{0}2.79$   & $1.24$  \\
B2     & $1.02$    & $1.65$   & $0.62$ & $\hphantom{0}3.46$ & $\hphantom{0}3.15$   & $1.10$  \\
B3     & $1.03$    & $1.66$   & $0.62$ & $\hphantom{0}3.90$ & $\hphantom{0}4.30$   & $0.91$  \\
B4     & $1.01$    & $1.60$   & $0.63$ & $\hphantom{0}3.32$ & $\hphantom{0}2.88$   & $1.15$  \\
B5     & $1.05$    & $1.63$   & $0.65$ & $\hphantom{0}3.96$ & $\hphantom{0}5.86$   & $0.67$  \\
B6     & $1.16$    & $1.71$   & $0.68$ & $\hphantom{0}7.37$ &                     $16.90$   & $0.44$  \\
B7     & $1.13$    & $1.34$   & $0.84$ &                     $11.51$ & $\hphantom{0}3.76$   & $3.06$  \\
B8     & $1.05$    & $1.25$   & $0.85$ &                     $10.81$ & $\hphantom{0}3.35$   & $3.23$  \\
B9     & $1.08$    & $1.29$   & $0.84$ &                     $12.28$ & $\hphantom{0}4.15$   & $2.96$  \\\hline
\end{tabular}
\caption{\label{tab:lrgenus0}  Ratios  between   the  number  of  link
  condition   tests   performed   by  \textbf{RS},   \textbf{S},   and
  \textbf{BF} on the triangulations B0-B9 in the second group, and their
  corresponding execution time ratios.}
\end{center}
\end{table}

The fact that $\epsilon_{\textit{s}}$ is very small and does not scale
up with  $n_f$ is  due to the  strategy used  by \textbf{S} to  find a
contractible edge:  first, a  vertex of lowest  degree in  the current
triangulation $K$ is chosen, and  then a contractible edge incident on
this vertex  is found.   Since the  genus of the  surface is  $0$, the
lowest degree vertex of $K$ is  very likely to have degree $3$ or $4$.
If $K$ is not (isomorphic to) $\T_4$ already, then every edge incident
on    a   degree-$3$    vertex   is    contractible   in    $K$   (see
Proposition~\ref{prop:contractible1}), and hence most edges are tested
against the  link condition by  \textbf{S} only once.   In particular,
for the cases in which $\epsilon_{\textit{s}} = 6$, only the six edges
of the final irreducible triangulation, which is isomorphic to $\T_4$,
were tested more than once.  \textbf{S} chooses a lowest degree vertex
from $K$  in $\mathcal{O}( \lg m )$  time, where $m$ is  the number of
(loose)  vertices   in  $K$.    Our  experiments  indicate   that  the
$\mathcal{O}( \lg m )$ cost  cancels out the gain obtained by reducing
the values of $\ell_{\textit{s}}$ and $\epsilon_{\textit{s}}$.

For triangulations  C0-C7, which have  a fixed genus of  $3,\!500$ and
variable-size,    the    scenario    regarding    $\ell_{\textit{s}}$,
$\ell_{\textit{bf}}$, $\epsilon_{s}$,  and $\epsilon_{\textit{bf}}$ is
quite the opposite to the  one for the genus-$0$ triangulations, B0-B9
(see Table~  \ref{tab:lctgenus3500}).  The  reason is that  the larger
the genus is  the smaller the probability that  a randomly chosen edge
from  any  of C0-C7  is  contractible.  Furthermore, this  probability
decreases even further as $n_f$ gets smaller.

\begin{table}[htbp]
\begin{center}
\begin{tabular}{|c|c|c|c|c|c|} \hline
\textbf{Triangulation} & $\bm{\ell_{\textit{rs}}}$                 & $\bm{\ell_{\textit{s}}}$               & $\bm{\ell_{\textit{bf}}}$                & $\bm{\epsilon_{\textit{s}}}$   & $\bm{\epsilon_{\textit{bf}}}$ \\ \hline\hline
C0                                 &  $3,\!499,\!517$                       & $3,\!336,\!889$                      & $3,\!595,\!222.5$                     & $346,\!735$                       & $52,\!815.0$ \\
C1                                 &  $1,\!846,\!458$                       & $1,\!780,\!868$                      & $1,\!967,\!691.3$                     & $200,\!636$                       & $49,\!757.9$ \\
C2                                 &  $1,\!018,\!801$                       & $1,\!006,\!051$                      & $1,\!131,\!182.2$                     & $128,\!339$                       & $48,\!063.1$ \\
C3                                 &  $\hphantom{0,}\!605,\!195$   & $\hphantom{0,}\!615,\!112$  & $\hphantom{0,}\!742,\!667.4$  & $\hphantom{0}91,\!493$   & $47,\!180.6$ \\
C4                                 &  $\hphantom{0,}\!318,\!114$   & $\hphantom{0,}\!347,\!941$  & $\hphantom{0,}\!459,\!221.8$  & $\hphantom{0}66,\!975$   & $46,\!630.6$ \\
C5                                 &  $\hphantom{0,}\!171,\!599$   & $\hphantom{0,}\!209,\!767$  & $\hphantom{0,}\!316,\!699.6$  & $\hphantom{0}54,\!353$   & $46,\!300.9$ \\
C6                                 &  $\hphantom{0,\!0}90,\!867$   & $\hphantom{0,}\!135,\!245$  & $\hphantom{0,}\!232,\!041.0$  & $\hphantom{0}47,\!094$   & $46,\!105.7$ \\
C7                                 &  $\hphantom{0,\!0}54,\!594$   & $\hphantom{0,}\!100,\!181$  & $\hphantom{0,}\!188,\!700.6$  & $\hphantom{0}45,\!856$   & $46,\!000.5$ \\  \hline
\end{tabular}
\caption{\label{tab:lctgenus3500}  The number  of link condition tests and
  the  number  of  edges  tested  more than  once  obtained  from  the
  executions  of  \textbf{RS},  \textbf{S},  and  \textbf{BF}  on  the
  triangulations C0-C7 in the third group.}
\end{center}
\end{table}

The  fact  that   $\epsilon_{s}  >  \epsilon_{\textit{bf}}$,  for  all
triangulations in  the third group,  but C7, tells  us that a  few low
degree vertices  are chosen over  and over again by  \textbf{S} before
they become trapped, and several  edges incident on these vertices are
tested  more than  once  against  the link  condition  and failed  the
test. So, the strategy adopted  by \textbf{S} is not so effective when
the  genus  of the  surface  is  large. Moreover,  as  we  can see  in
Tables~\ref{tab:lctgenus3500}   and~\ref{tab:lrgenus3500},   we   have
$\ell_{\textit{bf}}  >   \ell_{\textit{s}}$,  for  all  triangulations
C0-C7,  and $\ell_{\textit{bf}} /  \ell_{\textit{s}}$ gets  smaller as
$n_f$  gets  larger, varying  from  $1.88$  (for  C7) to  $1.08$  (for
C0). The larger values of $\ell_{\textit{bf}} / \ell_{\textit{s}}$ for
C2-C7  explain why  \textbf{S} outperforms  \textbf{BF} on  C2-C7 (see
Figure~\ref{fig:fixedgenustimings}).
 
It is worth mentioning that the number of link condition tests carried
out  by \textbf{RS}  (i.e., $\ell_{\textit{rs}}$)  is larger  than the
number  of  link condition  tests  carried  out  by \textbf{S}  (i.e.,
$\ell_{\textit{s}}$) for triangulations C0,  C1, and C2, which are the
ones  with a  larger  number,  $n_f$, of  triangles.   Even so,  ratio
$t_{\textit{s}}/t_{\textit{rs}}$ gets larger  as $n_f$ gets larger for
C0, C1,  and C2. This fact  indicates that the $\mathcal{O}(  \lg m )$
cost for  choosing a lowest  degree vertex from  the set of  $m$ loose
vertices of the current triangulation cancels out the gain obtained by
\textbf{S}  by executing  a smaller  number of  link  condition tests.
Since  $\epsilon_{\textit{s}}$  is large  (compared  to the  genus-$0$
scenario), the value of $m$ decreases more slowly, which increases the
overall cost of picking lowest degree vertices. Also, the overall cost
of testing a  set of edges in \textbf{RS} is  smaller than the overall
cost of testing  the same set of edges  in \textbf{S} (see discussions
in Section~\ref{sec:testedges} and Section~\ref{sec:rwork}).

\begin{table}[htbp]
\begin{center}
\begin{tabular}{|c|c|c|c|c|c|c|} \hline
\textbf{Triangulation}  & $\bm{\ell_{\textit{s}}} /
\bm{\ell_{\textit{rs}}}$ & $\bm{\ell_{\textit{bf}}} /
\bm{\ell_{\textit{rs}}}$ & $\bm{\ell_{\textit{s}}} /
\bm{\ell_{\textit{bf}}}$  & $\bm{t_{\textit{s}}} / \bm{t_{\textit{rs}}}$ & $\bm{t_{\textit{bf}}} / \bm{t_{\textit{rs}}}$ & $\bm{t_{\textit{s}}} / \bm{t_{\textit{bf}}}$ \\ \hline\hline
C0     & $0.95$    & $1.03$   & $0.93$ & $4.44$ & $\hphantom{0}2.39$   & $1.85$ \\
C1     & $0.96$    & $1.07$   & $0.91$ & $2.11$ & $\hphantom{0}1.49$   & $1.42$  \\
C2     & $0.99$    & $1.11$   & $0.89$ & $1.82$ & $\hphantom{0}1.89$   & $0.97$  \\
C3     & $1.02$    & $1.23$   & $0.83$ & $1.29$ & $\hphantom{0}2.41$   & $0.53$  \\
C4     & $1.09$    & $1.44$   & $0.76$ & $2.00$ & $\hphantom{0}7.02$   & $0.29$  \\
C5     & $1.22$    & $1.85$   & $0.66$ & $2.05$ &                     $12.26$   & $0.17$  \\
C6     & $1.49$    & $2.55$   & $0.58$ & $3.41$ &                     $43.97$   & $0.08$  \\
C7     & $1.84$    & $3.46$   & $0.53$ & $2.61$ &                     $58.88$   & $0.04$  \\\hline
\end{tabular}
\caption{\label{tab:lrgenus3500}  Ratios  between   the  number  of  link
  condition   tests   performed   by  \textbf{RS},   \textbf{S},   and
  \textbf{BF} on the triangulations C0-C7 in the third group, and their
  corresponding execution time ratios.}
\end{center}
\end{table}

The   experiment  with   triangulations  D0-D9   indicates   that  the
performance of \textbf{S} decreases  as the genus of the triangulation
gets larger and the number of  triangles is kept about the same. As we
can    see   in    Table~\ref{tab:lctgenusvary},    the   values    of
$\epsilon_{\textit{s}}$ scales  up very quickly with  the genus, which
is  not the  case for  the values  of  $\epsilon_{\textit{bf}}$.  This
observation tells  us that an edge  chosen by \textbf{S}  to be tested
against the link condition is  much more likely to be non-contractible
than an edge  chosen at random (as it is the  case in \textbf{BF}). As
we  pointed  out  before,  those  ``bad'' choices  increase  the  time
\textbf{S} takes to  choose a lowest degree vertex  and a contractible
edge incident on it.  This is why $t_{\textit{s}}/t_{\textit{rs}}$ and
$t_{\textit{s}}/t_{\textit{bf}}$ get larger as the triangulation genus
grows,        despite        the        fact        that        ratios
$\ell_{\textit{s}}/\ell_{\textit{rs}}$                              and
$\ell_{\textit{s}}/\ell_{\textit{bf}}$      get      smaller      (see
Table~\ref{tab:lrgenusvary}). Moreover, since $n_t$ is large and about
the    same    for     triangulations    D0-D9,    the    values    of
$\epsilon_{\textit{bf}} / n_e$ for D4-D9, where $n_e$ is the number of
edges  of  the triangulation,  are  much  smaller  than the  ones  for
triangulations C2-C7.   Thus, contrary to  what we observe  for C2-C7,
\textbf{BF}   outperforms    \textbf{S}   on   the    smallest   genus
triangulations,          D4-D9,          as         well          (see
Figure~\ref{fig:varyinggenustimings}).

Finally,  observe  that  the  experiments  with  triangulations  B0-B9
corroborates the  fact that \textbf{RS}  runs in linear time  in $n_f$
for      triangulations      of      genus-$0$      surfaces      (see
Figure~\ref{fig:bricktimings}).    Likewise,   the  experiments   with
triangulations  D0-D9 indicates  that  the runtime  of \textbf{RS}  is
proportional  to the  term $g  \cdot n_f$.  To see  that,  we computed
$\delta    f_i   =    (n_{f_i}-n_{f_9})/n_{f_9}$,   $\delta    g_i   =
(g_{i}-g_9)/g_9$, and $\delta t_i = (  t_i - t_9)/t_9$, for every $i =
8,\ldots,0$, where $n_{f_j}$ and $g_j$ are the number of triangles and
the genus of  triangulation D$j$, respectively, and $t_j$  is the time
taken by \textbf{RS} on  triangulation D$j$, for every $j=0,\ldots,9$.
Then, we verified that $\delta t_i  / ( \delta f_i \cdot \delta g_i )$
is approximately  constant for triangulations D0-D4,  which have genus
greater than or equal to $1,\!500$.

\begin{table}[ht!]
\begin{center}
\begin{tabular}{|c|c|c|c|c|c|} \hline
\textbf{Triangulation} & $\bm{\ell_{\textit{rs}}}$ & $\bm{\ell_{\textit{s}}}$ & $\bm{\ell_{\textit{bf}}}$ & $\bm{\epsilon_{\textit{s}}}$  & $\bm{\epsilon_{\textit{bf}}}$ \\ \hline\hline
D0                                 &  $3,\!499,\!517$       & $3,\!336,\!889$        & $3,\!585,\!933.5$     & $346,\!735$                       & $52,\!866.6$ \\
D1                                 &  $3,\!298,\!209$       & $3,\!161,\!516$        & $3,\!576,\!091.5$     & $297,\!628$                       & $47,\!936.2$ \\
D2                                 &  $3,\!095,\!251$       & $2,\!981,\!813$        & $3,\!554,\!938.0$     & $247,\!822$                       & $43,\!124.9$ \\
D3                                 &  $2,\!895,\!440$       & $2,\!804,\!562$        & $3,\!522,\!016.4$     & $197,\!884$                       & $38,\!400.2$ \\
D4                                 &  $2,\!693,\!425$       & $2,\!628,\!433$        & $3,\!495,\!457.5$     & $148,\!587$                       & $33,\!570.1$ \\
D5                                 &  $2,\!493,\!458$       & $2,\!451,\!622$        & $3,\!466,\!912.2$     & $\hphantom{0}99,\!075$   & $28,\!706.5$ \\
D6                                 &  $2,\!293,\!492$       & $2,\!274,\!066$        & $3,\!436,\!328.4$     & $\hphantom{0}49,\!415$   & $23,\!811.3$ \\
D7                                 &  $2,\!129,\!696$       & $2,\!132,\!148$        & $3,\!412,\!062.2$     & $\hphantom{00}9,\!718$   & $19,\!991.1$ \\ 
D8                                 &  $2,\!110,\!987$       & $2,\!115,\!194$        & $3,\!409,\!289.6$     & $\hphantom{00}4,\!948$   & $19,\!357.0$ \\
D9                                 &  $2,\!087,\!627$       & $2,\!100,\!756$        & $3,\!406,\!560.1$     & $\hphantom{00}1,\!019$   & $19,\!072.6$ \\  \hline
\end{tabular}
\caption{\label{tab:lctgenusvary} The  number of link  condition tests
  and  the number of  edges tested  more than  once obtained  from the
  executions  of  \textbf{RS},  \textbf{S},  and  \textbf{BF}  on  the
  triangulations D0-D9 in the fourth group.}
\end{center}
\end{table}

\begin{table}[ht!]
\begin{center}
\begin{tabular}{|c|c|c|c|c|c|c|} \hline
\textbf{Triangulation}  & $\bm{\ell_{\textit{s}}} /
\bm{\ell_{\textit{rs}}}$ & $\bm{\ell_{\textit{bf}}} /
\bm{\ell_{\textit{rs}}}$ & $\bm{\ell_{\textit{s}}} /
\bm{\ell_{\textit{bf}}}$  & $\bm{t_{\textit{s}}} / \bm{t_{\textit{rs}}}$ & $\bm{t_{\textit{bf}}} / \bm{t_{\textit{rs}}}$ & $\bm{t_{\textit{s}}} / \bm{t_{\textit{bf}}}$ \\ \hline\hline
D0     & $0.95$    & $1.02$   & $0.93$ & $4.69$ & $2.42$   & $1.94$ \\
D1     & $0.96$    & $1.08$   & $0.88$ & $4.44$ & $2.42$   & $1.84$  \\
D2     & $0.96$    & $1.15$   & $0.84$ & $4.22$ & $2.41$   & $1.75$  \\
D3     & $0.97$    & $1.22$   & $0.80$ & $4.14$ & $2.45$   & $1.69$  \\
D4     & $0.98$    & $1.30$   & $0.75$ & $4.02$ & $2.49$   & $1.61$  \\
D5     & $0.98$    & $1.39$   & $0.71$ & $3.86$ & $2.46$   & $1.57$  \\
D6     & $0.99$    & $1.50$   & $0.66$ & $3.80$ & $2.55$   & $1.49$  \\
D7     & $1.00$    & $1.60$   & $0.62$ & $3.82$ & $2.65$   & $1.44$  \\
D8     & $1.00$    & $1.62$   & $0.62$ & $3.70$ & $2.66$   & $1.39$  \\
D9     & $1.01$    & $1.63$   & $0.62$ & $3.80$ & $2.70$   & $1.40$  \\\hline
\end{tabular}
\caption{\label{tab:lrgenusvary}  Ratios between  the  number of  link
  condition   tests   performed   by  \textbf{RS},   \textbf{S},   and
  \textbf{BF}  on the triangulations  D0-D9 in  the fourth  group, and
  their corresponding execution time ratios.}
\end{center}
\end{table}

\section{Conclusions}\label{sec:conc}
We   presented  a   new   algorithm  for   computing  an   irreducible
triangulation  $\T^{\prime}$  from a  given  triangulation  $\T$ of  a
connected,  oriented, and compact  surface $\s$  in $\E^d$  with empty
boundary. If the genus $g$ of $\s$ is positive, then $\T^{\prime}$ can
be computed  in $\mathcal{O}( g^2 + g  \, n_f )$ time,  where $n_f$ is
the number of triangles in $\T$.  Otherwise, $\T^{\prime}$ is computed
in linear  time in  $n_f$. In  both cases, the  space required  by the
algorithm is in $\Theta( n_f )$.  The time upper bound derived in this
paper   improves  upon   the   previously  best   known  upper   bound
in~\cite{SCH91} by a $\lg n_f / g$ factor.

We  also implemented our  algorithm, the  algorithm given  by Schipper
in~\cite{SCH91},  and a  randomized, brute-force  algorithm,  and then
experimentally   compared  these  implementations   on  triangulations
typically found  in graphics  applications, as well  as triangulations
specially devised to  study the runtime of the  algorithms in extreme
scenarios.   Our algorithm  outperformed  the other  two  in all  case
studies, indicating that the key ideas we use to reduce the worst-case
time complexity  of our  algorithm are also  effective in  the average
case  and for  triangulations typically  encountered in  practice. Our
experiments  also  indicated  that  the key  ideas  behind  Schipper's
algorithm are  not very effective  for the same  type of data,  as his
algorithm was outperformed by the brute-force one.

In the description  of our algorithm, we required  $\s$ be orientable,
as the augmented  DCEL we used in the  implementation of the algorithm
does not support nonorientable  surfaces.  However, our algorithm also
works  for nonorientable  surfaces  within the  same  time bounds,  as
Theorem~\ref{the:itsize}  is stated  in terms  of the  Euler  genus of
$\s$, which is half the value  of the (usual) genus $g$ for orientable
surfaces. We have not yet  extended our algorithm to deal with compact
surfaces with  a nonempty boundary  either.  A starting  point towards
this extension is the very  recent work by Boulch, de Verdi\`{e}re and
Nakamoto~\cite{BVN13},  which gives  an  analogous result  to that  of
Theorem~\ref{the:itsize} for (non-orientable) surfaces with a nonempty
boundary.

We  are interested in  investigating the  possibility of  lowering the
$\mathcal{O}( g^2  + g  \, n_f )$  upper bound,  so that the  $g^2$ is
replaced with $g$ and the bound becomes $\mathcal{O}( g \, n_f )$.  If
this  is possible, is  the resulting  bound tight?   Another important
research venue is  the development of a fast  algorithm for generating
the  complete  set of  \textit{all}  irreducible  triangulations of  a
surface from a  given triangulation of the surface.  We are interested
in developing  such an  algorithm by using  ours as a  building block,
providing an alternative method to that of Sulanke \cite{SUL06a}.

It  would be interesting  to find  out whether  some ideas  behind our
algorithm could speed up some topology-preserving, mesh simplification
algorithms~\cite{LRCVWH03}.  In particular,  one can devise a parallel
version of our algorithm to take advantage of the increasingly popular
and powerful graphics processing units (GPU). The idea is to process a
few vertices  of the input triangulation  at a time,  rather than only
one vertex. Theorem~\ref{the:itsize} can be used to give us an idea of
the size of the initial set of vertices. The algorithm must handle the
case  in which the  same vertex  becomes a  neighbor of  two currently
processed vertices. At this point,  an edge contraction could make two
currently processed vertices neighbors  of each other. If their common
edge is contractible, then one of two vertices can be removed from the
current  triangulation   by  contracting  the   edge.   This  parallel
algorithm can efficiently build  a hierarchy of triangulations such as
the one in~\cite{CDP04}.

\section{Acknowledgments}\label{sec:ackn}
All  triangulations  in   Table~\ref{tab:group1}  were  obtained  from
publicly available triangle  mesh repositories. Namely, triangulations
\texttt{Armadillo},  \texttt{Botijo}, \texttt{Casting}, \texttt{Eros},
\texttt{Fertility},      \texttt{Filigree},     \texttt{Hand},     and
\texttt{Socket}          were          taken         from          the
\href{http://shapes.aim-at-shape.net/}{Aim@Shape           Repository},
triangulation    \texttt{Happy   Buddha}    was    taken   from    the
\href{http://www.cc.gatech.edu/projects/large\_models/}{Large
  Geometric Models Archive},  and triangulation \texttt{Iphigenia} was
taken  from the  \href{http://www.pmp-book.org/}{website} of  the book
in~\cite{BKPAL10}.

\section*{References}

\bibliography{irt}

\begin{appendix}
\section{Proof of Lemma~\ref{lem:tritwofaces}}\label{sec:proofs}
A proof for Lemma~\ref{lem:tritwofaces} is given below:
\begin{proof}
  Let $e$  be any edge  of $(  G , i  )$.  Aiming at  a contradiction,
  assume that there are at  least three faces, $\tau_1$, $\tau_2$, and
  $\tau_3$, incident on $e$. Let $p$  be any point of $e$. Since every
  face is  an open  disk, and since  each edge  incident on a  face is
  entirely  contained in  the face  boundary, there  exists  a positive
  number $r_j$, for each  $j=1,2,3$, such that $\overline{\tau}_j \cap
  B(  p  ,  r_j  )$  is  homeomorphic to  the  half-disk,  $D$,  where
  $\overline{\tau}_j$ is the closure of $\tau_j$, $B( p, r_j )$ is the
  open ball of radius $r_j$ centered at $p$, and $D = \{ ( x , y ) \in
  \E^2 \mid x \ge 0 , x^2 + y^2 < 1 \}$. Furthermore, since $e$ cannot
  contain a vertex, if each $r_j$ is chosen small enough, then we also
  have that every point in $B( p , r_j )$ which is also a point on the
  boundary of $\tau$ belongs to $e$, i.e., $\partial( \tau_j ) \cap B(
  p , r_j ) = e \cap B(  p , r_j )$, where $\partial( \tau_j )$ is the
  boundary  of $\tau_j$. By  definition of  subdivision, we  know that
  $\tau_j \cap \tau_k = \emptyset$,  for any two $j,k \in \{ 1,2,3\}$,
  with $j \neq k$. So, if we take $r = \min\{ r_1,r_2,r_3\}$, then the
  intersection of the three half-disks, $\overline{\tau_1} \cap B( p ,
  r )$,  $\overline{\tau_2} \cap B(  p , r )$,  and $\overline{\tau_3}
  \cap B(  p , r )$, is  equal to $e \cap  B( p , r_j  )$, while their
  union is not  homeomorphic to an open disk (no  matter how small $r$
  is). So, there cannot be any neighborhood of $\s$ around $p$ that is
  homeomorphic to  a disk, which contradicts  the fact that  $\s$ is a
  surface. Thus, edge $e$ must be incident on either one or two faces.
  But, from Definition~\ref{def:triangulation}, the vertices and edges
  in  the boundary  of  each  face of  a  triangulation are  distinct.
  Moreover, since  the closure $\bar{\tau}$  of a single  face $\tau$,
  which  is  bounded by  three  distinct  vertices  and edges,  cannot
  entirely  cover  a  boundaryless,  compact surface  in  $\E^3$,  the
  complement  of $\bar{\tau}$  with respect  to $\s$  must  contain at
  least one more face. Consequently,  every edge of $\tau$ is incident
  on two faces, and so must be $e$.
\end{proof}

\section{Correctness}\label{sec:correctness}
This  appendix presents  a correctness  proof for  our  algorithm (see
Section~\ref{sec:algo}). 

We  start  by   proving  that  the  set  of   vertices  of  the  final
triangulation, $\T^{\prime}$, is a subset of the processed vertices of
the input triangulation,  $\T$. Later, we show that  such vertices are
all  trapped,  establishing   that  $\T^{\prime}$  is  an  irreducible
triangulation.

\begin{pro}\label{prop:processed}
  Whenever  condition $Q  \neq  \emptyset$  is reached  in  line 4  of
  Algorithm~\ref{alg:contractions}, we have  $p( z ) = \textit{true}$,
  where $z$  is a vertex in  $\T$, if and  only if either (1)  $z$ has
  been processed by the algorithm or  (2) $z$ was removed from $\T$ by
  an edge contraction (or from a triangulation obtained from $\T$ by a
  sequence of edge contractions).
\end{pro}
\begin{proof}
  Our proof  is by  induction on  the number, $i$,  of times  that the
  condition  $Q \neq \emptyset$  is reached  and tested  in line  4 of
  Algorithm~\ref{alg:contractions}. We  note that $1  \le i \le  n_v +
  1$,  where   $n_v$  is   the  number  of   vertices  in   the  input
  triangulation, $\T$.   This is because  $Q$ is initialized  with all
  vertices  of $\T$  in  Algorithm~\ref{alg:continit} and  $Q$ is  not
  modified    until    the   while    loop    of    lines   4-37    of
  Algorithm~\ref{alg:contractions}  is  executed.   Furthermore,  each
  loop iteration removes exactly one  vertex from $Q$ and no vertex is
  ever inserted into $Q$ during the execution of the loop.

  \textbf{Base case ($\bm{i  = 1}$)}. If $i =  1$, then the hypothesis
  of our  claim is  \textit{if $Q \neq  \emptyset$ is reached  for the
    first time}. It  turns out that $Q \neq  \emptyset$ will always be
  reached and tested  at least once. In the first time,  we have $p (z
  )$ equals to  \textit{false}, for every vertex $z$ in  $\lte$, as $p( z
  )$   is   set   to    \textit{false}   during   the   execution   of
  Algorithm~\ref{alg:continit}.   Furthermore,   no  vertex  has  been
  processed  and no  edge  contraction  has been  carried  out by  the
  algorithm yet. So, our claim holds for $i = 1$.

  \textbf{Hypothesis ($\bm{i = k}$)}.  Assume that our claim is true for $i
  = k$, where  $k$ is an arbitrary (but fixed) integer,  with $1 \le k
  \le n_v$.  This is  equivalent to saying  that if condition  $Q \neq
  \emptyset$   is  tested   for  the   $k$-th  time   in  line   4  of
  Algorithm~\ref{alg:contractions}, then (at that moment) the value of
  $p( z )$ is \textit{true}, for  every vertex $z$ in $\T$, iff either
  (1)  $z$  has  been  processed  by  the algorithm  or  (2)  an  edge
  contraction removed  $z$ from $\T$ or from  a triangulation obtained
  from $\T$ by the edge contractions.

  \textbf{Inductive Step  ($\bm{i = k+1}$)}. Assume that  condition $Q \neq
  \emptyset$  is  tested  for  the   $(k+1)$-th  time  in  line  4  of
  Algorithm~\ref{alg:contractions}.   So,  let   us  unroll  the  last
  iteration of the loop and  consider the moment in which condition $Q
  \neq \emptyset$  was tested for  the $k$-th time.  By  the induction
  hypothesis, the value of $p( z )$ is \textit{true}, for every vertex
  $z$ in $\T$, iff either (1)  $z$ has been processed by the algorithm
  or  (2)  an  edge  contraction  removed  $z$ from  $\T$  or  from  a
  triangulation    previously    obtained    from   $\T$    by    edge
  contractions. During  the $k$-th iteration of the  outer while loop,
  there are  only two places  in which the  current value of $p(  z )$
  could be  changed: line 12 of  Algorithm~\ref{alg:contract} and line
  35  of Algorithm~\ref{alg:contractions}.  The  former line  sets the
  value of  $p( z  )$ to \textit{true}  iff $[  u , z  ]$ is  the edge
  chosen  to be  contracted by  the  algorithm, $u$  is the  currently
  processed vertex,  and $z$  is the vertex  removed from  the current
  triangulation by  contraction of $[ u  , z ]$.  In  turn, the latter
  line sets  the value  of $p( z  )$ to  \textit{true} iff $z$  is the
  currently  processed  vertex,  $u$.   So,  when  condition  $Q  \neq
  \emptyset$ is tested for  the $(k+1)$-th time, immediately after the
  $k$-th  iteration  of the  loop  ends,  the value  of  $p(  z )$  is
  \textit{true}, for every vertex $z$  in $\T$, iff either (1) $z$ has
  been processed by  the algorithm or (2) an  edge contraction removed
  $z$ from $\T$ or from  a triangulation previously obtained from $\T$
  by  edge contractions.  Thus, our  claim  is true  for all  $i =  1,
  \ldots, n_v+1$.
\end{proof}

The      next     five     results      allow     us      to     prove
Proposition~\ref{prop:listLUE}:

\begin{pro}\label{prop:listR}
  Let  $u$ be  any vertex  of  $\T$ selected  to be  processed by  the
  algorithm.  Then,  before the processing of $u$  starts, list $\lue$
  is  initialized  with  every  edge  $[  u ,  v  ]$  in  the  current
  triangulation, $K$, where $v$ is a vertex of the link, $\textit{lk}(
  u  ,   K  )$,   of  $u$   in  $K$  that   has  not   been  processed
  yet. Furthermore,  if $[ u , v  ]$ is an edge  inserted into $\lue$,
  where $v$  is a vertex  in $K$ of  degree, $d_v$, greater  than $3$,
  then edge $[ u , v ]$ cannot  precede an edge $[ u , w ]$ already in
  $\lue$ such that the degree, $d_w$, of $w$ is $3$.
\end{pro}
\begin{proof}
  By assumption, the  value of $p( u )$  is \textit{false} when vertex
  $u$     is     removed     from     $Q$     in     line     5     of
  Algorithm~\ref{alg:contractions}.                                From
  Proposition~\ref{prop:processed},   we  know  that   $p(  u   )$  is
  \textit{false} iff $u$  has not been processed by  the algorithm yet
  nor has it been removed from the current triangulation, $K$, or from
  a previous  triangulation, by an  edge contraction. This  means that
  $u$ was not processed before and  that $u$ belongs to $K$. Since $p(
  u   )$  is   \textit{false},  the   for  loop   in  lines   8-19  of
  Algorithm~\ref{alg:contractions}   is    executed.   We   can   view
  $\textit{lk}( u , K )$ as  a list of vertices, which means that line
  8 is executed $h + 1$ times and the loop body is executed $h$ times,
  where $h$ is  the number of vertices in $\textit{lk}( u  , K )$. For
  every  $i \in  \{  1, \ldots,  h \}$,  where  $h$ is  the number  of
  vertices in $\textit{lk}(  u , K )$, let $v_i$  be the $i$-th vertex
  of $\textit{lk}(  u , K  )$ considered by  the loop.  We  claim that
  {\em by the time line 8 is  executed for the $j$-th time, for $j \in
    \{ 1, \ldots, h + 1 \}$, list $\lue$ consists of every edge $[ u ,
    v_k ]$  in $K$, with $k  \in \{ 1, \ldots  , j - 1  \}$, such that
    $v_k$  has not  been processed  by  the algorithm  yet}. If  every
  vertex  $v_k$, with  $k \in  \{ 1,  \ldots  , j  - 1  \}$, has  been
  processed   by  the   algorithm  before,   then  $\lue$   is  empty.
  Furthermore, if  $[ u  , v  ]$ is an  edge in  $\lue$ such  that the
  degree, $d_v$,  of $v$ is greater  than $3$, then  edge $[ u ,  v ]$
  cannot precede an edge  $[ u , w ]$ in $\lue$  such that the degree,
  $d_w$, of  $w$ is $3$.   If this claim  is true, then it  also holds
  immediately after  the loop ends. Consequently, the  veracity of the
  claim implies the veracity of  our proposition.  In what follows, we
  prove the claim by induction on $j$.

  \textbf{Base  case ($\bm{j  = 1}$)}.  If  $j =  1$, then  line 8  of
  Algorithm~\ref{alg:contractions}    selects   vertex    $v_1$   from
  $\textit{lk}(  u ,  K  )$. At  this  moment, no  previous vertex  of
  $\textit{lk}( u  , K )$ has  been selected before and  the loop body
  has not been executed yet. Since  list $\lue$ was made empty in line
  7, it is empty when line 8  is reached. So, our claim holds for $j =
  1$.

  \textbf{Hypothesis ($\bm{j  = k}$)}.  Assume that our  claim is true
  for $j = k$, where $k$  is an arbitrary (but fixed) integer, with $1
  \le k  \le h$.   This is equivalent  to saying  that when line  8 of
  Algorithm~\ref{alg:contractions}  is executed  for the  $k$-th time,
  list $\lue$ consists of  each edge $[ u , v_l ]$,  with $l \in \{ 1,
  \ldots, k  - 1 \}$,  such that $v_l$  has not been processed  by the
  algorithm yet. If every vertex $v_l$, with $l \in \{ 1, \ldots , l -
  1 \}$,  has been processed by  the algorithm before,  then $\lue$ is
  empty. Furthermore, if  $[ u , v  ]$ is an edge in  $\lue$ such that
  $d_v > 3$,  then $[ u , v ]$  cannot precede an edge $[ u  , w ]$ in
  $\lue$ such that $d_w = 3$.

  \textbf{Inductive Step ($\bm{j  = k+1}$)}. When $j = k  + 1$, line 8
  of  Algorithm~\ref{alg:contractions} is executed  for the  $( k  + 1
  )$-th time.   Since $k \ge  1$, the loop  body has been  executed at
  least once. So, consider the moment in which line 8 was executed for
  the $k$-th  time.  At  this moment, vertex  $v_k$ was  selected from
  $\textit{lk}(   u    ,   K   )$,   where   $K$    is   the   current
  triangulation. Next, edge $[ u ,  v_k ]$ is inserted into $\lue$ iff
  $p   (   v_k   )$    is   \textit{false}   (see   lines   12-18   of
  Algorithm~\ref{alg:contractions}).                               From
  Proposition~\ref{prop:processed}, we can conclude that $[ u , v_k ]$
  is inserted into  $\lue$ iff $v_k$ belongs to $K$  and $v_k$ has not
  been      processed     yet.       From      lines     12-18      of
  Algorithm~\ref{alg:contractions},  we have  that $[  u ,  v_k  ]$ is
  inserted at  the front of  $\lue$ if the  degree of $v_k$ in  $K$ is
  equal to $3$.  Otherwise, it is  inserted at the back of $\lue$.  By
  the induction hypothesis,  list $\lue$ consists of all  edges $[ u ,
  v_l ]$ in $K$, with $l \in \{  1, \ldots, k - 1 \}$, such that $v_l$
  has not been  processed by the algorithm yet  (if any). Furthermore,
  if $[ u , v ]$ is an edge in $\lue$ such that $d_v > 3$, then $[ u ,
  v ]$ cannot precede  an edge $[ u , w ]$ in  $\lue$ such that $d_w =
  3$.  So, the property ``\textit{if $[ u  , v ]$ is an edge in $\lue$
    such that $d_v > 3$, then $[ u  , v ]$ cannot precede an edge $[ u
    , w  ]$ in $\lue$ such that  $d_w = 3$}'' still  holds.  Thus, our
  claim  holds  when  line  8 of  Algorithm~\ref{alg:contractions}  is
  executed for the $j$-th time,  with $j=k+1$. Thus, our claim is true
  for every  $j \in \{ 1,  \ldots, h+1 \}$.  Since  this claim implies
  that the  statement we want to  prove is true  immediately after the
  for loop of lines 8-19 is finished, we are done.
\end{proof}

The invariant stated  by Proposition~\ref{prop:listR} can be augmented
to include the facts that $n( v )  = u$, $o( v ) = \textit{ts}$, $t( v
) = -1$, and $c( v ) = 0$,  for every vertex $v$ such that $[ u , v ]$
belongs to $\lue$,  where $u$ is the vertex  selected to be processed,
and \textit{ts} is the current  ``time'' (which does not change during
the  entire execution  of  the for  loop  in lines  8-19 of  Algorithm
\ref{alg:contractions}).  We  can also prove that  a slightly modified
invariant about $\lue$ holds for  the repeat-until loop of lines 21-34
of  Algorithm~\ref{alg:contractions}.  To  do so,  we note  that every
edge  inserted into $\lue$  during the  loop execution  is done  so by
\textsc{Contract}$()$. So, we have the following:

\begin{cla}\label{cla:contract1}
  Every edge $[  u , z ]$ inserted  by \textsc{Contract}$()$ into $\lue$,
  during the processing of a vertex $u$  of $\T$, is such that $n( z )
  = u$, $o( z ) = \textit{ts}$, $t( z ) = -1$, and $c( z ) = 0$, where
  \textit{ts} is the insertion time.
\end{cla}
\begin{proof}
  Every edge inserted into  $\lue$ by \textsc{Contract}$()$ comes from
  the  temporary  list,  \textit{temp}.    This  list  is  made  empty
  immediately before  \textsc{Collapse}$()$ is called (see  line 14 of
  Algorithm   \ref{alg:contract}).    When  \textsc{Collapse}$()$   is
  executed, it updates the DCEL to account for the contraction of edge
  $[ u  , v ]$.  Let  $L$ be the triangulation  immediately before the
  contraction of $[ u , v ]$, i.e., $K = L - uv$, where $K$ is current
  triangulation.  Then,  updating the  DCEL amounts to  replacing each
  edge of the  form $[ v , z ]$ in  $L$ with an edge $[ u  , z ]$, for
  every $z \in \po_{uv}$.  As a result,  vertex $v$, edge $[ u , v ]$
  and the $2$-faces $[ u  , v , x]$ and $[ u , v  , y ]$ do not belong
  to $K$, where $x$ and $y$ are  the vertices in $\textit{lk}( [ u , v
  ] , L )$.   \textsc{Collapse}$()$ also inserts each replacement edge
  $[ u ,  z ]$ into \textit{temp}. So,  list \textit{temp} consists of
  every  edge $[ u  , z  ]$ such  that $z  \in \po_{uv}$.   But, each
  vertex $z$  in $\textit{lk}( v ,  L )$ had its  attributes $n$, $c$,
  $o$, and $t$  set to $u$, $0$, \textit{ts},  and $-1$, respectively,
  by the  for loop in lines 4-11  of Algorithm~\ref{alg:contract}. The
  value of $t( z )$ is  the same for the vertices $z$ in $\textit{lk}(
  v , L )$, and it is equal to the current value of \textit{ts}, which
  in turn was last  updated in line 3 of Algorithm~\ref{alg:contract}.
  Finally, every edge in \textit{temp}  is inserted into $\lue$ by the
  for  loop  in lines  25-43  of Algorithm~\ref{alg:contract}.   Since
  \textit{ts} is  only updated  in line $3$  of \textsc{Contract}$()$,
  the value of $t(z)$  can be viewed as the time in which  $[ u , z ]$
  is  inserted in  $\lue$, as  well as  the time  in which  vertex $z$
  became a vertex of $\textit{lk}( u , K )$.
\end{proof}

\begin{cla}\label{cla:contract2}
  Let  $u$ be any  vertex of  $\T$ processed  by the  algorithm. Then,
  during  the processing  of $u$,  if the  following  three conditions
  regarding     $\lue$    hold     immediately     before    executing
  \textsc{Contract}$()$,  then they also  hold immediately  after: (1)
  every edge  $[ u ,  w ]$ in  $\lue$ is also  an edge of  the current
  triangulation, $K$, (2)  if $[ u ,  w ]$ is any edge  in $\lue$ such
  that the degree, $d_w$, of $w$ is greater than $3$, then edge $[ u ,
  w ]$  cannot precede an  edge $[ u  , z ]$  in $\lue$ such  that the
  degree, $d_z$, of $z$ is equal to $3$, and (3) if $[ u , w ]$ is any
  edge in  $\lue$, then  $w$ has not  been processed by  the algorithm
  yet.
 \end{cla} 
\begin{proof}
  Let  $[  u  ,  v  ]$  be  the  edge  of  $K$  to  be  contracted  by
  \textsc{Contract}$()$.  As  we already pointed  out in the  proof of
  Claim~\ref{cla:contract1}, the  contraction of $[ u ,  v ]$ replaces
  each edge of the form  $[ v , z ]$ in $K$ with an  edge $[ u , z ]$,
  for every $z \in \po_{uv}$. Moreover,  vertex $v$, edge $[ u , v ]$
  and $2$-faces $[ u , v , x]$ and  $[ u , v , y ]$, where $x$ and $y$
  are the  two vertices  in $\textit{lk}( [  u ,  v ] ,  K )$,  do not
  belong  to   the  resulting  triangulation,  $K  -   uv$,  and  each
  replacement edge $[ u , z ]$ is inserted into \textit{temp}. Observe
  that every replacement edge is an edge of $K - uv$.  It is possible,
  though, that $p( z ) =  \textit{true}$ for a vertex $z$ such that $[
  u , z ]$ is $\lue$. However, edge  $[ u , z ]$ will be inserted into
  $\lue$   iff  $p(   z  )   =   \textit{false}$  (see   line  26   of
  Algorithm~\ref{alg:contract}).   So, conditions  (2)  and (3)  hold.
  Regarding condition  (2), note that  the only vertices  whose degree
  are modified by \textsc{Contract}$()$ are $u$, $x$, and $y$. So, the
  only edges that could violate condition (2) are $[ u , x ]$ and $[ u
  ,  y ]$.   But, lines  19-24 of  Algorithm~\ref{alg:contract} checks
  whether each of $[ u , x ]$ and $[  u , y ]$ is in $\lue$. If $[ u ,
  x ]$ (resp.  $[  u , y ]$) is in $\lue$  and the degree $d_x$ (resp.
  $d_y$) of $x$ (resp.  $y$) is  equal to $3$ after the contraction of
  $[ u , v ]$  (i.e., in $K - uv$), then $[ u , x  ]$ (resp.  $[ u , y
  ]$) is moved to the front of $\lue$.  So, condition (2) also holds.
\end{proof}

\begin{pro}\label{prop:listRwhile}
  Let $u$  be any  vertex of $\T$  processed by the  algorithm.  Then,
  during  the  processing of  $u$,  if  the  following hold  when  the
  condition    of    the    while    loop   in    lines    22-30    of
  Algorithm~\ref{alg:contractions} is tested  for the first time, they
  hold every time the condition is tested again:
\begin{itemize}
\item[(1)] every edge $[ u , w ]$ in $\lue$ is also an edge of the
  current triangulation, $K$;

\item[(2)] if $[ u ,  w ]$ is an edge  in $\lue$ such  that the degree,
  $d_w$, of $w$ is greater than $3$, then  edge $[ u , w ]$ cannot precede an
  edge $[ u , z ]$ in $\lue$ such that the degree, $d_z$, of $z$ is equal
  to $3$;

\item[(3)] the value of $p(z)$ is \textit{false}, for every vertex $z$
  such that $[ u , z ]$ is in $\lue$; 

\item[(4)] the value of $o(z)$ is no longer $-1$, for every vertex $z$ such that
  $[ u , z ]$ is in $\lue$;

\item[(5)] the value of $t(z)$ is $-1$, for every vertex $z$ such that
  $[ u , z ]$ is in $\lue$;

\item[(6)] the value of $c(z)$ is $0$, for every vertex $z$ such that
  $[ u , z ]$ is in $\lue$; and

\item[(7)] no edge in $\lue$ has been tested against the link condition
  yet.
\end{itemize}
\end{pro}
\begin{proof}
  Our proof is by induction on  the number $i$ of times that the while
  loop   condition,  $\lue   \neq   \emptyset$,  in   line  22-30   of
  Algorithm~\ref{alg:contractions}  is reached  and tested.   Since we
  have not proved termination of the loop yet, we will assume that the
  loop  ends after $m$  iterations, for  a non-negative  integer, $m$.
  So, we can conclude that $1 \le  i \le m + 1$. Termination is proved
  later.

  \textbf{Base case  ($\bm{i = 1}$)}.  When the while  loop condition,
  $\lue \neq  \emptyset$, is  reached and tested  for the  first time,
  conditions (1) to (7) are all true \textit{by the hypothesis} of our
  claim.  When  we   show  that  (1)-(7)  also  hold   for  the  outer
  repeat-until loop, this hypothesis will be proved to be true a valid
  one.

  \textbf{Hypothesis ($\bm{i =  k}$)}.  Assume that conditions (1)-(7)
  hold for  $i = k$,  where $k$ is  an arbitrary (but  fixed) integer,
  with $1 \le k \le m$;  that is, conditions (1)-(7) hold if condition
  $\lue \neq \emptyset$ is tested for the $k$-th time.

  \textbf{Inductive Step  ($i = k+1$)}.  When  $i = k +  1$, the while
  loop condition, $\lue \neq \emptyset$, is reached and tested for the
  $(k+1)$-th time.  Since  $k \ge 1$, the loop  body has been executed
  at least once.   So, consider the moment in  which the condition was
  reached  and tested for  the $k$-th  time.  Since  the loop  body is
  executed,  the while  loop condition  holds and  thus $\lue$  is not
  empty.  At this point, the  induction hypothesis implies that if any
  of conditions (1)-(7) is  violated, then the violation occurs during
  the $k$-th iteration of the loop.   We now show that this is not the
  case.

  The $k$-th iteration of the loop  starts with the removal of edge $[
  u , v ]$ from the front of  list $\lue$. Next, the value of $t( v )$
  is set  to \textit{ts}. Since $[  u , v  ]$ is no longer  in $\lue$,
  this assignment does  not violate condition (4). Next,  if the degree
  of    $v$   is    equal   to    $3$,   then    the    procedure   in
  Algorithm~\ref{alg:processdegree3}   is  invoked.    Otherwise,  the
  procedure  in   Algorithm~\ref{alg:processdegreegt3}  is.   In  both
  cases, no edge  incident with $u$ other  than $[ u , v  ]$ is tested
  against the link  condition or further removed from  $\lue$ and from
  the  current triangulation.  Furthermore,  every edge  inserted into
  $\lue$   is   done   so   by   \textsc{Contract}$()$.    But,   from
  Claim~\ref{cla:contract2}, we can conclude that conditions (1), (2),
  and  (3) cannot  be  violated.   Conditions (4)  and  (5) cannot  be
  violated either, as the only  places in which attributes $o$ and $t$
  are   modified  (during  the   loop  execution)   are  line   24  of
  Algorithm~\ref{alg:contractions}    and   lines    8   and    9   of
  Algorithm~\ref{alg:contract}. Line 8 of Algorithm~\ref{alg:contract}
  sets  the  value  of the  $o$  attribute  to  the current  value  of
  \textit{ts},   which   is   always   non-negative.    Line   24   of
  Algorithm~\ref{alg:contractions} modifies the value of $t( v )$, but
  $[ u , v ]$ is no longer in $\lue$ at this point. In turn, line 9 of
  Algorithm~\ref{alg:contract} set  the value of the  $t$ attribute to
  $-1$.  Finally, for  every  vertex $z$  such  that $[  u  , z]$  was
  inserted  into $\lue$ by  Algorithm~\ref{alg:contract}, the  $o$ and
  $t$ attributes of $z$ were modified  by lines 8 and 9. So,
  conditions (4) and  (5) cannot be violated.

  During any while  loop iteration, the $c$ attribute  of a vertex can
  only      be      modified     by      Algorithm~\ref{alg:contract},
  Algorithm~\ref{alg:processdegree3}                                and
  Algorithm~\ref{alg:processdegreegt3}.    Algorithm~\ref{alg:contract}
  can only  modify the value of the  $c$ attribute of a  vertex in two
  lines, namely: line 7 and line 36. The former line sets the value of
  the $c$ attribute to $0$, while the latter line increments the value
  of the  $c$ attribute of a  vertex $w$ by  $1$.  But, if line  36 is
  executed then $t( w ) \neq -1$, which  means that $[ u , w ]$ is not
  in   $\lue$.    Thus,   condition   (6)  cannot   be   violated   by
  Algorithm~\ref{alg:contract}.                 In               turn,
  Algorithm~\ref{alg:processdegree3} can only modify the $c$ attribute
  of vertices $x$ and $y$ in $\textit{lk}( [ u , v ] , K )$, where $K$
  is the  triangulation before  the contraction  of $[ u  , v  ]$ (see
  lines 5-24 of  Algorithm~\ref{alg:processdegree3}).  However, if $c(
  x )$ (resp.  $c( y )$) is  changed, then $t( x ) \neq -1$ (resp. $t(
  y ) \neq -1$). But, we just showed that $[ u , x ]$ (resp.  $[ u , y
  ]$) cannot  be in $\lue$ if  $t( x ) \neq  -1$ (resp.  $t(  y ) \neq
  -1$); otherwise, condition (5)  would be violated. So, condition (6)
  cannot be violated  by Algorithm \ref{alg:processdegree3}.  Finally,
  Algorithm~\ref{alg:processdegreegt3}   can  only   modify   the  $c$
  attribute of a  vertex $v$ and of  a vertex $z$ such that  $z$ is in
  $\textit{lk}( v ,  K )$, where $K$ is  the current triangulation and
  $[ u , v ]$ is the edge given as input to the procedure (see lines 4
  and 6 of Algorithm~\ref{alg:processdegreegt3}).  But, since $[ u , v
  ]$ is no longer in $\lue$, the  possible change of $c( v )$ does not
  violate condition (6).  Moreover, if  $c(z)$ is changed in line 6 of
  Algorithm~\ref{alg:processdegreegt3},  then $(  u ,  v ,  z )$  is a
  critical cycle  of $K$ and  $t( z )  \neq -1$.  But, we  just showed
  that if  $t( z )  \neq -1$, then  $[ u , z  ]$ cannot be  in $\lue$;
  otherwise, condition  (5) would be  violated.  So, condition  (6) is
  not violated either.  

  Regarding condition (7), recall  that all edges inserted into $\lue$
  during     every     loop      iteration     are     inserted     by
  \textsc{Contract}$()$. But, these edges have not been tested against
  the  link condition before,  as they  do not  belong to  the current
  triangulation, $K$,  to the initial triangulation, $\T$,  nor to any
  triangulation obtained from  $\T$ before $K$.  The reason  is that a
  previously  contracted   edge  can  never  be   resurrected  by  the
  algorithm, as this would require  the resurrection of a vertex. But,
  nowhere  in the  algorithm a  vertex  is introduced  in the  current
  triangulation.  So, condition (7)  cannot be violated either.  Thus,
  our claim holds for $i = 1, \ldots, m+1$.
\end{proof}

\begin{pro}\label{prop:whiletermination}  The  while  loop  in  lines
  22-30 of Algorithm~\ref{alg:contractions} terminates.
\end{pro}
\begin{proof}
  We  must  show that  list  $\lue$  will  be eventually  empty  after
  finitely  many  iterations of  the  while  loop  in lines  22-30  of
  Algorithm~\ref{alg:contractions}.   Indeed, every  iteration removes
  exactly one edge, $[ u , v ]$, from $\lue$. If $[ u , v ]$ fails the
  link condition, then  no edge is inserted into  $\lue$ until the end
  of the current  loop iteration. Otherwise, one or  more edges of the
  form $[ u ,  z ]$ are inserted into $\lue$, where  $u$ is the vertex
  being processed by the algorithm, $z$ is a vertex in $\textit{lk}( v
  ,  K  )$,  with $z  \not  \in  \pd_{uv}$,  and $K$  is  the
  triangulation immediately before the contraction of $[ u , v ]$.  We
  claim  that only  finitely  many  edges can  ever  be inserted  into
  $\lue$.  Indeed, during the execution of  the loop, an edge $[ u , z
  ]$ is inserted into $\lue$ if and only if $[ u , z ]$ belongs to the
  triangulation, $K - uv$, resulting from  the contraction of $[ u , v
  ]$ in $K$. But, edge $[ u , z ]$ cannot belong to $K$, else $( u , v
  , z )$ would be a critical cycle  in $K$ and hence $[ u , v ]$ would
  not be  contractible in $K$.   In general, edge  $[ u , z  ]$ cannot
  belong to  $\T$ nor to  any triangulation obtained from  $\T$ before
  $K$ is.  Otherwise, either $[ u , z ]$ would belong to $K$ or $[ u ,
  z ]$ would have been  removed from some triangulation preceding $K$.
  Since $[  u , z ]$  cannot belong to  $K$ (for the reason  we stated
  before), $[  u , z ]$  would have been removed  from a triangulation
  obtained  from $\T$  before $K$  is  obtained.  But,  in this  case,
  vertex $z$ would not be a vertex of $K$, as the algorithm never adds
  vertices to a triangulation.  Since $\T$ has a finite number, $n_v$,
  of vertices, there can be  at most $n_v-1$ edge insertions in $\lue$
  during the  entire execution  of the while  loop.  But,  since every
  iteration of the  loop removes an edge from  $\lue$, we can conclude
  that $\lue$ will be  eventually empty after finitely many iterations
  of the while loop.
 \end{proof}

 We   now  give  a   proof  for   Proposition~\ref{prop:listLUE}  from
 Section~\ref{sec:counting}:

\begin{proof}
  Our  proof is  by induction  on  the number  $i$ of  times that  the
  repeat-until  loop condition,  $\lue  = \emptyset$,  in  line 34  of
  Algorithm~\ref{alg:contractions}  is reached.   We  have not  proved
  termination of the  loop yet.  So, let us assume  that the loop ends
  after $m$  iterations, for some  non-negative integer, $m$.   So, we
  can conclude that $1 \le i \le m + 1$.

  \textbf{Base  case  ($\bm{i  =  1}$)}. When  the  repeat-until  loop
  condition, $\lue  = \emptyset$, is  reached for the first  time, the
  loop body has  been executed exactly once.  So,  let us consider the
  moment  in  which line  21  of Algorithm~\ref{alg:contractions}  was
  executed    for     the    first    time.      At    this    moment,
  Proposition~\ref{prop:listR}  ensures that conditions  (1)-(3) hold.
  Conditions (4) and (5) also hold,  as the values of $o(z)$ and $t( z
  )$, for every vertex $z$ such that  $[ u , z ]$ is $\textit{lk}( u ,
  K )$,  were set  to \textit{ts} and  $-1$, respectively, by  the for
  loop in lines 8-19 of Algorithm~\ref{alg:contractions}, where $K$ is
  the current  triangulation at the time  line 21 of  was executed for
  the first  time.  Furthermore,  every edge in  $\lue$ is an  edge in
  $\textit{lk}(    u    ,    K    )$.     Due    to    line    9    of
  Algorithm~\ref{alg:continit},  condition (6)  also  holds.  This  is
  also the case  of condition (7), as no edge  has been tested against
  the  link  condition  yet  \textit{during the  processing  of  $u$}.
  Proposition~\ref{prop:listRwhile}  ensures  that conditions  (1)-(7)
  are  not violated  by the  while loop  in lines  22-30 if  they hold
  immediately before the loop takes  place (which we just argued to be
  the    case).     The    while    loop    in    lines    22-30    of
  Algorithm~\ref{alg:contractions}    must   end.     Otherwise,   our
  hypothesis would not be true, i.e., the repeat-until loop condition,
  $\lue = \emptyset$,  in line 34 would never  have been reached.  So,
  line 32 of Algorithm~\ref{alg:contractions} may be executed.  If so,
  list $\lue$ is only modified by possible edge insertions carried out
  by   \textsc{Contract}$()$.    But,  Claim~\ref{cla:contract1}   and
  Claim~\ref{cla:contract2}  ensure that  conditions  (1)-(5) are  not
  violated.   In  turn, Claim~\ref{cla:contract1}  and  the fact  that
  condition  (5) is  not  violated  imply that  condition  (6) is  not
  violated  either (see  the arguments  in the  inductive step  of the
  proof  of  Proposition~\ref{prop:listRwhile}).   Finally, since  the
  only  edge tested  against the  link condition  by the  procedure in
  Algorithm~\ref{alg:processlistS}  is an  edge  removed from  $\lte$,
  which is not in $\lue$, we can conclude that condition (7) cannot be
  violated either. So, our claim is true for $i = 1$.

  \textbf{Hypothesis ($\bm{i =  k}$)}. Assume that conditions (1)-(7)
  hold for  $i = k$,  where $k$ is  an arbitrary (but  fixed) integer,
  with $1 \le k \le m$.  That is, conditions (1)-(7) hold if condition
  $\lue = \emptyset$ is tested for the $k$-th time.

  \textbf{Inductive  Step ($\bm{i =  k+1}$)}. When  $i =  k +  1$, the
  repeat-until loop condition, $\lue  = \emptyset$, is reached for the
  $(k+1)$-th time. Since $k \ge 1$, the loop body has been executed at
  least twice.   So, consider  the moment in  which the  condition was
  reached  for   the  $k$-th  time.   By   the  induction  hypothesis,
  conditions  (1)-(7) will  hold immediately  after the  loop  test is
  executed, i.e., they will be true in the beginning of the $(k+1)$-th
  iteration. So, using  the same arguments from the  base case, we can
  show that  conditions (1)-(7) will  hold immediately after  the loop
  body  is  executed.   Consequently,  conditions  (1)-(7)  hold  when
  condition   $\lue  =   \emptyset$   is  reached   in   line  34   of
  Algorithm~\ref{alg:contractions}   for   the   $i$-th   time,   with
  $i=k+1$. Thus, our claim is true for $i = 1, \ldots, m+1$.
\end{proof}

The previous proof  assumed that the repeat-until loop  in lines 21-34
of             Algorithm~\ref{alg:contractions}            terminates.
Proposition~\ref{prop:repeatuntiltermination}  below  states that  our
assumption is valid.

\begin{pro}\label{prop:repeatuntiltermination}  The repeat-until loop
  in lines 21-34 of Algorithm~\ref{alg:contractions} terminates.
\end{pro}
\begin{proof}
  We  must  show that  list  $\lue$  will  be eventually  empty  after
  finitely many iterations of the  repeat-until loop in lines 21-34 of
  Algorithm~\ref{alg:contractions}.
  Proposition~\ref{prop:whiletermination} ensures  that the while loop
  in   lines    22-30   of   Algorithm~\ref{alg:contractions}   always
  terminates.  So, in each  iteration of  the repeat-until  loop, list
  $\lue$ gets empty after the while loop is executed. Thus, line 30 is
  reached.  From this  point on, list $\lue$ can  still be modified by
  edge insertions carried out by \textsc{Contract}$()$. However, using
  the      same       arguments      from      the       proof      of
  Proposition~\ref{prop:whiletermination}, we  can conclude that there
  are  only  finitely  many  edges  that  \textit{all}  executions  of
  \textsc{Contract}$()$ could  possibly insert into  $\lue$. So, after
  some  finitely many  iterations of  the repeat-until  loop,  no more
  edges   are    inserted   into   $\lue$   by    the   procedure   in
  Algorithm~\ref{alg:processlistS}.  Thus, list $\lue$ will eventually
  remain  empty after  this procedure  is executed.  Consequently, the
  loop ends.
 \end{proof}

 Edges  are  inserted  into list  $\lte$  at  only  one place  in  our
 algorithm:    line   16    of   Algorithm~\ref{alg:processdegreegt3}.
 Furthermore, an edge inserted into  $\lte$ has just been removed from
 $\lue$ and  tested against the  link condition.  Since the  same edge
 removed from  $\lue$ never  gets back to  $\lue$ again (see  proof of
 Proposition~\ref{prop:whiletermination}),  lists  $\lue$  and  $\lte$
 have no edge in common  during the entire execution of the algorithm.
 Moreover, every edge inserted into $\lte$ has been tested against the
 link  condition before  and  has failed  the  test. So,  we have  the
 following:

\begin{pro}\label{prop:listSwhile}
  Let $u$  be any  vertex of $\T$  processed by the  algorithm.  Then,
  during  the  processing of  $u$,  if  the  following hold  when  the
  condition    of    the    while    loop   in    lines    22-30    of
  Algorithm~\ref{alg:contractions} is tested  for the first time, they
  hold every time it is tested:
\begin{itemize}
\item[(1)] $\lte$ and $\lue$ have no edge in common;

\item[(2)] if $[ u , z ]$ is an edge in $\lte$ then $t( z ) \ge o(z) >
  -1$; and

\item[(3)]  every edge  in $\lte$  has  been tested  against the  link
  condition exactly once and failed the test.
\end{itemize}
\end{pro}
\begin{proof}
  Our proof is by induction on  the number $i$ of times that the while
  loop   condition,  $\lue   \neq   \emptyset$,  in   line  22-30   of
  Algorithm~\ref{alg:contractions}             is             reached.
  Proposition~\ref{prop:whiletermination}  ensures that the  loop ends
  after $m$  iterations, for some  non-negative integer, $m$.   So, we
  have $1 \le i \le m + 1$.

  \textbf{Base case  ($\bm{i = 1}$)}.  When the  while loop condition,
  $\lue \neq \emptyset$, is reached for the first time, conditions (1)
  to (3) are  all true \textit{by the hypothesis}  of our claim.  When
  we show that (1)-(3) also hold for the outer repeat-until loop, then
  this hypothesis will be proved to be true as well.

  \textbf{Hypothesis ($\bm{i =  k}$)}.  Assume that conditions (1)-(3)
  hold for  $i = k$,  where $k$ is  an arbitrary (but  fixed) integer,
  with $1 \le k \le m$;  that is, conditions (1)-(3) hold if condition
  $\lue \neq \emptyset$ is tested for the $k$-th time.

  \textbf{Inductive Step  ($\bm{i =  k+1}$)}.  When $i  = k +  1$, the
  while  loop condition,  $\lue \neq  \emptyset$, is  reached  for the
  $(k+1)$-th time.  Since  $k \ge 1$, the loop  body has been executed
  at least once.   So, consider the moment in  which the condition was
  reached for the  $k$-th time.  Since the loop  body is executed, the
  while loop  condition holds and thus  $\lue$ is not  empty.  At this
  point, the  induction hypothesis implies  that if any  of conditions
  (1)-(3)  is violated, then  the violation  occurs during  the $k$-th
  iteration of the loop.  We now show that this is not the case.

  The $k$-th iteration of the loop  starts with the removal of edge $[
  u , v ]$ from the front of  list $\lue$. Next, the value of $t( v )$
  is      set     to     \textit{ts}      (see     line      24     of
  Algorithm~\ref{alg:contractions}). Note that \textit{ts} is equal to
  or  greater than $0$,  as \textit{ts}  is set  to $0$  in line  3 of
  Algorithm~\ref{alg:contractions}  and is  never  decremented by  the
  algorithm. Furthermore,  we get $t(z) \ge  o(z) > -1$,  as $o(z)$ is
  equal to  the value held by  \textit{ts} when the for  loop in lines
  8-19 of Algorithm~\ref{alg:contractions} was executed.  Next, if the
  degree   of  $v$   is  equal   to   $3$,  then   the  procedure   in
  Algorithm~\ref{alg:processdegree3}   is  executed;   otherwise,  the
  procedure in Algorithm~\ref{alg:processdegreegt3}  is. In the former
  case, list $\lte$ is not  modified.  In the latter case, list $\lte$
  may be modified by the insertion of edge $[ u , v ]$, which has just
  been     removed     from     $\lue$     (see     line     16     of
  Algorithm~\ref{alg:processdegreegt3}). However,  in both cases, list
  $\lue$ is modified  by the insertion of  edges of the form $[  u , z
  ]$, where $z$  is a vertex in  $\textit{lk}( v , K )$,  with $z \not
  \in  \pd_{uv}$, and  $K$ is  the  triangulation immediately
  before the  contraction of  $[ u ,  v ]$. Since  $[ u  , v ]$  is no
  longer  in $\lue$, and  since the  edges inserted  in $\lue$  do not
  belong to $K$,  to $\T$ nor to any  triangulation obtained from $\T$
  before $K$, condition (1) is not violated. We also know that $t( v )
  =  \textit{ts}$   and  $\textit{ts}  \ge   0$  before  line   16  of
  Algorithm~\ref{alg:processdegreegt3} is executed.  Moreover, line 24
  of Algorithm~\ref{alg:contractions} is the  only one that can modify
  the $t$  attribute of  a vertex during  the entire iteration  of the
  while  loop in  lines  22-30.   So, condition  (2)  is not  violated
  either.         Finally,       before        line        16       of
  Algorithm~\ref{alg:processdegreegt3}  is  executed,  lines 2-12  are
  executed to test  edge $[ u , v ]$ against  the link condition.  So,
  condition (3)  cannot be violated.  As a  result, conditions (1)-(3)
  remain true during the $k$-th loop iteration, and thus they are true
  when the  loop condition is reached  for the $i$-th time,  with $i =
  k+1$. Thus, our claim holds for $i = 1, \ldots, m+1$.
\end{proof}

We   now  give   a  proof   for   Proposition~\ref{prop:listLTE}  from
Section~\ref{sec:counting}:

\begin{proof}
  Our  proof is  by induction  on  the number  $i$ of  times that  the
  repeat-until  loop condition,  $\lue  = \emptyset$,  in  line 34  of
  Algorithm~\ref{alg:contractions}             is             reached.
  Proposition~\ref{prop:repeatuntiltermination}   ensures   that   the
  repeat-until loop terminates.  So, we  can assume that the loop ends
  after $m$  iterations, for some non-negative integer,  $m$. Thus, we
  have that $1 \le i \le m + 1$.

  \textbf{Base  case  ($\bm{i  =  1}$)}. When  the  repeat-until  loop
  condition, $\lue  = \emptyset$, is  reached for the first  time, the
  loop body  has been executed exactly  once. So, let  us consider the
  moment  at  which line  21  of Algorithm~\ref{alg:contractions}  was
  executed for the first time.  At this moment, we know that $\lte$ is
  empty  (see  line   20  of  Algorithm~\ref{alg:contractions}).   So,
  conditions (1)-(3) holds immediately  before the while loop of lines
  22-30     is    executed     for    the     first     time.     From
  Proposition~\ref{prop:listSwhile},  they also  hold during  the loop
  execution  and  immediately after  it  ends,  which will  eventually
  happen  (see  Proposition~\ref{prop:whiletermination}).   Next,  the
  procedure in Algorithm~\ref{alg:processlistS} may be invoked in line
  32. During the execution of this procedure, one or more edges may be
  removed    from     $\lte$    (see    lines    4     and    13    of
  Algorithm~\ref{alg:processlistS}),  but  no   edge  is  inserted  in
  $\lte$.   In  addition, the  $t$  attribute  of  a vertex  is  never
  modified  by \textsc{ProcessEdgeList}$()$.   So, conditions  (2) and
  (3) must remain true. In turn,  no edge is removed from list $\lue$,
  but some edges  may be inserted into $\lue$. If  an edge is inserted
  into $\lue$, then edge $[ u , v ]$ was removed from $\lte$ before in
  line 13 of Algorithm~\ref{alg:processlistS}.   Moreover, edge $[ u ,
  v ]$  is contracted.  The  contraction of $[  u , v ]$  triggers the
  insertion into $\lue$  of all edges of  the form $[ u ,  z ]$, where
  $z$  is a  vertex  in  $\textit{lk}( v  ,  K )$,  with  $z \not  \in
  \pd_{uv}$.  But,  these edges do not belong  to the current
  triangulation, $K$,  to $\T$ nor to any  triangulation obtained from
  $\T$  before $K$.  So, none  of  them can  be in  $\lte$, and  hence
  condition (1)  remains true.   Thus, our claim  holds when  the loop
  condition is reached for the first time.

  \textbf{Hypothesis ($\bm{i =  k}$)}.  Assume that conditions (1)-(3)
  hold for  $i = k$,  where $k$ is  an arbitrary (but  fixed) integer,
  with $1 \le k \le m$.  That is, conditions (1)-(3) hold if condition
  $\lue = \emptyset$ is tested for the $k$-th time.

  \textbf{Inductive Step  ($\bm{i =  k+1}$)}.  When $i  = k +  1$, the
  repeat-until loop condition, $\lue  = \emptyset$, is reached for the
  $(k+1)$-th time. Since $k \ge 1$, the loop body has been executed at
  least twice.   So, consider  the moment in  which the  condition was
  reached  for   the  $k$-th  time.   By   the  induction  hypothesis,
  conditions  (1)-(3) are  true  immediately after  the  loop test  is
  executed, i.e.,  they are  true in the  beginning of  the $(k+1)$-th
  iteration.   So, from  Proposition~\ref{prop:listSwhile}, conditions
  (1)-(3) remain true immediately after  the while loop in lines 22-30
  ends,       which       will       eventually       occur       (see
  Proposition~\ref{prop:whiletermination}).  Using  the same arguments
  from the  base case, we can  show that conditions  (1)-(3) also hold
  immediately   after  lines   31-33   are  executed.    Consequently,
  conditions (1)-(3) hold when condition $\lue = \emptyset$ is reached
  in  line  34  of Algorithm~\ref{alg:contractions}  for  the
  $i$-th  time, with  $i=k+1$. Thus,  our claim  is true  for $i  = 1,
  \ldots, m+1$.
\end{proof}

The    following     three    results    are     needed    to    prove
Proposition~\ref{prop:listLTE2} from Section~\ref{sec:procedges}:

\begin{pro}\label{prop:cvalue}
  Let $u$ be the currently  processed vertex of the algorithm, and let
  $[ u  , v ]$ be any  edge of the current  triangulation, $K$, tested
  against    the   link    condition    by   the    lines   2-12    of
  Algorithm~\ref{alg:processdegreegt3}.  Then, the  value  of the  $c$
  attribute  of $v$ is  zero immediately  before the  test, and  it is
  equal to the number of critical cycles  that contains $[ u , v ]$ in
  $K$ immediately after the test.
\end{pro}
\begin{proof}
  Every edge,  $[ u , v ]$,  tested against the link  condition by the
  lines 2-12  of Algorithm~\ref{alg:processdegreegt3} is  an edge just
  removed from $\lue$  in line 23 of Algorithm~\ref{alg:contractions}.
  This means that  $[ u , v  ]$ belonged to $\lue$, and  hence the $c$
  value of  vertex $v$  is zero immediately  before the lines  2-12 of
  Algorithm~\ref{alg:processdegreegt3}  are executed. This  is because
  Algorithm~\ref{alg:contractions}  does not modify  the $c$  value of
  $v$ before  Algorithm~\ref{alg:processdegreegt3} is invoked  in line
  26.    During   the  link   condition   test   in   lines  2-12   of
  Algorithm~\ref{alg:processdegreegt3},    the    value   $c(v)$    is
  incremented by  one $m$ times, where  $m$ is the  number of vertices
  $z$ in  $\po_{uv}$ such  that $n(  z ) =  u$. Note  that if  $z \in
  \po_{uv}$ and $n( z  ) = u$ then $z$ is a  neighbor of both $u$ and
  $v$, which means  that $( u , v ,  z )$ is a cycle  of length $3$ in
  $K$. Such a  cycle, if any, cannot bound a $2$-face  in $K$, as edge
  $[ u ,  v ]$ is already incident on  two $2$-faces, $[ u ,  v , x ]$
  and $[  u ,  v , y  ]$, where $x$  and $y$  are the two  vertices in
  $\textit{lk}([ u ,  v ], K )$.  Thus,  cycle $( u , v, z  )$ must be
  critical in $K$. Conversely, if $z \in \pd_{uv}$ or $n( z )
  \neq  u$, then either  $z \in  \{ u  , x  , y  \}$ or  $z$ is  not a
  neighbor of $u$. If $z = u$ or $n( z  ) \neq u$ then $( u , v , z )$
  is not a cycle of length $3$. Otherwise,  if $z = x$ or $z = y$ then
  $( u , v , z )$ is a  $2$-face of $K$. In either case, triple $( u ,
  v , z  )$ does not define a  critical cycle.  So, since $c(  v )$ is
  equal   to   zero   before   the   execution  of   lines   2-12   of
  Algorithm~\ref{alg:processdegreegt3}, and  since $c(v)$ is  equal to
  $c(   v   )  +   m$   after  the   execution   of   lines  2-12   of
  Algorithm~\ref{alg:processdegreegt3},  we have  that  $c(v)$ is  the
  number of critical cycles that contains $[ u , v ]$ in $K$ after the
  link condition test ends. Note that  we rely on the premise that $n(
  z )$ is equal to $u$ if and only if $z$ is a neighbor of $u$ in $K$,
  which is in fact  true.  To see why, recall that $n(  v )$ is set to
  $u$, for  every vertex  $v$ in  $\textit{lk}( u ,  K )$,  during the
  initialization      of      $\lue$       in      line      9      of
  Algorithm~\ref{alg:contractions}.  In addition, during the execution
  of     the     repeat-until    loop     in     lines    21-34     of
  Algorithm~\ref{alg:contractions}), the $n$  attribute of a vertex is
  only   modified   in    line   6   of   \textsc{Contract}$()$   (see
  Algorithm~\ref{alg:contract}). However,  this is done  precisely for
  the vertices $z$ of $\textit{lk}( v  , K )$ that become neighbors of
  $u$ after the contraction of $[  u , v ]$.  Conversely, a vertex $z$
  in the current triangulation, $K$, can only become a neighbor of $u$
  if the contraction of  an edge, $[ u , v ]$,  takes place and $z$ is
  in $\po_{uv}$. So, whenever $[ u  , v ]$ is tested against the link
  condition in Algorithm~\ref{alg:processdegreegt3},  we have that $n(
  z )$ is equal to $u$ if and only if $z$ is a neighbor of $u$ in $K$.
\end{proof}

An immediate  consequence of Proposition~\ref{prop:cvalue}  is that an
edge $[ u , v ]$ is inserted  into $\lte$ if and only if it is part of
a     critical     cycle     in     $K$,     as     line     16     of
Algorithm~\ref{alg:processdegreegt3} is executed if  and only if $c( v
) $  is not  equal to zero  immediately after  $[ u ,  v ]$  is tested
against     the     link     condition     in    lines     2-12     of
Algorithm~\ref{alg:processdegreegt3}. Furthermore,  at this point, the
value of $c(v)$ is equal to  the number of critical cycles in $K$ edge
$[ u , v ]$ belongs to.

\begin{cor}\label{cor:cvalue}
  Let $u$ be the vertex  currently processed by the algorithm, and let
  $[  u , v  ]$ be  an edge  removed from  list $\lue$  by line  23 of
  Algorithm~\ref{alg:contractions}. Then, edge $[ u , v ]$ is inserted
  into list $\lte$  in line 16 of Algorithm~\ref{alg:processdegreegt3}
  if and  only if  $[ u  , v  ]$ is part  of a  critical cycle  in the
  current triangulation,  $K$. Furthermore, the  value of $c( v  )$ is
  equal to the number of critical cycles in $K$ to which edge $[ u , v
  ]$ belongs.
\end{cor}

The result  stated by Corollary~\ref{cor:cvalue} is only  valid at the
moment that edge  $[ u , v ]$ is inserted  into $\lte$.  Indeed, while
$[ u , v ]$ is in $\lte$,  the contraction of an edge may give rise to
a critical  cycle containing $[  u ,  v ]$ or  it can make  a critical
cycle containing $[ u , v ]$ non-critical.

\begin{pro}\label{prop:listSwhile2}
  Let $u$ be the vertex currently processed by the algorithm. Then, if
  the following  condition holds immediately before the  while loop in
  lines 22-30 of Algorithm~\ref{alg:contractions} is executed, it will
  remain true every time the loop condition, $\lue \neq \emptyset$, is
  reached in line  22: {\em no edge $[  u , v ]$ in  $\lte$, such that
    $c( v ) > 0$, can precede an  edge $[ u , w ]$ in $\lte$ such that
    $c( w ) = 0$}.
\end{pro}
\begin{proof}
  Our proof is by induction on  the number $i$ of times that the while
  loop   condition,   $\lue   \neq   \emptyset$,   in   line   22   of
  Algorithm~\ref{alg:contractions}       is       reached.        From
  Proposition~\ref{prop:whiletermination},  we   know  that  the  loop
  eventually ends after finitely many iterations. Let $m$ be the total
  number   of   loop  iterations,   where   $m$   is  a   non-negative
  integer. Then, we have that $1 \le i \le m + 1$.

  \textbf{Base case  ($\bm{i = 1}$)}.  When the  while loop condition,
  $\lue  \neq \emptyset$,  is reached  for the  first time,  our claim
  holds     \textit{by    hypothesis}.      In     the    proof     of
  Proposition~\ref{prop:listLTE2}, we  show that the  claim also holds
  for     the    repeat-until     loop    of     lines     21-34    of
  Algorithm~\ref{alg:contractions}.  So, our hypothesis is valid.

  \textbf{Hypothesis ($\bm{i = k}$)}.  Assume that our claim holds for
  $i = k$, where $k$ is  an arbitrary (but fixed) integer, with $1 \le
  k \le m$; that  is, assume that {\em no edge $[ u  , v ]$ in $\lte$,
    such that $c( v ) > 0$, can  precede an edge $[ u , w ]$ in $\lte$
    such that $c( w ) = 0$}, when the while loop condition, $\lue \neq
  \emptyset$, is tested for the $k$-th time.

  \textbf{Inductive  Step ($\bm{i =  k+1}$)}. When  $i =  k +  1$, the
  while  loop condition,  $\lue \neq  \emptyset$, is  reached  for the
  $(k+1)$-th time.  Since  $k \ge 1$, the loop  body has been executed
  at least once.   So, consider the moment in  which the condition was
  reached for the  $k$-th time.  Since the loop  body is executed, the
  while loop  condition holds and thus  $\lue$ is not  empty.  At this
  point, the induction  hypothesis implies that {\em no edge  $[ u , v
    ]$ in $\lte$, such that $c( v ) > 0$, can precede an edge $[ u , w
    ]$ in  $\lte$ such  that $c(  w ) =  0$}. We  must show  that this
  property  is  not  violated  during  the  execution  of  the  $k$-th
  iteration of the while loop.   To that end, we consider all possible
  situations  in which  the structure  of $\lte$  is modified,  or the
  value of the $c$ attribute of a vertex $w$, such that $[ u , w ]$ is
  an edge in $\lte$, is modified.

  If an  edge is inserted into  $\lte$ during the  $k$-th iteration of
  the while loop, then the insertion  must have occurred in line 16 of
  Algorithm~\ref{alg:processdegreegt3}. At  this moment, we  know that
  $c( w )  > 0$, as the condition of the  if-then construction in line
  13 is  {\em false}. Moreover,  Corollary~\ref{cor:cvalue} ensures us
  that $c( w )$ is equal to  the number of critical cycles edge $[ u ,
  w ]$ is  part of in the current triangulation. So,  the value of $c(
  w)$ cannot be negative. Since $[ u , w ]$ is inserted at the rear of
  $\lte$, the property  in our claim cannot be violated  by line 16 of
  Algorithm~\ref{alg:processdegreegt3}.  If an  edge $[  u ,  w  ]$ in
  $\lte$ is moved to the front  of $\lte$, then such a change can only
  occur     in      lines     9,     12,     17,      or     22     of
  Algorithm~\ref{alg:processdegree3}.  However, in  each case, edge $[
  u  , w ]$  is moved  to the  front of  $\lte$ because  $c( w  )$ was
  decremented before  and became $0$.   So, the property in  our claim
  cannot  be  violated  by  lines  9,  12, 17,  and  22  of  Algorithm
  \ref{alg:processdegree3}.  If an edge $[ u , w ]$ in $\lte$ is moved
  to the rear of  $\lte$, then such a change can only  occur in line 8
  of    Algorithm~\ref{alg:processdegreegt3}    or    line    38    of
  Algorithm~\ref{alg:contract}. But, in either case,  edge $[ u , w ]$
  is moved  to the  rear of  $\lte$ because $c(  w )$  was incremented
  before  and became  $1$. So,  the property  in our  claim  cannot be
  violated by  line 8 of  Algorithm~\ref{alg:processdegreegt3} or line
  38 of Algorithm~\ref{alg:contract} either.

  If the  value of $c(  w )$ is  incremented while $[ u  , w ]$  is in
  $\lte$,  then  this  must  occur  either  in  line  6  of  Algorithm
  \ref{alg:processdegreegt3}      or      in      line      36      of
  Algorithm~\ref{alg:contract}. But, in either  case, the value of $c(
  w )$ is  compared to $1$ in the  following line, and $[ u ,  w ]$ is
  moved to the rear  of $\lte$ whenever $c( w )$ is  equal to $1$. So,
  the property in our claim cannot  be violated after the value of $c(
  w  )$  is  incremented. Likewise,  if  the  value  of  $c( w  )$  is
  decremented while $[ u , w ]$  is in $\lte$, then this must occur in
  line 6, 7, 15,  or 20 of Algorithm~\ref{alg:processdegree3}. But, in
  each case,  the value  of $c(  w )$ is  compared to  $0$ immediately
  after $c( w )$ is decremented, and $[ u , w ]$ is moved to the front
  of $\lte$ whenever $c( w )$ is equal to $0$. So, the property in our
  claim  cannot  be   violated  after  the  value  of   $c(  w  )$  is
  decremented. Thus,  we can conclude  that the property in  our claim
  cannot be violated  during the execution of the  $k$-th iteration of
  the while loop. As a result,  the property holds when the while loop
  condition is  reached for the $i$-th time,  with $i = k  + 1$. Thus,
  our claim holds for $i = 1, \ldots, m+1$.
\end{proof}

We      now       prove      Proposition~\ref{prop:listLTE2}      from
Section~\ref{sec:procedges}:

\begin{proof}
  Our  proof is  by induction  on  the number  $i$ of  times that  the
  repeat-until  loop condition,  $\lue  = \emptyset$,  in  line 34  of
  Algorithm~\ref{alg:contractions}   is  reached  and   tested.   From
  Proposition~\ref{prop:repeatuntiltermination}, we know that the loop
  eventually ends after finitely many iterations. Let $m$ be the total
  number   of   loop  iterations,   where   $m$   is  a   non-negative
  integer. Then, we get $1 \le i \le m + 1$.

  \textbf{Base  case  ($\bm{i  =  1}$)}. When  the  repeat-until  loop
  condition, $\lue  = \emptyset$, is  reached for the first  time, the
  loop body has  been executed exactly once.  So,  let us consider the
  moment  in  which line  21  of Algorithm~\ref{alg:contractions}  was
  executed for the first time.   At this moment, list $\lte$ is empty.
  So, the property in our claim holds when the while loop condition in
  line 22 of Algorithm~\ref{alg:contractions} is reached for the first
  time.   From Proposition~\ref{prop:listSwhile2},  we  know that  the
  property remains  true immediately after the execution  of the while
  loop. So, if the property is violated, then it must be so during the
  execution         of         \textsc{ProcessEdgeList}$()$         in
  Algorithm~\ref{alg:processlistS}.   But,  this  procedure  can  only
  cause a change of structure in $\lte$ or in the value of $c( w )$ of
  a   vertex   $w$,    with   $[   u   ,   w    ]$   in   $\lte$,   if
  Algorithm~\ref{alg:processdegree3}  or  Algorithm~\ref{alg:contract}
  are       executed       in        lines       5       and       14,
  respectively. Algorithm~\ref{alg:processdegree3} can move an edge $[
  u , w ]$ in $\lte$ to  the front of $\lte$ or decrement the value of
  $c( w )$. But,  edge $[ u , w ]$ is moved to  the front of $\lte$ if
  and only  if $c(  w )$ is  equal to  $0$ after being  decremented by
  $1$. In turn, Algorithm~\ref{alg:contract} can move an edge $[ u , w
  ]$ in $\lte$ to  the rear of $\lte$ or increment the  value of $c( w
  )$. But, edge $[ u , w ]$ is moved to the rear of $\lte$ if and only
  if $c( w )$ is equal to  $1$ after being incremented by $1$. So, the
  property in our claim cannot  be violated during the first iteration
  of  the  repeat-until loop,  and  hence  our  claim holds  when  the
  condition of the loop is reached for the first time.

  \textbf{Hypothesis ($\bm{i = k}$)}.  Assume that our claim holds for
  $i = k$, where $k$ is  an arbitrary (but fixed) integer, with $1 \le
  k \le m$; that  is, assume that {\em no edge $[ u  , v ]$ in $\lte$,
    such that $c( v ) > 0$, can  precede an edge $[ u , w ]$ in $\lte$
    such that  $c( w  ) = 0$},  when the repeat-until  loop condition,
  $\lue = \emptyset$, is tested for the $k$-th time.

  \textbf{Inductive  Step ($\bm{i =  k+1}$)}. When  $i =  k +  1$, the
  repeat-until loop condition, $\lue  = \emptyset$, is reached for the
  $(k+1)$-th time. Since $k \ge 1$, the loop body has been executed at
  least twice.   So, consider  the moment in  which the  condition was
  reached  for the  $k$-th  time.  By  the  induction hypothesis,  the
  property of our claim holds  immediately after the loop condition is
  tested, which means that the property also holds in the beginning of
  the   $(k+1)$-th   iteration    of   the   loop.    In   particular,
  Proposition~\ref{prop:listSwhile2}  holds.   So,  using exactly  the
  same arguments from the base case,  we can show that the property of
  our claim remains true immediately  after the loop body is executed.
  Consequently, the  property holds when condition  $\lue = \emptyset$
  is reached  in line 34  of Algorithm \ref{alg:contractions}  for the
  $i$-th time, with $i  = k+1$. Thus, our claim is true  for each $i =
  1, \ldots, m+1$.
\end{proof}

Finally,   we    can   prove   Proposition~\ref{pro:ccequality}   from
Section~\ref{sec:procedges}:

\begin{proof}
  Assume that $u$ is the  vertex currently processed by the algorithm,
  and let $[  u , w]$ be  any edge in $\lte$ during  the processing of
  $u$.  From  Proposition~\ref{prop:listLTE2}, we  know that $c(  w )$
  was greater  than $0$ at the  moment that $[  u , w ]$  was inserted
  into $\lte$  by line 16  of Algorithm~\ref{alg:processdegreegt3}. We
  must  show  that  $c(w)$   equals  the  number  of  critical  cycles
  containing   $[    u   ,    w   ]$   every    time   line    32   of
  Algorithm~\ref{alg:contractions} is reached and  $[ u , w]$ is still
  in $\lte$.  To  that end, we must prove  that $c(w)$ is consistently
  modified  by our algorithm;  that is,  we must  show that  $c(w)$ is
  decremented by $1$  if only if a critical cycle containing  $[ u , w
  ]$ is eliminated by the algorithm, and that $c(w)$ is incremented by
  $1$  if and  only if  a critical  cycle containing  $[ u  , w  ]$ is
  created by the  algorithm.  

  Assume that a critical cycle containing $[ u , w ]$ is eliminated by
  the algorithm.   This is only  possible if  an edge $[  u , v  ]$ is
  contracted, $v$  is a degree-$3$ vertex,  and $w$ is one  of the two
  vertices of $\textit{lk}( [ u , v  ] , K )$.  Furthermore, edge $[ u
  , v]$ is  contracted by invoking \textsc{Contract}$()$ in  line 3 of
  Algorithm~\ref{alg:processdegree3}.   However,  after  line  $3$  is
  executed, the value  of $c(w)$ is decremented by $1$  in lines 6, 7,
  15, or 20 of Algorithm~\ref{alg:processdegree3} (with $w$ labeled by
  either $x$ or $y$). Now,  assume that a critical cycle containing $[
  u , w  ]$ is created by  the algorithm.  This is only  possible if a
  vertex $z$, which is a neighbor of $w$, became a neighbor of $u$ due
  to the contraction of an edge $[ u , v ]$. Furthermore, $( u , z , w
  )$ is a  critical cycle in the resulting  triangulation.  If $z$ has
  already  been   processed  by  the   algorithm,  then  line   36  of
  \textsc{Contract}$()$ is executed to  increment $c(w)$ by $1$, which
  accounts for $(  u , z , w  )$.  If $z$ has not  been processed yet,
  then edge $[ u  , z ]$ is inserted into $\lue$  by either line 28 or
  line  30 of  \textsc{Contract}$()$.  At this  moment,  the value  of
  $c(w)$  is not  incremented by  $1$ to  account  for $(  u ,  z ,  w
  )$. However,  as soon as $[  u , z ]$  is removed from  $\lue$ to be
  tested   against    the   link    condition   (see   line    23   of
  Algorithm~\ref{alg:contractions}),
  Algorithm~\ref{alg:processdegreegt3} is  executed (as the  degree of
  $z$ must be  greater than $3$; else $( u  , z , w )$  would not be a
  critical cycle) and $c(w)$ is incremented by $1$ to account for $( u
  , z , w )$.

  Conversely, assume  that $c(w)$ has  been modified by  the algorithm
  while $[ u , w ]$ is  in $\lte$.  if so, then then either (1) $c(w)$
  was   decremented  by   $1$   in  lines   6,   7,  15,   or  20   of
  Algorithm~\ref{alg:processdegree3} (with  $w$ labeled by  either $x$
  or  $y$),  or  (2) $c(w)$  was  incremented  by  $1$  in line  6  of
  Algorithm~\ref{alg:processdegreegt3} (with  $w$ labeled as  $z$), or
  (3)   $c(w)$    was   incremented   by    $1$   in   line    36   of
  Algorithm~\ref{alg:contract}. If  (1) holds, then vertex  $w$ is one
  of the two vertices in $\textit{lk}( [ u  , v ] , K )$, where $[ u ,
  v ]$  is an edge contracted  by the algorithm  in triangulation $K$.
  Since $v$ is a  degree-$3$ vertex, cycle $( u , x  , y )$, which was
  critical in $K$ before the contraction, becomes non-critical in $K -
  uv$. Since $w = x$ or $w =  y$, the number of critical cycles $[ u ,
  w ]$ belongs  to is decremented by $1$.  If (2)  holds, then an edge
  $[ u , v ]$ just removed from $\lue$ is found to be non-contractible
  by Algorithm~\ref{alg:processdegreegt3}  and $(  u , v  , w )$  is a
  critical cycle  in the current  triangulation, $K$. Since line  6 of
  Algorithm~\ref{alg:processdegreegt3}  is executed  for $z  =  w$, we
  have that $t( w  ) < o( v )$. This means  that $v$ became a neighbor
  of  $u$ \textit{after}  $[ u  ,  w ]$  was tested  against the  link
  condition  and  inserted  into  $\lte$.   So, $c(w)$  has  not  been
  incremented before to take into account  critical cycle $( u , v , w
  )$, and line 6 modifies the value  of $c(w)$ to account for $( u , v
  , w )$.   If (3) holds, then the  contraction of an edge $[  u , v]$
  caused a previously  processed vertex, $z$, to become  a neighbor of
  $u$. Furthermore, vertex $w$ is a  neighbor of both $u$ and $z$. So,
  cycle $(  u , z  , w )$  is critical in the  triangulation resulting
  from the  contraction.  Since  this cycle did  not exist  before the
  contraction,  $c(w)$ has  not been  incremented before  to  take the
  cycle into account.  So, line 32 increments $c(w)$ by $1$ to account
  for $( u , z , w )$.
\end{proof}

\begin{cor}\label{cor:equality2}
  Let $u$  be a vertex of  $\T$ processed by the  algorithm.  Then, by
  the time $u$ is processed,  every edge in $\lte$ is non-contractible
  in the current triangulation.
\end{cor}

\begin{lem}\label{lem:correct1}
  The algorithm described  in Section~\ref{sec:algo} always terminates
  and produces  a triangulation, $\T^{\prime}$, whose  set of vertices
  is the set of processed vertices of $\T$.
\end{lem}
\begin{proof}
  Recall that every vertex, $u$, of $\T$ is placed in a queue, $Q$, by
  Algorithm      \ref{alg:continit}.      Later,      in     Algorithm
  \ref{alg:contractions}, every vertex $u$ in $Q$ is removed from $Q$,
  one at a time, and no vertex is further inserted in $Q$. After being
  removed from  $Q$, a vertex $u$  is selected to be  processed by the
  algorithm if and  only if $p( u )$  is {\em false}.  If $p(  u )$ is
  {\em true}, then vertex $u$ is discarded and a new vertex is removed
  from $Q$,  if any. If $p(  u )$ is  {\em false}, then vertex  $u$ is
  processed          by          the         algorithm.           From
  Proposition~\ref{prop:repeatuntiltermination}  ,  we  know that  the
  processing  of  $u$  ends  after  finitely many  iterations  of  the
  repeat-until        loop       in        lines        20-34       of
  Algorithm~\ref{alg:contractions}.  Since  there are a  finite number
  of vertices in $\T$, and since no vertex is inserted into $Q$ during
  the execution of Algorithm~\ref{alg:contractions}, $Q$ must be empty
  after every vertex of $\T$ is  removed from it.  At this moment, the
  algorithm ends.  We now show  that a vertex belongs to $\T^{\prime}$
  if  and only  if it  was processed  by the  algorithm. Indeed,  if a
  vertex $u$  was processed by the  algorithm, then $p( u  )$ was {\em
    false}     when    $u$    was     removed    from     $Q$.     So,
  Proposition~\ref{prop:processed} asserts that either (1) $u$ has not
  been processed by  the algorithm before, and $u$  belongs to $K$, or
  (2) $u$ has been processed by the algorithm before, and $u$ does not
  belong to $K$.   We claim that assertion (2) is  false.  In fact, if
  $u$ had been processed before, then  $p( u )$ would have been set to
  {\em true} immediately after the repeat-until loop in lines 20-34 of
  Algorithm~\ref{alg:contractions}  is executed  to  process $u$  (see
  line 35  of Algorithm  \ref{alg:contractions}).  But, if  $p (  u )$
  becomes {\em true}, then Proposition~\ref{prop:listLUE} ensures that
  an edge of  the form $[ z ,  u ]$ can never be  inserted into $\lue$
  during  the processing of  a vertex  $z$, which  was selected  to be
  processed after $u$.  Thus, edge $[  z , u ]$ cannot be removed from
  the   current   triangulation.   So,   assertion   (1)  holds,   and
  consequently  vertex $u$  is  processed by  the  algorithm, and  $u$
  belongs  to  $K$.   But, since  $p(  u  )$  becomes true  after  the
  processing  of $u$, vertex  $u$ can  no longer  be removed  from the
  current triangulation nor from  any triangulation derived from it by
  further   edge   contractions.     So,   vertex   $u$   belongs   to
  $\T^{\prime}$. Conversely, if a vertex $u$ belongs to $\T^{\prime}$,
  then this  vertex was never removed from  any triangulation obtained
  from  $\T$ by  edge contractions.   So, vertex  $u$ belonged  to the
  current triangulation when  it was removed from $Q$.   Thus, at that
  point, $p( u )$ was {\em false} and hence $u$ was processed.
\end{proof}

\begin{lem}\label{lem:correct2}
  Triangulation $\T^{\prime}$ is of the same topological type of $\T$.
\end{lem}
\begin{proof}
  Since  our algorithm modifies  $\T$ by  contracting edges  only, and
  since  every contraction is  topology-preserving ---  as an  edge is
  contracted  only if  it  is a  contractible  edge ---  triangulation
  $\T^{\prime}$ must be of the same topological type of $\T$.
\end{proof}

\begin{thm}\label{thm:correct}
  The algorithm described in Section~\ref{sec:algo} is correct.
\end{thm}
\begin{proof}
  Since the  algorithm terminates (see  Lemma~\ref{lem:correct1}), and
  since $\T$ and $\T^{\prime}$  are triangulations of the same surface
  (see Lemma~\ref{lem:correct2}), it suffices  to show that every edge
  of  $\T^{\prime}$ is  non-contractible, i.e.,  that every  vertex of
  $\T^{\prime}$ is a  trapped one. To prove this  assertion, let $u_1,
  u_2, \ldots, u_n$, with $n \in \Z$ and $n \ge 1$, be the sequence of
  vertices selected to be processed  by the algorithm (in this order).
  We will  use induction on $k$ to  show that $u_k$ is  trapped by the
  time    it     is    completely    processed.      If    so,    then
  Lemma~\ref{lem:trapped} ensures that  $u_k$ will remain trapped when
  the algorithm ends.

  \textbf{Base  case ($\bm{k =  1}$)}.  Consider  the moment  in which
  $u_1$ is  completely processed.  If the  current triangulation, $K$,
  is  (isomorphic to)  $\T_4$, then  we are  done. So,  let  us assume
  otherwise.   Since  $u_1$  is  the  only  vertex  processed  by  the
  algorithm so  far, every edge incident  with $u_1$ in $K$  is of the
  form $[ u_1 , z ]$,  where $z$ has not been processed before.  Since
  $z$  belongs to  $K$, Proposition~\ref{prop:processed}  implies that
  $p( z )$ is {\em false}. We claim that $[ u_1 , z ]$ belongs to list
  $\lte$. To  prove our claim,  we must  show that $[  u_1 , z  ]$ was
  inserted  into list $\lue$  during the  processing of  $u_1$, tested
  against  the link  condition, and  then inserted  into  $\lte$ after
  failing the test. Aiming at a  contradiction, assume that $[ u_1 , z
  ]$ was  never inserted into  $\lue$. So, edge  $[ u_1 , z  ]$ cannot
  belong      to     $\T$.       Otherwise,     lines      8-19     of
  Algorithm~\ref{alg:contractions} would  have inserted $[ u_1  , z ]$
  into $\lue$ immediately before vertex $u_1$ is processed. This means
  that $z$ became adjacent to $u_1$ due to the contraction of an edge,
  $[ u_1 , z_1 ]$, in  $\T$ or in some triangulation derived from $\T$
  by previous edge contractions.  Let $L$ be this triangulation. Since
  $[ u_1 ,  z_1 ]$ was contracted,  edge $[ u_1 , z_1  ]$ was inserted
  into  $\lue$  before  the   contraction;  else  it  would  never  be
  contracted by the algorithm. Now, since $z$ belongs to $\textit{lk}(
  z_1  , L  )$, \textsc{Collapse}$()$  inserts $[  u_1 ,  z ]$  in the
  temporary list, \textit{temp}, of procedure \textsc{Contract}$()$ in
  Algorithm~\ref{alg:contract}. Since $p(  z ) = \textit{false}$, edge
  $[  u_1  ,  z  ]$  is   inserted  into  $\lue$  in  lines  27-31  of
  Algorithm~\ref{alg:contract}.  So,  edge $[ u_1  , z ]$  is inserted
  into $\lue$ during the processing  of $u_1$.  As a result, this edge
  is eventually removed  from $\lue$.  Since $[ u_1 ,  z ]$ belongs to
  $\T^{\prime}$,  it was  not  contracted by  the  algorithm.  So,  we
  distinguish two situations: (1) the degree, $d_z$, of $z$ was $3$ or
  (2) the degree, $d_z$, of $z$ was  greater than $3$ when $[ u , z ]$
  was removed from $\lue$. If (1) holds then the current triangulation
  is (isomorphic to) $\T_4$, as $[ u_1 , z ]$ was not contracted. But,
  if this is the case, then  no more edge contractions take place, the
  algorithm  ends, and every  vertex of  the current  triangulation is
  trapped. If (2) holds then edge $[ u_1 , z ]$ was tested against the
  link         condition        in        lines         2-12        of
  Algorithm~\ref{alg:processdegreegt3},  failed   the  test,  and  was
  inserted     into     list      $\lte$     in     line     16     of
  Algorithm~\ref{alg:processdegreegt3}.   Since $[  u_1 ,  z ]$  is in
  $\T^{\prime}$,  edge $[ u_1 , z ]$  was not  contracted  after  being inserted  into
  $\lte$.  Finally, from Corollary~\ref{cor:equality2}, edge $[ u ,
  z ]$ cannot  be contractible in $K$. So, we  can conclude that $u_1$
  is trapped in $K$ after being processed.

  \textbf{Hypothesis   ($\bm{k  \le   h}$,   with  $\bm{1   \le  h   <
      n}$)}. Assume  that our claim  is true for  every $k \in \{  1 ,
  \ldots , h \}$, i.e., $u_k$  is trapped by the time it is processed,
  for every $k  \in \{ 1 , \ldots  , h \}$, where $h$  is an arbitrary
  (but, fixed) integer, with $1 \le h < n$.

  \textbf{Inductive Step  ($\bm{k = h+1}$)}.   Let $K$ be  the current
  triangulation by  the time  that $u_{h+1}$ is  processed. If  $K$ is
  (isomorphic to) $\T_4$, then we are  done, as every vertex of $K$ is
  trapped and $u_{h+1}$  is in $K$.  So, let  us assume otherwise.  In
  this case, we can partition the set of edges incident with $u_{h+1}$
  in $K$ in  two sets: (1) the set,  $A$, of all edges of  the form $[
  u_{h+1} , q  ]$ in $K$ such  that $q$ has not been  processed by the
  algorithm before,  and the  set, $B$,  of all edges  of the  form $[
  u_{h+1}  , z  ]$ in  $K$ such  that $z$  has been  processed  by the
  algorithm before. Using the same argument from the base case, we can
  prove that  every edge $[ u ,  q ]$ in $A$  is non-contractible.  By
  the induction hypothesis, for every vertex $z$ in $K$ such that $[ u
  , z ]$ is  in $B$, vertex $z$ was trapped by  the time the algorithm
  finished processing  $z$.  From Lemma~\ref{lem:trapped},  vertex $z$
  remains trapped until  the end of the first stage, and  hence $e = [
  u_{h+1}  , z  ]$ is  also  non-contractible.  So,  since every  edge
  incident with $u_{h+1}$ in $K$ is non-contractible, vertex $u_{h+1}$
  is trapped by the time  it is completely processed by the algorithm.
  Thus, our claim is true for $k = 1, \ldots, n$.
\end{proof}

\end{appendix}

\end{document}